\documentclass{article}[twosided]

\newif\ifaistats
\newif\ifarxiv
\newif\ifarxivtwocolumns
\newif\ifneurips
\newif\ificml
\newif\ifaistatsnot
\newif\ifarxivnot
\newif\ifarxivtwocolumnsnot
\newif\ifneuripsnot
\newif\ificmlnot

\newif\ifextension
\newif\ifextensionnot
\extensionfalse
\extensionnottrue
\newif\ifhidden
\newif\ifhiddennot
\hiddenfalse
\hiddennottrue

\aistatsfalse
\arxivfalse
\arxivtwocolumnsnotfalse
\neuripsfalse
\icmlfalse
\aistatsnottrue
\arxivnottrue
\neuripsnottrue
\icmlnottrue

\ifaistats \aistatsnotfalse \fi
\ifarxiv \arxivnotfalse \fi
\ifarxivtwocolumnsnot \ifarxivtwocolumnsnotfalse \fi
\ifneurips \neuripsnotfalse \fi
\ificml \icmlnotfalse \fi
\ifextension \extensionnotfalse \fi
\ifhidden \hiddennotfalse \fi
\arxivtwocolumnstrue

\newcommand{\mytitle}{Newton Method Revisited:
Global Convergence Rates up to  $\mathcal {O}\left(k^{-3} \right)$ for Stepsize Schedules and Linesearch Procedures
}

\newcommand{\mbzuai}{\texttt{MBZUAI}}
\newcommand{\authors}{ 
    \makecell{Slavom\'ir Hanzely \\ \mbzuai\footremember{MBZUAI} {Mohammed bin Zayed University of Artificial Inteligence.}}
    \and \makecell{Farshed Abdukhakimov \\ \mbzuai\footrecall{MBZUAI}}
    \and \makecell{Martin Tak\'a\v{c} \\ \mbzuai\footrecall{MBZUAI}}
}

\newcommand{\ds}{D_s}
\newcommand{\dss}[1]{D_{#1}}
\newcommand{\br}[2]{\beta_f(#1,#2)}
\newcommand{\brmax}{V_f}

\newcommand{\alow}{\underline{\alpha}}
\newcommand{\aup}{\overline{\alpha}}
\newcommand{\ndir}{n_x}
\newcommand{\ndirk}{n_{x^k}}
\newcommand{\cupper}{\overline c_x}
\newcommand{\csup}{\overline c}

\ifaistats \aistatsnotfalse \fi
\ifarxiv \arxivnotfalse \fi
\ifneurips \neuripsnotfalse \fi
\ificml \icmlnotfalse \fi
\ifextension \extensionnotfalse \fi
\ifhidden \hiddennotfalse \fi

\usepackage[utf8]{inputenc} 
\usepackage[T1]{fontenc}    
\usepackage{hyperref}       
\usepackage{url}            
\usepackage{booktabs}       
\usepackage{amsfonts}       
\usepackage{nicefrac}       
\usepackage{microtype}      
\usepackage{xcolor}         
\usepackage{wrapfig}
\ifaistatsnot
    \usepackage{natbib}
\fi
\usepackage{xargs}                      
\usepackage{comment}

\usepackage{hyperref}
\usepackage{url, times}

\usepackage{amssymb,amsmath,amscd,amsfonts,amstext,amsthm,bbm, enumerate, dsfont, mathtools}
\usepackage{array}
\usepackage{multicol}
\usepackage{graphicx}
\usepackage{varwidth}
\usepackage{hyperref}
\usepackage{import}
\usepackage{caption,subcaption}
\ificmlnot
    \usepackage{algorithm}
    \usepackage{algpseudocode}
\fi
\usepackage[none]{hyphenat}
\usepackage{footnote}
\usepackage{cleveref}
\usepackage{makecell}
\usepackage{threeparttable}
\usepackage{booktabs}
\usepackage{tcolorbox}
\usepackage{caption}
\usepackage{subcaption}

\usepackage{scrextend}

\usepackage{tablefootnote}
\usepackage{changepage}
\usepackage{stackengine}
\usepackage{pifont}
\usepackage{mdframed}
\usepackage{etoolbox}
\AtBeginEnvironment{quote}{\par}

\graphicspath{ {images/} }

\newtheorem{theorem}{Theorem}
\newtheorem{proposition}{Proposition}
\newtheorem{lemma}{Lemma}
\newtheorem{corollary}{Corollary}
\newtheorem{assumption}{Assumption}
\newtheorem*{remark}{Remark}
\newtheorem{definition}{Definition}

\DeclareMathOperator*{\argmin}{argmin}

\newcommand{\EE}{\mathbb{E}}
\newcommand{\R}{\mathbb{R}}

\newcommand{\rdtor}{\R^d \to \R}

\newcommand{\eqdef}{\stackrel{\text{def}}{=}}
\newcommand{\cO}{{\cal O}}
\newcommand{\cOb}[1]{\cO \lr #1 \rr}

\mdfdefinestyle{definition}{
	backgroundcolor=lightgray!20,
    hidealllines=true,
}
\mdfdefinestyle{theorem}{
    linecolor=white,
	backgroundcolor=blue!3,
}
\mdfdefinestyle{proof}{
    linecolor=orange,
}
\mdfdefinestyle{proposition}{
    backgroundcolor=orange!10,
    hidealllines=true,
}
\mdfdefinestyle{corollary}{
    linecolor=blue!50,
}
\mdfdefinestyle{lemma}{
    linecolor=blue!50,
}

\surroundwithmdframed[style=definition]{definition}
\surroundwithmdframed[style=definition]{defalign}
\surroundwithmdframed[style=theorem]{theorem}

\newcolumntype{?}{!{\vrule width 1pt}}
\usepackage{colortbl}

\definecolor{mydarkgreen}{RGB}{39,130,67}

\definecolor{mydarkred}{RGB}{192,47,25}

\newcommand{\fopt}{f_*}
\newcommand{\xopt}{x^*}

\newcommand{\Lstandard}{{L_{\text{sc}}}}

\newcommand{\Lq}{M_q}
\newcommand{\Lsup}{L_{p, \nu}}
\newcommand{\Lsupnu}[1]{L_{p, #1}}
\newcommand{\Lsupp}[1]{L_{#1, \nu}}
\newcommand{\Lsuppnu}[2]{L_{#1, #2}}

\newcommand{\level}{\mathcal Q(x^0)}

\newcommand{\okd}{$\mathcal O \left( k^{-2} \right)$}

\newcommand{\algstyle}{\sf}
\newcommand{\aicn}{{\algstyle AICN}}

\newcommand{\gm}{{\algstyle GM}}
\newcommand{\grn}{{\algstyle GRN}}
\newcommand{\libsvm}{{\algstyle LIBSVM}}

\newcommand{\grls}{{\algstyle GRLS}}
\newcommand{\greedy}{{\algstyle GN}}

\newcommand{\coluniversal}{blue}
\newcommand{\expo}{\col \coluniversal {\beta}}
\newcommand{\expok}{\col \coluniversal {\expo_k}}
\newcommand{\regc}{\col \coluniversal{\sigma}}
\newcommand{\regck}{\col \coluniversal{\regc_k}}
\newcommand{\lun}{\col \coluniversal{\lambda}}
\newcommand{\lunk}{\col \coluniversal{\lun_k}}
\newcommand{\alun}{\col \coluniversal{\theta}}
\newcommand{\alunk}{\col \coluniversal{\theta_k}}
\newcommand{\alunkest}[1]{\col \coluniversal{\alun_{k,#1} }}
\newcommand{\cbound}{\col {mydarkred} \gamma}

\newcommand{\pr}{\col {red}\alpha}
\newcommand{\prk}{ \pr _ {\col {red} k}}
\newcommand{\gb}{\col {mydarkgreen}{\beta}}
\newcommand{\univ}[1]{ \mathcal H\lr #1 \rr }

\newcommand{\tn}[1]{\tnote{\color{red}(#1)}}

\newcommand{\pof}[1]{Proof of #1.}

\newcommand{\td}{\nabla^3 f}
\newcommand{\h}{\nabla^2 f}
\newcommand{\g}{\nabla f}

\newcommand{\mI}{\mathbf I}

\newcommand{\normM}[2]{{\left \| #1 \right\|}_{#2}}
\newcommand{\normsM}[2]{{\left \| #1 \right\|}_{#2}^2}
\newcommand{\normMd}[2]{{\left \| #1 \right\|}_{#2}^*}
\newcommand{\normsMd}[2]{{\left \| #1 \right\|}_{#2}^{*2}}

\newcommand{\gbox}{\colorbox{lightgray!20}}
\newcommand{\gboxeq}[1]{\text{\colorbox{lightgray!20} {$#1$}}}

\usepackage{theoremref}

\usepackage[colorinlistoftodos,bordercolor=orange,backgroundcolor=orange!15,linecolor=orange,textsize=scriptsize]{todonotes}

\newcommand{\prodin}[2]{ \prod \limits_{#1=1}^{#2}}

\newcommand{\xdiff}{x^{k+1} -x^k}

\newcommand{\N}{\mathbb N}

\newcommand{\blue}[1]{{\color{blue} #1}}

\newcommand{\purple}[1]{{ #1}}
\newcommand{\orange}[1]{{ #1}}
\newcommand{\magenta}[1]{{ #1}}

\newcommand{\col}[2]{{\color{#1} #2}}

\newcommand{\cc}{{c_5}}

\usepackage{enumitem}

\newcommand{\lr}{\left(} 
\newcommand{\rr}{\right)}
\newcommand{\ls}{\left[} 
\newcommand{\rs}{\right]}
\newcommand{\lc}{\left\lbrace} 
\newcommand{\rc}{\right\rbrace}
\newcommand{\la}{\left\langle} 
\newcommand{\ra}{\right\rangle}

\newcommand{\taylor}[1]{\Phi_{#1}}

\newcommand{\modelun}[1]{T_{\regc, \expo}\left(#1 \right)}

\newcommand{\modelunlexpo}[3]{T_{#1, #2}\left(#3 \right)}
\newcommand{\un}{{\textbf{\algstyle UN}}}
\newcommand{\rn}{{\textbf{\algstyle RN}}}

\newcommand{\footremember}[2]{%
    \footnote{#2}
    \newcounter{#1}
    \setcounter{#1}{\value{footnote}}%
}
\newcommand{\footrecall}[1]{%
    \footnotemark[\value{#1}]%
}

\ifaistats
    \usepackage{\string~"./styles/aistats2025"}
    
    \usepackage[round]{natbib}

\fi

\ifarxiv
    \usepackage{\string~"./styles/arxiv"}
    \author{\authors}
    \date{}
    \title{\mytitle}
\fi

\ifarxivtwocolumns
    \usepackage{\string~"./styles/arxivtwocolumns"}
    \author{\authors}
    \date{}
    \title{\mytitle}
\fi

\ifneurips
    \usepackage{\string~"./styles/neurips_2024"}
    \usepackage[utf8]{inputenc} 
    \usepackage[T1]{fontenc}    
    \usepackage{hyperref}       
    \usepackage{url}            
    \usepackage{booktabs}       
    \usepackage{amsfonts}       
    \usepackage{nicefrac}       
    \usepackage{microtype}      
    \usepackage{xcolor}         
    \author{\meauthor}
    \date{}
    \title{\mytitle}
\fi

\ificml
    \usepackage{\string~"./styles/icml2023"}
    \author{\authors}
    \date{}
    \title{\mytitle}
    
    \ificml
        \newcommand{\State}{\STATE}
        \newcommand{\Comment}{\COMMENT}
        \newcommand{\For}{\FOR}
        \newcommand{\EndFor}{\ENDFOR}
        \renewcommand{\gbox}{}
        \renewcommand{\gboxeq}{}
    \fi
\fi

\begin{document}
\ifaistats
        
        \twocolumn[
        \aistatstitle{\mytitle}
        \aistatsauthor{ Author 1 \And Author 2 \And  Author 3 }
        \aistatsaddress{ Institution 1 \And  Institution 2 \And Institution 3 } ]
\fi
\ifaistatsnot
	   \maketitle
\fi
\begin{abstract}
This paper investigates the global convergence of stepsized Newton methods for convex functions with Hölder continuous Hessians or third derivatives. We propose several simple stepsize schedules with fast global convergence guarantees, up to $\mathcal O\lr k^{-3} \rr$. For cases with multiple plausible smoothness parameterizations or an unknown smoothness constant, we introduce a stepsize linesearch and a backtracking procedure with provable convergence as if the optimal smoothness parameters were known in advance.
Additionally, we present strong convergence guarantees for the practically popular Newton method with exact linesearch.
\end{abstract}

\section{Introduction}

Second-order methods are fundamental to scientific computing. With its rich history that can be traced back to works \citet{Newton}, \citet{Raphson}, \citep{Simpson}, they have remained widely used up to the present day \citep{ypma1995historical, conn2000trust}. 
The main advantage of second-order methods is their independence from the conditioning of the underlying problem, enabling an extremely fast local quadratic convergence rate, where precision doubles with each iteration. Additionally, they are inherently invariant to rescaling and coordinate transformations, which greatly simplifies parameter tuning.
In contrast, the convergence of first-order methods is highly dependent on the problem's conditioning, resulting in a slower linear local convergence rate and a greater sensitivity to parameter tuning.

Despite their extremely fast local convergence, second-order methods often lack global convergence guarantees. Even the classical Newton method,
\begin{equation}
    x^{k+1} = x^k - \ls \h(x^k) \rs^{-1}\g(x^k),
\end{equation}
can diverge when initialized far from the solution \citep{jarre2016simple,mascarenhas2007divergence}.
Global convergence guarantees are typically achieved through various combinations of stepsize schedules \citep{nesterov1994interior}, line-search procedures \citep{kantorovich1948functional, nocedal1999numerical}, trust-region methods \citep{conn2000trust}, and Levenberg-Marquardt regularization \citep{levenberg1944method, marquardt1963algorithm}.

The simplest globalization strategy is to employ stepsize schedules. These schedules can be based on implicit descent conditions, which often require an additional subroutine per iteration, such as exact linesearch \citep{cauchy1847methode, shea2024greedy}, Armijo linesearch \citep{armijo1966minimization}, Wolfe condition \citep{wolfe1969convergence}, Goldstein condition \citep{nocedal1999numerical}. However, those methods often lack global convergence guarantees achieved by simple stepsize schedules. Notably, \citet{nesterov1994interior} introduced a simple stepsize schedule with global rate $\cO \lr k^{-\frac 12} \rr$. \citet{hanzely2022damped} improved upon this result by discovering duality between Newton stepsizes and Lavenberg-Marquardt regularization and proposing a stepsize with global rate \okd{} matching regularized Newton methods \citep{nesterov2006cubic, mishchenko2023regularized,doikov2024gradient}.

Despite all recent advances, current guarantees still fall short of the optimal rate for functions with Hölder continuous Hessians, $\Omega \lr k^{-\frac 72} \rr$ \citep{gasnikov2019optimal, agarwal2018lower, arjevani2019oracle}. 
It remains an open question whether the rate $\cO \lr k^{-2} \rr$ achieved by \citet{hanzely2022damped} is optimal for the Newton method or if more efficient stepsize schedules are yet to be discovered.
In the context of first-order methods, several nontrivial stepsize schedules have been shown to improve convergence of Gradient Descent. \citet{young1953richardson} introduced a stepsize schedule based on Chebyshev polynomials achieving the optimal rate for quadratic functions. \citet{polyak1987introduction} proposed a stepsize schedule optimal for non-smooth convex functions, and \citet{altschuler2023acceleration}, \citet{grimmer2024accelerated} proposed stepsize schedules with guaranteed semi-accelerated rate for general convex, Lipschitz smooth functions. This motivates us to ask the question:
\begin{quote}
\textit{Is it possible to guarantee a global convergence rate better than \okd{} for a simple stepsize schedule of the Newton method?}
\end{quote}
The answer is positive. 
We demonstrate that the stepsized Newton method can be analyzed under the assumption of Hölder continuity of third derivatives, achieving convergence guarantees resembling third-order tensor methods, up to $\mathcal O(k^{-3})$\footnote{For functions with Hölder continuous third derivatives, the achievable lower bound is $\Omega \lr k^{-5} \rr$ \citep{gasnikov2019optimal}.}. 
Analyzing the Newton method as the third-order method is a novel and unexpected approach, as the Newton method has traditionally been regarded as the most classical second-order method.


\subsection{Benefits of basic methods}
While it is possible to achieve optimal rates using acceleration techniques with a more complex structure \citep{gasnikov2019optimal}, basic methods are often preferred in practice for several reasons.

Firstly, basic methods are simple and easy to understand. They are also inherently robust, typically involving fewer hyperparameters, which minimizes the need for complex and costly hyperparameter tuning. In contrast, accelerated methods often require multiple sequences of iterates and additional hyperparameters, significantly increasing the complexity of tuning.

Moreover, basic methods can be seamlessly integrated with various techniques to enhance practical performance, such as parameter searches, data sampling strategies, momentum estimation, and gradient clipping. Combining these techniques with accelerated methods, however, introduces significant challenges. In the context of first-order methods, acceleration with parameter searches provides limited improvement over basic Gradient Descent with stepsize linesearch.

For second-order methods, the basic stepsized Newton method is particularly popular due to its affine invariance (i.e., invariance to changes in basis and data scaling), making it an efficient and convenient optimization tool.

\subsection{Notation}
For convex function $f:\R^d \to \R,$ we consider the optimization objective
\begin{equation} \label{eq:objective}
    \min_{x \in \R^d} f(x),
\end{equation}
where $f$ is twice differentiable with nondegenerate Hessians and potentially ill-conditioned. 

Our paper uses a nontrivial amount of notation; hence, we highlight definitions in gray and theorems in blue for easier reference. Denote any minimizer of the function \gbox{$\xopt \in \argmin_{x\in \R^d} f(x)$} and the optimal value \gbox{$\fopt \eqdef f(\xopt)$.}
We define norms based on a symmetric positive definite matrix $\mathbf H \in \R^{d \times d}$. For all $x,g \in \R^d,$
\begin{equation*}
\gboxeq{\|x\|_ {\mathbf H} \eqdef  \la \mathbf Hx,x\ra^{1/2}, } \quad \gboxeq{\|g\|_{\mathbf H}^{\ast}\eqdef\la g,\mathbf H^{-1}g\ra^{1/2}.}
\end{equation*} 
As a special case $\mathbf H= \mI$, we get $l_2$ norm $\|x\|_\mI = \la x,x\ra^{1/2}$. We will utilize \emph{local Hessian norm} $\mathbf H = \nabla^2 f(x)$, with shorthand notation for $h,g \in \R^d$
\begin{equation*}
\gboxeq{\|h\|_x \eqdef \la \nabla^2 f(x) h,h\ra^{1/2}, \|g\|_{x}^{\ast} \eqdef \la g,\nabla^2 f(x)^{-1}g\ra^{1/2}.}
\end{equation*}

\subsection{Stepsizes as a form of regularization}

\citet{hanzely2022damped} demonstrated that a stepsize schedule for the Newton method is equivalent to cubical regularization of the Newton method \citep{nesterov2006cubic} if the regularization is measured in the local Hessian norms. As the regularized Newton methods leverage the Taylor polynomial, we denote the second-order Taylor approximation of $f(y)$ by information at point $x$ as 
\begin{equation*}
    \gboxeq{\taylor{x}(y)\eqdef f(x)+ \la \g(x), y-x \ra + \frac 12 \normsM {y-x} x.}
\end{equation*}
In particular, \citet{hanzely2022damped} showed that
    \begin{align*}
    x^{k+1} &= T(x^k), \, T(x) = \argmin_{y \in \R^d} \left \{ 
        \taylor{x}(y)
        +\frac \regc {3} \normM {y-x} x ^{3} \right\}
    \end{align*}
    is equivalent to a Newton method with stepsize \aicn{}\footnote{We present the stepsize in a simplified but equivalent form. \citet{hanzely2022damped} expressed its stepsize as $\prk = \frac {-1+\sqrt{1+2\regc \normMd {\g(x^k)} {x^k} }}{\regc \normMd {\g(x^k)} {x^k}}$.
    } 
    \begin{align}
    x^{k+1} &= x^k - \prk [\h(x^k)]^{-1} \g(x^k), \label{eq:newton_stepsize} \\ 
    & \qquad \text{for } \prk = \frac {2}{1+\sqrt{1+2\regc \normMd {\g(x^k)} {x^k}}}. \label{eq:aicn}
    \end{align}
Note that stepsize schedule \eqref{eq:aicn} preserves much larger stepsize when initialized far from the solution, $\normMd {\g(x^0)} {x^0} \gg 1$, compared to the stepsize of Damped Newton method \citep{nesterov1994interior}, which sets stepsize  for $L_{sc}$-self-concordant functions as
\begin{equation}
    \prk = \frac 1 {1+ L_{sc}\normMd {\g(x^k)} {x^k}}.
\end{equation}
Aiming to extend this dependence beyond $\Lsuppnu 21$-Hölder continuous functions (\Cref{def:holder}), in \Cref{sec:alg} we present algorithm \rn{} that under general $\Lsup$-Hölder continuity (Def \ref{def:holder}) and $q=p+\nu \in [2,4]$ supports stepsize 
\begin{equation}
    \prk=\frac 1 {1+ (9\Lsup)^{\frac 1 {q-1}} \normM {\g(x^k)} {x^k} ^{* \frac {q-2} {q-1}}},
\end{equation}
up to a constant recovering schedules of both \aicn{} stepsize \eqref{eq:aicn} (for $\Lsuppnu 2 1$-Hölder continuous functions, $q=3$) and constant stepsizes of \citet{karimireddy2018global,RSN} (for $\Lsuppnu 2 0$-Hölder continuous functions, $q=2$).


\begin{remark}
Stepsized Newton methods often enjoy much simpler analysis compared to Newton methods regularized in $l_2$ norms, as it is possible to transition easily between gradients and model differences with an exact identity 
\begin{equation} \label{eq:transition}
    \normM {x^{k+1}-x^k}{x^k} \stackrel{\eqref{eq:newton_stepsize}}= \prk \normMd {\g(x^k)}{x^k}.
\end{equation}
\end{remark}

\subsection{Higher order of regularization}
Extending cubic regularization \citep{nesterov2006cubic}, tensor methods achieve better convergence guarantees by regularizing $p$-th order Taylor approximations by $(p+1)$-th order regularization (survey in \citet{kamzolov2023exploiting}).

For third-order tensor methods, \citet{nesterov2021implementable} showed that regularization can avoid computation of third-order derivatives, and \citet{doikov2024super} simplified regularization using technique of \citet{mishchenko2023regularized} to
    \begin{align}
        x^{k+1}&=T(x^k), \, \text{where for $ \expo, \regc \geq 0$,} \label{eq:intro_reg_doikov} \\
        T(x) 
        &= \argmin_{y\in \R^d} \left \{ 
        \taylor x (y)
        + \frac \regc {2} \normsM {y-x} 2 \normM {\g(x)} 2 ^{\expo} \right\}.
    \end{align}
Combining insights about higher-order regularization with the regularization-stepsize duality of \citet{hanzely2022damped}
, we show that the higher-order regularization in local norms
\begin{align}
    x^{k+1}&=\modelun{x^k}, \, \text{where for $ \expo, \regc \geq 0$,}\label{eq:intro_reg_un}\\
    \modelun{x} &= \argmin_{y \in \R^d} \left \{ 
    \taylor x (y)
    + \frac \regc {2+\expo} \normM {y-x} x ^{2+\expo} \right\}, \label{eq:intro_reg_un_model}
\end{align}
is equivalent to a Newton method with stepsize $\prk\in(0,1]$, where $\prk$ is the \emph{unique} positive root of the polynomial $P[\alpha] \eqdef 1- \alpha - \alpha^{1+\expo} \regc \normM {\g(x^k)} {x^k} ^{* \expo}$. Even though the polynomial $P$ lacks an explicit formula for its roots, we derive algorithm \rn{} (\Cref{alg:rn}) with a simple and exactly computed stepsize.

This method can be viewed as a third-order tensor method, as the model \eqref{eq:intro_reg_un_model} bounds the third-order term of Taylor polynomial similarly to \citep[Lemma 3]{nesterov2021implementable}.
\begin{lemma}\label{le:td_bound}
    Let function $f:\R^d \to \R$ be third-order $\Lsupp 3$-Hölder continuous (Def. \ref{def:holder}). Then $\forall x^k, x^{k+1} \in \R^d,$
    \begin{align*}
    \normMd {\nabla ^3 f(x^k) [\xdiff]^2} {x^k}
    &\\
    &\hspace{-0.8cm} \leq 2 \lr \frac {\Lsupp 3 } {1+\nu} \rr ^{\frac 1 {1+\nu}}
     \normsM {\xdiff} {x^k}.
    \end{align*}
\end{lemma}

\subsubsection*{Generality of higher-order regularization} \label{ssec:regularization_generalized}
Investigating generality of the regularization \eqref{eq:intro_reg_un_model}, w can observe that \eqref{eq:intro_reg_un_model} also encapsulates all polynomial upper bounds of polynomials $P[\normM {x-y}x]$ with smaller exponents. 
Writing regularization as a polynomial,
\begin{align}
    f(y) 
    &\leq 
    \taylor x (y)
    + P[\normM {x-y}x],\\
    \intertext{this can be bounded as}
    f(y) &\leq 
    \taylor x (y)
    + A_1 + A_2 \normM {x-y}x^p, \label{eq:poly_simplified}
\end{align}
where constants $A_1, A_2>0$ and degree $p$ are expressed in the lemma below. Notably, the next iterate $x^+$ set as the minimizer of the right-hand side of \eqref{eq:poly_simplified} is not affected by $A_1$, but the $A_1$ worsens guarantees on functional value decrease, $f(x^+) \leq f(x)+A_1.$
 \begin{lemma}\label{le:polynomial_ag}
        A polynomial $P$ with $d_P$ coefficients $a_k\geq 0$ and exponents $0 \leq b_1 \leq \dots \leq  b_{d_P}$, 
        \begin{equation*}
            P[x] \eqdef \sum _{k=0} ^{d_P} a_k x^{b_k},
        \end{equation*}
        satisfies following bound with any $p \geq \max_{k \in \lc 1, \dots, d_P \rc}{b_k}$,
        \begin{equation*}
            P[x] \leq A_1 + A_2 x^p, 
        \end{equation*}
        where $A_1= \frac 1 p \sum _{k=0} ^{d_P} a_k (p-b_k), A_2= \frac 1 p \sum _{k=0} ^{d_P} a_k b_k$.
    \end{lemma}

\textbf{A surprising observation:} Similarly, we can replace even the quadratic term from Taylor polynomial, $\frac 12 \normsM {y-x} x$, by an upper bound in the form $ A_1 + A_2 \normM {x-y}x^p$. This further simplifies the regularization and results in the Newton method with the \textbf{unbounded stepsize}
\begin{equation*}
x^+ = x- \lr\frac {1}{(\regc+1) \normM {\g(x^k)} {x^k} ^{* \expo}} \rr^{\frac 1{1+\expo}} \ls \h(x) \rs^{-1} \g(x).
\end{equation*}
As the gradient diminishes, the stepsize diverges to infinity. Yet, simultaneously, the functional value is guaranteed to not deteriorate by more than a constant factor. 
We refer the reader to the \Cref{sec:simplified_reg} for more details.

\section{Contributions}
Our contributions can be summarized as follows:
\begin{itemize}[leftmargin=*]
    \setlength\itemsep{0em}
    \item \textbf{Newton method as a third-order tensor method:}\\
    We analyze the stepsized Newton method for functions with Hölder continuous third-derivatives (\Cref{def:holder}). This reframes the classical second-order Newton method as a \textbf{third-order} method, bridging the gap between second-order methods and third-order tensor methods.
    
    \item \textbf{Simple stepsizes for fast global convergence:}\\
    We propose multiple stepsize schedules for the Newton method (\rn{}, Alg \ref{alg:rn}), leveraging \textbf{various} Hölder continuity assumptions (Def \ref{def:holder}). Although the stepsize is chosen to be a root of a non-quadratic polynomial, it is surprisingly \textbf{simple and directly computable}.
    
    Depending on the considered variant of the Hölder continuity assumption, they can achieve a global convergence rate up to $\cOb {k^{-3}}$ (\Cref{th:holder_un}). 
    These are the first Newton method stepsizes improving upon the rate \okd of \citet{hanzely2022damped}.

    Additionally, we establish the following guarantees:
    \begin{itemize}[leftmargin=*, noitemsep, topsep=0pt]
        \item a \textbf{local superlinear} convergence rate (\Cref{th:superlinear}), 
        
        \item a \textbf{global linear} convergence (Theorems \ref{th:linrate_1}, \ref{th:linrate2}) under additional assumption of finite \textit{$s$-relative size} (\Cref{def:relsize}) \citep{doikov2024super},
        
        \item and a \textbf{global superlinear} convergence (\Cref{th:optami_superlinear}) under the additional assumption of uniform star-convexity (\Cref{def:star-convexity}) of degree $s\geq 2$.
    \end{itemize}
    
    \item \textbf{Stepsize linesearches for unknown parameters:}\\
    In practice, smoothness constants are often unknown, requiring approximation or fine-tuning. To address this, we introduce a \textbf{linesearch} procedure \grls{} \eqref{eq:grls} and a \textbf{stepsize backtracking} method \un{} (\Cref{alg:un}), both of which provably converge as if the \textbf{optimal} parameterization was known in advance (Col \ref{col:linesearch}, Th \ref{th:global_universal}).

    \item \textbf{Guarantees for popular Newton linesearch:}\\
    As a byproduct of our analysis, we prove similar convergence guarantees for the popular Newton method with greedy linesearch \eqref{eq:greedy_linesearch} (Col \ref{col:greedy_linesearch}, Th \ref{th:optami_superlinear}). This is, to our best knowledge, the first result of such kind.

    \item \textbf{Experimental comparison:}\\
    In \Cref{sec:experiments}, we experimentally compare the proposed algorithms (\rn{}, \un{}, and \grls{}) with existing methods and demonstrate that they outperform their counterparts in most of the considered scenarios.
    
    Also, we observe that the stepsizes of linesearch procedure \grls{} closely resemble stepsizes of popular Greedy Newton linesearch. 
\end{itemize}

\subsection{Most relevant literature}
Our theoretical framework leverages multiple insights of works \citet{hanzely2022damped} and \citet{doikov2024super}. We will outline the key differences between those approaches.\\

Compared to our approach, the \aicn{} method of \citet{hanzely2022damped} is restricted to cubic regularization and achieves only an \okd{} convergence rate. In contrast, our schedules incorporate a range of smoothness notions, including the Hölder continuity of the third derivative, allowing \Cref{alg:rn} to achieve rates up to $\cOb {k^{-3}}$. Additionally, while \aicn{} requires prior knowledge of the smoothness constant, our backtracking linesearch \Cref{alg:un} provably converge as if the optimal parameterization was known in advance.

Furthermore, while \citet{hanzely2022damped} relies on cubic regularization, resulting in a stepsize that is the root of a quadratic polynomial, higher-order regularizations yield a stepsize that is the root of a higher-order polynomial. Surprisingly, we show that even with higher-order regularization, there is a unique positive root in the interval $(0,1]$, and we present algorithms (\Cref{alg:rn} and \Cref{alg:un}) that can operate without requiring any additional linesearch.\\

In comparison to \citet{doikov2024super}, which utilizes standard $l_2$ norms for regularization, our approach leverages the local Hessian norms suggested by \citep{hanzely2022damped}. By utilizing local norms, the minimizers of various regularization models \eqref{eq:intro_reg_un_model} align on the same line, naturally connecting different regularizations from a geometric perspective. Local norms also result in a simpler algorithm invariant to linear transformations (e.g., data scaling or choice of basis), which is a valuable property in practice, as it significantly reduces hyperparameter tuning.

We would like to highlight that our results explain the success of popular stepsize linesearches in the Newton direction. 
These insights have implications far beyond our newly proposed methods. In comparison, the results presented in \citet{doikov2024super} do not provide a novel theoretical explanation for any established method.


\section{Simple stepsize schedule} \label{sec:alg}
Now we are ready to present our new stepsize schedule.
\begin{theorem} \label{th:equivalence}
    For any constants $\regc,\expo \geq 0$, the following modifications of the Newton method are equivalent:
    \begin{align}
    \text{Regularize: } x^{k+1} &= x^k + \argmin_{y\in \R^d} \modelun {x^k}, \label{eq:un_reg}\\
    \text{Damping: } x^{k+1} &= x^k - \prk [\h(x^k)]^{-1} \g(x^k), \label{eq:un_newton}
    \end{align}
    where,
    \begin{align*}
        & \modelun{x} = \argmin_{y\in\R^d} \left \{ 
        \taylor x (y)
        + \frac \regc {2+\expo} \normM {y-x} x ^{2+\expo} \right\}, \\ 
         & \text{and } \prk \in (0,1] \text{ is the only positive root of polynomial}\\
         & P[\alpha] \eqdef 1- \alpha - \alpha^{1+\expo} \regc \normM {\g(x^k)} {x^k} ^{* \expo}.
    \end{align*}
    We call this algorithm Root Newton (\rn), \Cref{alg:rn}.
\end{theorem}
To simplify calculations, we reparametrize the \rn{} as
\gbox{$\alun \eqdef \pr^\expo \regc \normM{\g(x)} x ^{*\expo}$}, and $\alun \geq 0.$ 
Now, the polynomial $P$ simplifies to $P[\alpha] = 1-\alpha - \alpha \alun$ and for fixed $\alun$, the positive root of $P$ can be expressed as $\pr = \frac 1 {1+\alun}$, with $\pr \alun < 1.$

\subsection{Hölder continuity}
Our analysis is built upon the assumption that the function has Hölder continuous Hessian or third derivative. 
\begin{definition}\label{def:holder}
For $f: \rdtor,$ and $p\in \N,$ we say that $p$-times differentiable convex function is \emph{Hölder continuous} of $p$-th order, if for some $\nu \in [0,1]$ there exists a constant $\Lsup <\infty$, such thata $\forall x,y \in \R^d,$
\begin{align}
    \normM {\nabla^p f(x) - \nabla ^p f(y)} {op} \leq \Lsup \normM {x-y} x ^\nu, \qquad
\end{align}
We say that the $f$ has Hölder continuous Hessian if $\Lsupp 2<\infty$ (for some $\nu\in[0,1]$) and Hölder continuous third derivative if $\Lsupp 3<\infty$ (for some $\nu \in [0,1]$).  
\end{definition}
\newcommand{\gx}{g_x}
\begin{table*}[t]
    \centering
    \setlength\tabcolsep{4pt} 
    \begin{threeparttable}[b]
        {
            \renewcommand\arraystretch{2.2}
            \caption{Global convergence guarantees of stepsized Newton methods under various notions of Hölder continuity (\Cref{def:holder}). For simplicity, we report dependence only on the number of iterations $k$.}
            \label{tab:guarantees}
            \centering 
            \begin{tabular}{ccccc}\toprule[.2em]
                \bf Stepsize schedule & \makecell{ \bf Stepsize for \\ $\gx \eqdef \normMd{\g(x)} x$} &  \bf \makecell{Smoothness \\assumption} & \bf Global rate & \bf Reference \\
                \bottomrule[.2em]
                Damped Newton B & $\frac 1 {1+L_{sc} \gx}$\tn{0} & $L_{sc}$\tn{0} & $\cO\lr k^{-\frac 12} \rr$\tn{1} & \citep{nesterov1994interior}\tn{1} \\
                \hline
                \aicn{} & $\frac 2 {1+\sqrt{1+ 2\Lsuppnu 21 \gx}}$ \tn{2} & $\Lsuppnu 21$ & $\cO\lr k^{-2} \rr$ & \citep{hanzely2022damped} \\
                \hline
                \makecell{\rn{} \\ (\Cref{alg:rn})} & $\frac 1 {1+ (9\Lsup)^{\frac 1 {q-1}} \gx^{\frac {q-2} {q-1}}}$\tn{3} & $\Lsup$\tn{3} & $\cOb{k^{-(p+\nu-1)}}$\tn{3} & \makecell{\textbf{This work} \\(\Cref{th:global})} \\
                \hline
                \makecell{\grls{} \eqref{eq:grls}} & Linesearched & \makecell{$\Lsup$\tn{3}\\ (unknown)} & $\min_{p,\nu}$ $\cOb{k^{-(p+\nu-1)}}$\tn{3}  & \makecell{\textbf{This work} \\(\Cref{col:linesearch})}\\
                \hline
                \makecell{\un{} \\ (\Cref{alg:un})} & Backtracked & \makecell{$\Lsup$\tn{3}\\ (unknown)} & $\min_{p,\nu}$ $\cOb{k^{-(p+\nu-1)}}$\tn{3} & \makecell{\textbf{This work} \\ (\Cref{th:global_universal})} \\
                \hline
                \makecell{Greedy Newton\\
                \eqref{eq:greedy_linesearch}} & Linesearched & \makecell{$\Lsup$\tn{3}\\ (unknown)} & $\min_{p,\nu}$ $\cOb{k^{-(p+\nu-1)}}$\tn{3}  & \makecell{Folklore\\ Rate: \Cref{col:greedy_linesearch} \textbf{(new)}}\\
                \bottomrule[.2em]
            \end{tabular}
        }
        \begin{tablenotes}
            {
                \item [\color{red}(0)] Constant $L_{sc}$ represents self-concordance constant and is implied by $\Lsuppnu 21$-Hölder continuity.
                \item [\color{red}(1)] Authors show global decrease $f(x^{k+1})\leq f(x^k)-c$ for some $c>0.$ Rate $\cO (k^{-\frac 12} )$ is reported in \citet{hanzely2022damped}. We were unable to prove or find the convergence guarantee for Damped Newton B of the form $\cO(k^{-\alpha})$.
                \item [\color{red}(2)] We present a simplified form of the stepsize. Authors proposed \aicn{} stepsize in equivalent form $\frac {-1+\sqrt{1+ \Lsuppnu 21 \gx}}{\Lsuppnu 21 \gx}$.
                \item [\color{red}(3)] Parameters $p,\nu$ are fixed and satisfy $p\in \lc 2,3 \rc, \nu \in [0,1]$ and $p+\nu -1 \in [1,3]$.
            }
        \end{tablenotes}
    \end{threeparttable}
\end{table*}
In particular, $\Lsuppnu 3 0 = \normM{\nabla^3 f(x)-\nabla ^3 f(y)} {op}$ and $\Lsuppnu 2 1 = \sup_x \normM{\nabla^3 f(x)}{op}$ matches the definition of semi-strong self-concordance \citep{hanzely2022damped}. Function $\Lsup$ is log-convex in $\nu$ and hence for $0 \leq \nu_1 \leq \nu \leq \nu_2 \leq 1,$ hold
\begin{align*}
    \Lsup &\leq 
    \ls \Lsupnu {\nu_1} \rs^{\frac {\nu_2 -\nu} {\nu_2 -\nu_1}} 
    \ls \Lsupnu {\nu_2} \rs^{\frac {\nu -\nu_1} {\nu_2 -\nu_1}}, \\
    \Lsup &\leq \Lsupnu 0 ^{1-\nu} \Lsupnu 1 ^\nu.
\end{align*}
We will use the properties of the Hölder continuity summarized in the proposition below. 

\begin{proposition}\label{prop:holder_eq}
    $\Lsupp 2$-Hölder continuous functions satisfy
    \begin{equation*}
        \normMd { \g(y) - \g(x) - \h(x)[y-x] } x
        \leq \frac {\Lsupp 2 }{1+\nu} \normM{y-x} x ^{1+ \nu}.
    \end{equation*}
    $\Lsupp 3$-Hölder continuous functions satisfy
    \begin{align*}
        &\Big \Vert \g(y)-\g(x)-\h(x) \ls y-x \rs - \\
        & \quad \left. -\frac 12 \td(x)[y-x]^2 \right \Vert_x^* \leq \frac{\Lsupp 3}{(1+\nu)(2+\nu)} \normM{y-x}x^{2+\nu}.
    \end{align*}
\end{proposition}
For further discussion of smoothness constants, we refer the reader to \Cref{sec:smoothness_relations}.

\subsection{One-step decrease Hölder continuity}
We are going to show that from the Hölder continuity for sufficiently large $\alunk$ follows bound 
    \begin{align*}
        &\la \g(x^{k+1}), \ls \h(x^k)\rs^{-1} \g(x^k) \ra \\
        & \hspace{3cm} \geq \frac 1{2c_1(1-\prk)}  \normsMd{\g(x^{k+1})}{x^k},
    \end{align*}
for $c_1\in \lc 1,2\rc$, implying the one-step decrease as
 \begin{align}
        f(x^{k}) - f&(x^{k+1})\nonumber\\
        &\geq -\la \g (x^{k+1}), \xdiff \ra \nonumber \\
        &= \la \g(x^{k+1}), \prk \ls \h(x^k) \rs^{-1} \g(x^k) \ra \nonumber \\
        &\geq \frac {\prk} {2c_1 (1-\prk)}\normsMd { \g(x^{k+1})} {x^k} \nonumber\\
        &= \frac 1 {2c_1\alunk} \normsMd { \g(x^{k+1})} {x^k}. \label{eq:preholder}
    \end{align}

\begin{lemma}\label{le:holder_bound_two}
    Let $\normMd {\g(x^k)} {x^k} >0$, and $x^k\in \R^d, x^{k+1}=x^k-\prk \ls\h(x^k)\rs^{-1}\g(x^k)$, as \rn{}.    
    Hölder continuity of \textbf{Hessian} (\Cref{def:holder} with $p=2$) implies that for $\alunk$ larger than
    \begin{align}
        \alunk \geq \frac {\Lsupp 2} {1+\nu} \prk^\nu \normM {\g(x^k)}{x^k} ^{*\nu}, \label{eq:alunkmintwo}
    \end{align}
    holds
    \begin{align*}
        &\la \g(x^{k+1}), \ls \h(x^k)\rs^{-1} \g(x^k) \ra \\
        & \hspace{3.5cm}\geq \frac 1{2(1-\prk)}  \normsMd{\g(x^{k+1})}{x^k}.
    \end{align*}
\end{lemma}

\begin{lemma} \label{le:holder_bound_three}
    Let $\normMd {\g(x^k)} {x^k} >0$, and $x^k\in \R^d, x^{k+1}=x^k-\prk \ls\h(x^k)\rs^{-1}\g(x^k)$, as \rn{}.
    Hölder continuity of the \textbf{third derivative} (\Cref{def:holder} with $p=3$) implies that for $\alunk$ larger than
    \begin{align}
        \alunk 
        &\geq  \prk \normMd {\g(x^k)}{x^k} \max \left \{ 6 \lr \frac  {\Lsupp 3 } {1+\nu} \rr ^{\frac 1 {1+\nu}}, \nonumber \right. \\
        & \hspace{1cm} \left .\frac {\sqrt {3} \Lsupp 3 }{(1+\nu) (2+ \nu)}  \lr\prk \normMd {\g(x^k)}{x^k} \rr ^{\nu} \right \},  \label{eq:alunkminthree}
    \end{align}    
    holds
    \begin{align*}
        &\la \g(x^{k+1}), \ls \h(x^k) \rs ^{-1}\g(x^{k}) \ra\\ 
        & \hspace{3.5cm}\geq \frac 1{4(1-\prk)} \normsMd { \g(x^{k+1})}{x^k}.
    \end{align*}
\end{lemma}

\subsection{Generalized one-step decrease}

In \Cref{le:holder_bound_two} and \Cref{le:holder_bound_three}, the requirement on $\alunk$ is dependent on $\prk$. We can use the following observation to derive a bound dependent only on the norm of the gradient. 
\begin{lemma}\label{le:alunk_alpha}
     For $c_3, \delta>0$, choice $\alunk \geq c_3^{\frac 1{1+\delta}} \normM{\g(x^k)}{x^k} ^{*\frac {\delta}{1+\delta}}$
    ensures
        $\alunk \geq c_3 \lr \prk \normMd{\g(x^k)}{x^k} \rr ^{\delta}.$
\end{lemma}
With \Cref{le:alunk_alpha}, we can unify the cases $p \in \{2,3\}$ (see \Cref{col:holder_three} for the additional explanation). Let us reparametrize as \gbox{$q\eqdef p+\nu \in [2,4]$}, \gbox{$\Lq \eqdef \Lsup$}.
\begin{theorem} \label{th:holder_un}
    Let $\normMd {\g(x)} x >0$. 
    Hölder continuity (\Cref{def:holder}) with $q=p+\nu\in[2,4]$ for points $x^k, x^{k+1}=x^k-\prk \ls\h(x^k)\rs^{-1}\g(x^k)$ from $\rn{}$ implies that for $\alunk$ such that
    \begin{align}
        \alunk 
        &\geq \lr 9\Lq \rr^{\frac 1 {q-1}}  \normM {\g(x^k)}{x^k} ^{*\frac {q-2}{q-1}}     \label{eq:alunkminun}   
    \end{align}    
    holds
    \begin{align}
        &\la \g(x^{k+1}), \ls \h(x^k) \rs ^{-1}\g(x^{k}) \ra \nonumber \\ 
        & \hspace{3cm} \geq \frac 1{2\prk\alunk} \normsMd { \g(x^{k+1})}{x^k}. \label{eq:ip_un}
    \end{align}
    In particular, in view of \eqref{eq:preholder}, we have that the choice $\alunk = \lr 9\Lq \rr ^{\frac 1 {q-1}} \normM {\g(x^k)}{x^k}^{*\frac {q-2}{q-1}}$ guarantees decrease
    \begin{align}
        f(x^k)-f(x^{k+1})
        \geq \frac 12 \lr \frac {1} {9\Lq} \rr^{\frac 1 {q-1}} \frac{\normsMd {\g(x^{k+1})}{x^k}}{ \normM {\g(x^k)}{x^k}^{*\frac{q-2}{q-1}} }. \label{eq:one_step_gen}
    \end{align}
\end{theorem}

This naturally leads to an optimization algorithm \rn{}.
\begin{algorithm} 
    \caption{\rn{}: Root Newton stepsize schedule} \label{alg:rn}
    \begin{algorithmic}[1]
        \State \textbf{Requires:} Initial point $x^0 \in \R^d,$ Hölder continuity exponent $q\in [2,4]$ and constant $\Lq< \infty$.
        \For {$k=0,1,2\dots$}
            \State $n^k = \ls\h(x^k)\rs^{-1}\g(x^k)$ \Comment{Newton direction}
            \State $g_k = \la \g(x^k), n^k \ra ^{\frac 12}$ \Comment{$g_k= \normMd{\g(x^k)} {x^k}$}
            \State $\alunk = \lr 9\Lq \rr^{\frac 1 {q-1}} g_k^{\frac{q-2}{q-1}}$
            \Comment{Sufficient regularization}
            \State $\prk = \frac 1 {1+\alunk}$
            \Comment {$\prk$ is the root of $P[\alpha]$}
            \State $x^{k+1} = x^k - \prk n^k$ \Comment{Step, $x^k = \modelunlexpo{\regck}{\expo}{x^k}$}
        \EndFor
    \end{algorithmic}
\end{algorithm}

\section{Convergence garantees of \rn{}}
Denote the functional suboptimality \gbox {$f_k\eqdef f(x^k) - \fopt,$} the initial level set \gbox{$\level \eqdef \left\{ x\in\R^d: f(x) \leq f(x^0) \right\},$} and its diameter as \gbox{$D \eqdef \sup_{x,y \in \level} \normM {x-y}x.$}
Note that convexity and bounded diameter of $\level$, $D<\infty$ together imply $D \normMd {\g(x^k)} {x^k} \geq f_k$.
We need the Hessian not to change much between iterations to guarantee the global convergence rate.
\begin{assumption}\label{def:cbound}
There exists a $\cbound >0$ bounding norms of the gradients in the consecutive iterates,
 \begin{align*}
    \cbound \leq \frac{\normsMd {\g(x^{k+1})}{x^k}}{ \normsMd {\g(x^{k+1})}{x^{k+1}}}. 
 \end{align*}    
\end{assumption}
Required $\cbound$ exists in many cases. 
        For $L$-smooth $\mu$-strongly convex functions, $\cbound=\frac \mu L$.
        For functions with $\hat c$-stable Hessian \citep{KSJ-Newton2018}, $\cbound=\hat c$.
        For $\Lstandard$-self-concordant functions, it holds when the points $x,x^+$ are close to each other \citep{nesterov1994interior} or in the neighborhood of the solution (\Cref{prop:lsconv}).
        \begin{proposition} [\citet{hanzely2022damped}, Lemma 4]
        	\label{prop:lsconv}
        		For convex $\Lstandard$-self-concordant function $f: \R^d \to \R$ and for any $0<c_4<1$ in the neighborhood of solution $x^k \in \left\{ x :  \normMd {\g (x)} {x} \leq \tfrac{(2c_4+1)^2-1}{2\Lstandard} \right\}$ holds
            \begin{equation*}
        	\h(x^{k+1})^{-1} \preceq \left( 1 - c_4 \right)^{-2} \h(x^k)^{-1}. 
        		 \end{equation*}
        \end{proposition}
First, we present the local convergence of the \rn{}.
\begin{theorem} \label{th:superlinear}
Let function $f: \R^d \to \R$ be convex, $\Lsup$-Hölder continuous ($q=p+\nu$) with $\cbound$-bounded Hessian change \eqref{def:cbound}. Algorithm \rn{} has a superlinear local convergence rate,
\begin{equation*}
    \normMd { \g(x^{k+1})}{x^{k+1}} 
    \leq \frac 2 \cbound \lr 9\Lq \rr^{\frac 1 {q-1}}  \normM {\g(x^k)}{x^k} ^{*\lr 2-\frac {1}{q-1} \rr}.
\end{equation*}
\end{theorem}
Now we quantify the global convergence rate following from \Cref{th:holder_un} and present the rate of \rn{}.
\begin{lemma}\label{le:one_step_to_rate}
    Let function $f: \R^d \to \R$ be convex with $\cbound$-bounded Hessian change \eqref{def:cbound} and the bound level sets with diameter $D$. If an algorithm $\mathcal A$ generates the iterates $\lc x^k\rc_{k=1}^{n}$ with one-step decrease for $q \geq 2$ and $\cc\geq 0$ as
    \begin{equation} \label{eq:one_step_to_rate_requirement}
        f(x^k)-f(x^{k+1})
        \geq \cc \frac{\normsMd {\g(x^{k+1})}{x^k}}{ \normM {\g(x^k)}{x^k}^{*\frac{q-2}{q-1}} },
    \end{equation}
    then $\mathcal A$ has the global convergence rate
    \begin{equation*}
        f_n \leq\frac {D^q \lr 2\cbound (q-1) \rr ^{q-1}}{\cc^{q-1} n^{q-1}} + \normMd {\g(x^0)} {x^{0}} D \exp\lr -\frac k4\rr.
    \end{equation*}
\end{lemma}

\begin{theorem} \label{th:global}
     Let function $f: \R^d \to \R$ be convex, $\Lsup$-Hölder continuous ($q=p+\nu$) with $\cbound$-bounded Hessian change \eqref{def:cbound} and the bound level sets with diameter $D$. \rn{} (\Cref{alg:rn}) with known parameters $q,\Lq$ converges as 
    \begin{align*}
        f(x^k)-\fopt \leq
        & \frac {9\Lq D^q \lr 4 \cbound (q-1) \rr ^{q-1}}{k^{q-1}} \\
        & \hspace{15mm} + \normMd {\g(x^0)} {x^{0}} D \exp\lr -\frac k4\rr, 
    \end{align*}
    which in $\cO$ notation is simplifies to $\cO \lr \frac {\Lq D^q} {k^{q-1}} \rr$.
\end{theorem}

Note that the loss function can satisfy Hölder continuity (\Cref{def:holder}) with multiple different $\Lsup$, and therefore different pairs ($q$, $\Lq$) can be used. The best parametrization might not be known.

\section{Unknown parametrization}
To address unknown parameterization, we propose a stepsize linesearch Gradient-Regulated Line Search \grls{} simultaneously minimizing loss and gradient norms as
\begin{align}
    x^{k+1} &= \argmin_{y \in \lc x-\alpha n_{x^k} | \alpha \in [0,1] \rc} \frac {f(y)-f(x^k)}{\normsMd {\g(y)}{x^k}}, \label{eq:grls}
\end{align}
where \gbox{$\ndir \eqdef [\h(x)]^{-1}\g(x)$} is a shorthand for Newton's direction at point $x$. Linesearch \grls{} is directly minimizing bound \eqref{eq:one_step_to_rate_requirement} in \Cref{le:one_step_to_rate}, and therefore has the corresponding convergence rate. 

\begin{corollary}\label{col:linesearch}
     Let function $f: \R^d \to \R$, be convex, Hölder continuous with some $\Lq<\infty$, with $\cbound$-bounded Hessian change \eqref{def:cbound}, and the bound level sets with diameter $D<\infty$. 
    Linesearch \grls{} converges as $\min_{q\in [2,4]} \cO \lr \frac {\Lq D^q} {k^{q-1}} \rr$
    \begin{align}
        f(x^k)-\fopt 
        & \leq \min_{q\in [2,4]} \frac {9\Lq D^q \lr 4 \cbound (q-1) \rr ^{q-1}}{k^{q-1}} \nonumber \\
        & \hspace{1cm} + \normMd {\g(x^0)} {x^{0}} D \exp\lr -\frac k4\rr. \label{eq:linesearch_rate}
    \end{align}
\end{corollary}
Observe that for small stepsizes $\prk \in [0, \aup],$ for some $\aup \ll 1$, model differences are small $x^{k+1} \approx x^k$ and $\g(x^k) \approx \g(x^{k+1})$.
Therefore, expression \eqref{eq:grls} minimized by \grls{} can be approximated as
\begin{align}
    \frac {f(y)-f(x^k)}{\normsMd{\g(y)}{x^k}}
    \approx & 
     \frac {f(y)-f(x^k)}{\normsMd{\g(x^k)}{x^k}}, \label{eq:greedy_grls_approx}
\end{align}
and the minimizer of the right-hand-side is equivalent to the practically popular Newton method with greedy linesearch
\begin{align} \label{eq:greedy_linesearch}
    x^{k+1} &= 
    \argmin_{y \in \lc x^k-\alpha n_{x^k} | \alpha \in [0,1] \rc} f(y),
\end{align}
which we will call \emph{Greedy Newton} (\greedy{}). Leveraging this insight, we obtain the convergence rate for \eqref{eq:greedy_linesearch} in the corollary below. More details can be found in \Cref{sec:greedy_newton}.
\begin{corollary}\label{col:greedy_linesearch}
     Let function $f: \R^d \to \R$, be convex, $\Lq$-Hölder continuous for some $\Lq<\infty$, with $\cbound$-bounded Hessian change \eqref{def:cbound}, and the bound level sets with diameter $D<\infty$.
     If the Greedy Newton linesearch \eqref{eq:greedy_linesearch} satisfies the inequality $\normMd {\g(x^{k+1})} {x^k} \leq \csup \normMd{ \g(x^k) }{x^k}$ with some constant $\csup \geq 0$ for all iterates $x^k$, then it has convergence guarantee $\min_{q\in [2,4]} \cO \lr \frac {\Lq D^q \csup^{2(q-1)}} {k^{q-1}} \rr$
    \begin{align*}
        f(x^k)-\fopt 
        & \leq \min_{q\in [2,4]} \frac {9\Lq D^q \lr 4 \cbound \csup^2 (q-1) \rr ^{q-1}}{k^{q-1}} \\
        & \hspace {2cm}+ \normMd {\g(x^0)} {x^{0}} D \exp\lr -\frac k4\rr.
    \end{align*}
\end{corollary}
\begin{algorithm}[t] 
    \caption{\un{}: Universal stepsize backtracking procedure for the Newton method} \label{alg:un}
    \begin{algorithmic}[1]
        \State \textbf{Input:} Initial point $x^0 \in \R^d$, any constants $\expo \in \ls\frac 23,1\rs, \regc_0, \cbound>1$
        \Comment{$\expo \geq \frac {q-2}{q-1}$ for $q\in [2,4]$}
        \For {$k=0,1,2\dots$}
            \State $n^k = \ls\h(x^k)\rs^{-1}\g(x^k)$ \Comment{Newton direction}
            \State $g_k = \la \g(x^k), n^k \ra^{\frac 12} $ \Comment{$= \normMd{\g(x^k)} {x^k}$}
            \For{$j_k=0,1,2\dots$}
                \State $\alunkest {j_k}=\cbound^{j_k} \regck g_k^\expo$ \Comment{Increase regularization}
                \State $\pr_{k,j_k} = \frac 1 {1+\alunkest {j_k}}$ \Comment{Update stepsize}
                \State $x^k_{j_k} = x^k - \pr_{k,j_k} n^k$ \Comment{$= \modelunlexpo{\cbound^{j_k}\regck}{\expok}{x^k}$}
                \If{$\la \g(x^k_{j_k}), n^k \ra \geq \frac 1 {2\pr_{k,j_k} \alunkest {j_k}} \normsMd{\g(x^k_{j_k})} {x^k} $}  
                    \State $x^{k+1}=x^k_{j_k}$
                    \State $\regc_{k+1}=\cbound^{j_k-1} \regc_k$
                    \State \textbf{break}
                \EndIf
            \EndFor            
        \EndFor
    \end{algorithmic}
\end{algorithm} 

Linesearches \grls{} \eqref{eq:grls} and \greedy{} \eqref{eq:greedy_linesearch} have fast convergence guarantees without knowledge of smoothness parametrization $(q,\Lq)$, yet their implicit nature might not be suitable for all practical scenarios.
To remedy that, in the next section, we present a stepsize backtracking procedure with matching convergence guarantees.

\section{Universal stepsize backtracking}\label{sec:un}
Our backtracking procedure is based on the observation that the knowledge of the parametrization $(q,\Lq)$ in \rn{} (\Cref{alg:rn}) is required only for setting $\alunk$. We start with an estimate of $\alunk$ smaller than the true value and increase it until it achieves the theoretically predicted decrease.
We claim that the resulting algorithm, \un{}, \Cref{alg:un}, is well-defined with a bounded number of backtracking steps and a fast global convergence rate.

To formalize this claim, we first define a quantity to identify the smallest plausible true parameter $\alunk$ to be estimated first,
    \gbox{$\univ x \eqdef \inf_{q \in[2,4]} \lr 9\Lq \rr^{\frac 1 {q-1}} \normM {\g(x)} {x}^{*\lr\frac {q-2}{q-1}-\expo\rr},$}
for $q\in [2,4]$ and $\expo\geq \frac 23$.
\begin{lemma} \label{le:universal}
    If $\Lq < \infty$ for some $q\in [2,4]$, and the initial estimate $\regc_0$ small enough,
        $\regc_0 \leq \univ {x^0},$
    then all iterations $\lc x^k \rc_{k=0}^{n}$ of \un{}, such that $\normMd {\g(x^k)} {x^k} >0$, satisfy
        $\regc_{k+1}
        = \frac {\alunkest {j_k-1}}  {\normM {\g(x^{k})} {x^{k}}^{*\expo}}
        \leq \univ {x^k}.$
    Moreover, the total number $N_K$ of backtracking steps during the first $k$ iterations is bounded,
    \begin{align*}
        N_k \leq 2k + \log_c \frac {\univ {x^{k-1}}}{\regc_0}.
    \end{align*}
\end{lemma}
\vspace{-5mm}
\begin{theorem} \label{th:global_universal}
    Let function $f: \R^d \to \R$, be convex, Hölder continuous with some $\Lq<\infty$, with $\cbound$-bounded Hessian change (\Cref{def:cbound}), and the bound level sets with diameter $D<\infty$. \un{} (\Cref{alg:un}) converges with the rate $\min_{q\in [2,4]} \cO \lr \frac {\Lq D^q} {k^{q-1}} \rr,$
    \begin{align*}
        f(x^k)-\fopt 
        &\leq \min_{q\in [2,4]} \frac {9\Lq D^q \lr 4 \cbound^2 (q-1) \rr ^{q-1}}{k^{q-1}} \\
        & \hspace{1.5cm} + \normMd {\g(x^0)} {x^{0}} D \exp\lr - k/4\rr .
    \end{align*}
\end{theorem}

\section{Global (super)linear convergence rate}
Stepsized Newton method is known to be able to achieve a global linear rate if the Hessian is bounded and stepsize is constant \citep{karimireddy2018global, gower2019rsn}, or when the function is $\Lsuppnu 21$-Hölder continuous with stepsize following schedule \aicn{} \citep[proof in \citep{hanzely2023sketch}]{hanzely2022damped}. 

In line with those results, we present global linear rates for algorithms \rn{}, \un{}, \grls{} on $\Lsup$-Hölder continuous functions with finite $(p+\nu)$-relative size characteristic \citep{doikov2024super}. The proof is in \Cref{sec:doikov_linear}.

\begin{definition}[\citep{doikov2024super}]\label{def:relsize}
For strictly convex function $f: \R^d \to \R$ we call \emph{$s$-relative size} characteristic 
\begin{align*}
    \ds &\eqdef \sup_{x,y \in \level} \lc \normM {x-y}x \lr \frac {\brmax}{\br xy} \rr^{\frac 1s} \rc,
\end{align*}
where $\br xy \eqdef \la \nabla f(x)-\nabla f(y) , x-y \ra>0$ and $\brmax \eqdef \sup_{x,y \in \level} \br xy$.
    \end{definition}
\vspace{-5mm}
\begin{theorem}\label{th:doikov_linear_from_main}
    Let function $f$ be $\Lsup$-Hölder continuous, with finite relative size $D_q<\infty$ for $q=p+\nu$ (\Cref{def:relsize}) and $\cbound$-bounded Hessian change (\Cref{def:cbound}).
    Algorithms \rn{}, \un{} and \grls{} find points in the $\varepsilon$-neighborhood, $f(x^k)-f(x^*) \leq \varepsilon$, in
    \begin{align*}
        k
        &\leq  \cO \lr \cbound \lr \frac {\Lq\dss q ^q}{\brmax} \rr^{\frac 1{q-1}} \ln \frac {f_0} \varepsilon 
        + \ln \frac {\normMd {\g(x^0)} {x^{0}} D} {\varepsilon} \rr
        \hspace{-4.5mm}
    \end{align*}
    iterations, implying a global linear convergence rate.
\end{theorem}
\begin{remark}
    In view of \eqref{eq:greedy_grls_approx}, analogous convergence guarantee (with a worse constant) can be proven for \greedy{}.
\end{remark}
Replacing relative size assumption with uniform star-convexity of degree $s$ ($q>s \geq 2$), we can guarantee a global superlinear rate for \rn{} and \greedy{} similarly to \citet{kamzolov2024optami}.
The proof is in \Cref{sec:glob_superlinear}. 
\begin{definition} \label{def:star-convexity}
    For $s\geq2$ and $\mu_s\geq 0$ we call function $f:\R^d \to \R$ \emph{$\mu_s$-uniformly star-convex} of degree $s$ in local norms with respect to a minimizer $\xopt$ if $\forall x \in \R^d, \forall \eta \in [0,1]$ holds
    \begin{align*}
        f\lr \eta x + (1-\eta) \xopt \rr  
        &\leq \eta f(x) + (1-\eta) \fopt \\
        & \qquad - \frac {\eta (1-\eta) \mu_s}s \normM {x-\xopt} x ^s.
    \end{align*}
    If this inequality holds for $\mu_s=0$, we call function $f$ \emph{star-convex in local norms} (w.r.t. minimizer $\xopt$).
\end{definition}

\begin{theorem}\label{th:optami_superlinear}
    Let the function $f:\R^d \to R$ be $\Lsup$-Hölder continuous (\Cref{def:holder}) and $\mu_s$-uniformly star-convex of degree $s$ in local norms (\Cref{def:star-convexity}) and $q\eqdef p+\nu\geq s \geq 2$ then \rn{} and \greedy{} have global decrease in functional value suboptimality,
    \begin{align*}
    f(x^{k})-\fopt &
    \leq \lr f(x^0)-\fopt\rr \prod_{t=0}^{k-1} (1-\hat \eta_t),
\end{align*}
where $\hat \eta_k \in[0,1]$ is the only positive root of $E_k(\eta) \eqdef (1-\eta) \frac {\mu_s}s - \eta^{q-1} \lr \frac {\Lq} {(p+1)!} + \frac {\regc} {q} \rr \normM {x^k - \xopt}{x^k} ^{q-s}$. 

If $q=s$, then $\hat \eta_k$ is constant throughout all iterations and the rate is \textbf{globally linear}. 

If $q>s,$ then $\hat \eta_k$ is monotonically increasing as $\normM {x^k - \xopt}{x^k}$ decreases, $1-\hat \eta_k \to 0$, and therefore, the resutling rate is \textbf{globally superlinear}. 
\end{theorem}

\section{Numerical experiments} \label{sec:experiments}
\subsection*{Logistic regression}

In \Cref{fig:logistic}, we compare the performance of the proposed algorithms on binary classification on datasets from \libsvm{} repository \citep{Chang2011LIBSVM}. For datapoints ${\{(a_i,b_i)\}}_{i=1}^n,$ where $a_i \in \R^d, b_i \in \{-1, +1\}$, and regularizer $\mu = 10^{-3}$, we aim to minimize
\begin{align*}\label{eq:logistic}
    \min_{x \in \R^d }\Big\{ 
    f(x) &= 
    \frac{1}{n}\sum_{i=1}^n \log \lr 1 + e^{-b_i \langle a_i, x\rangle} \rr + \frac{\mu}{2} \|x\|_2^2\Big\}.
\end{align*}
We initialize all methods at $x_0 = 10 \cdot [ 1, 1, \dots , 1 ]^T \in \R^d.$

\subsection*{Polytope feasibility}
In \Cref{fig:polytope}, we compare proposed algorithms on \textit{polytope feasibility} problem, aiming to find a point from a polytope $\mathcal{P} = \Big\{ x \in \R^d: \langle a_i, x \rangle \leq b_i, \; 1 \leq i \leq n \Big\} $, reformulated as
\begin{equation}\label{eq:polytope}
    \min_{x \in \R^d }\Big\{ f(x) = \sum_{i=0}^{n} (\langle a_i, x \rangle - b_i)^p_+ \Big\},
\end{equation}
where $(t)_+ \eqdef \max\{t, 0\}$ and $p \geq 2$. 
We generate data points $(a_i, b_i)$ and the solution $x^*$ synthetically as $a_i, x^* \sim \mathcal{N}(0,\,1)$ and set $b_i = \la a_i, x^* \ra$.

We initialize all methods at $x_0 = [ 1, 1, \dots , 1 ]^T \in \R^d$.

\subsection*{Rosenbrock function}
Linesearch procedures solve the abovementioned problems in just a few steps. 
For a more challenging task, \Cref{fig:rosenbrock} presents the notorious $d$-dimensional \textit{Rosenbrock} function,
\begin{equation}\label{eq:rosenbrock}
    \min_{x \in \R^d }\Big\{ f(x) = \sum_{i=0}^{d-1} [100 (x_{i+1} - x_i^2)^2 + (1 - x_i)^2]\Big\}.
\end{equation}
Notably, the Rosenbrock function \eqref{eq:rosenbrock} is nonconvex, which breaks assumptions in our convergence theorems.

The function \eqref{eq:rosenbrock} has the global solution at $x^*=[1, \dots, 1]^T$, and therefore we choose the initial point from a normal distribution, $x^0 \sim \mathcal{N}(0, I_d) \cdot 20$.

\subsection{Experimental comparison}
In Figures \ref{fig:logreg-rn}, \ref{fig:poly-rn}, we compare higher-order methods \textit{without} any linesearch procedures, namely \rn{} (\Cref{alg:rn}), \aicn{} \citep{hanzely2022damped} and Gradient Regularization of Newton Method (\grn{}) \citep[Alg. 1]{doikov2024super}. 
As additional baselines, we use the damped Newton method with a fixed fine-tuned stepsize and classical first-order Gradient Method (\gm{}) \citep{10.5555/3317111}.
\rn{}  and \aicn{} show similar performance while \grn{} has a slight disadvantage. As expected, the first-order method \gm{} that does not utilize Hessian has quicker iterations but slower per-iteration convergence.

In Figures \ref{fig:logreg-univ}, \ref{fig:poly-univ}, we compare higher-order regularization methods \textit{with} smoothness constant estimation procedures, \un{} and Super-universal Newton method \citep[Alg. 2]{doikov2024super}.
As an additional baseline, we use the damped Newton method with a fixed but fine-tuned stepsize.
We show that \un{} displays faster convergence than the Super-universal Newton method. Moreover, we show that the exponent of the regularization term $\expo$ that appears in both \un{} and super-universal Newton method \eqref{eq:intro_reg_doikov} does not have a significant impact on overall performance.

Figures \ref{fig:logreg-linesearch}, \ref{fig:poly-linesearch}, \ref{fig:rosenbrock} compare implicit linesearch procedures for Newton stepsizes, namely \grls{}, Armijo stepsize, and Greedy Newton stepsize (\greedy{}) \citep{cauchy1847methode, shea2024greedy}. 
Our theory presents convergence guarantees for \grls{} and \greedy{} with stepsizes limited to the interval $ [0,1]$. 
We go beyond this limitation and perform parameter linesearches over $\alpha \in \R_+$ instead.

Figures \ref{fig:logreg-linesearch}, \ref{fig:poly-linesearch} demonstrate that on logistic regression and polytope feasibility problems, linesearch procedures \grls{} and \greedy{} use almost indistinguishable stespsizes and converge faster than Armijo linesearch and fixed stepsize Newton. 
On the Rosenbrock function (\Cref{fig:rosenbrock}), \grls{} significantly outperforms all other linesearches procedures. 

\newcommand{\figsizeap}{\textwidth}
\newcommand{\figsize}{\textwidth}
\newcommand{\sfigsize}{0.49\textwidth}

\begin{figure}[h]
    \centering
\includegraphics[width=\linewidth]{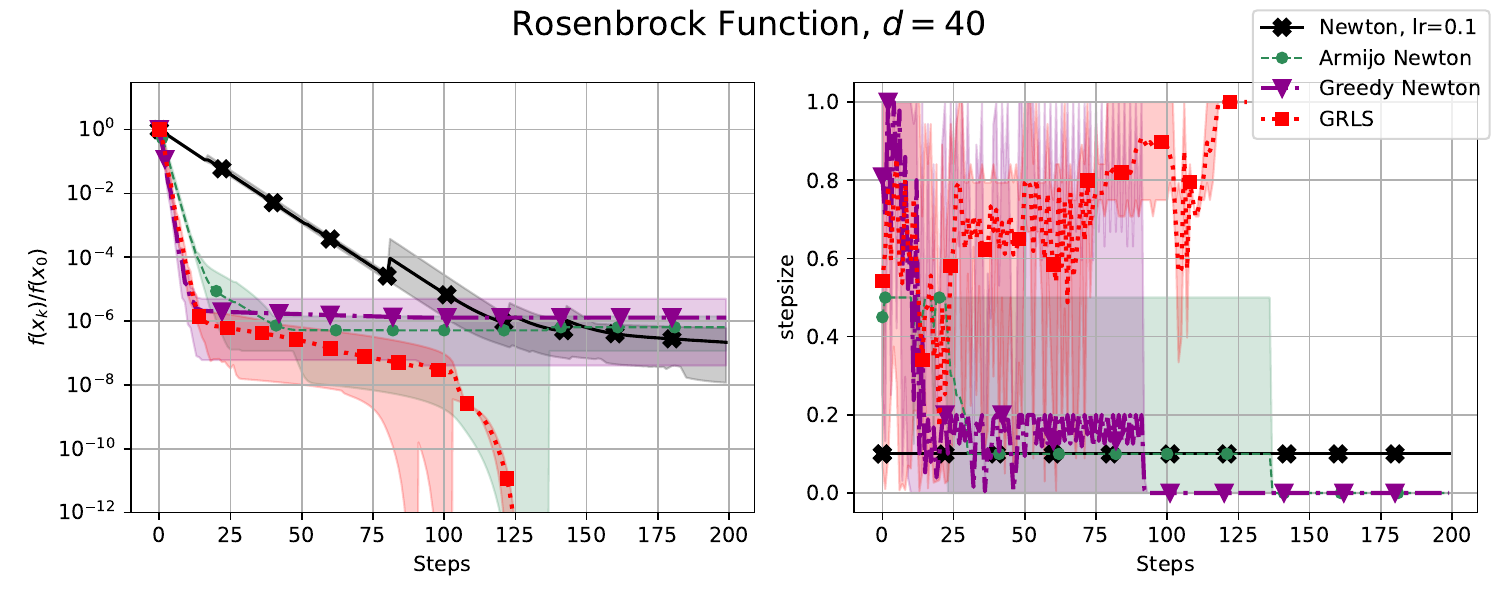}
    \caption{Performance of Newton method stepsize lineserch procedures on nonconvex \textbf{Rosenbrock function} \eqref{eq:rosenbrock}. 
    We plot mean $\pm$ standard deviation of $5$ random initializations. We crop stepsize standard deviation at $0$.}
    \label{fig:rosenbrock}
    \vspace{-5mm}
\end{figure}

\begin{figure*}
    \centering
    \begin{subfigure}{\figsize}
    \centering
        \includegraphics[width=\sfigsize]{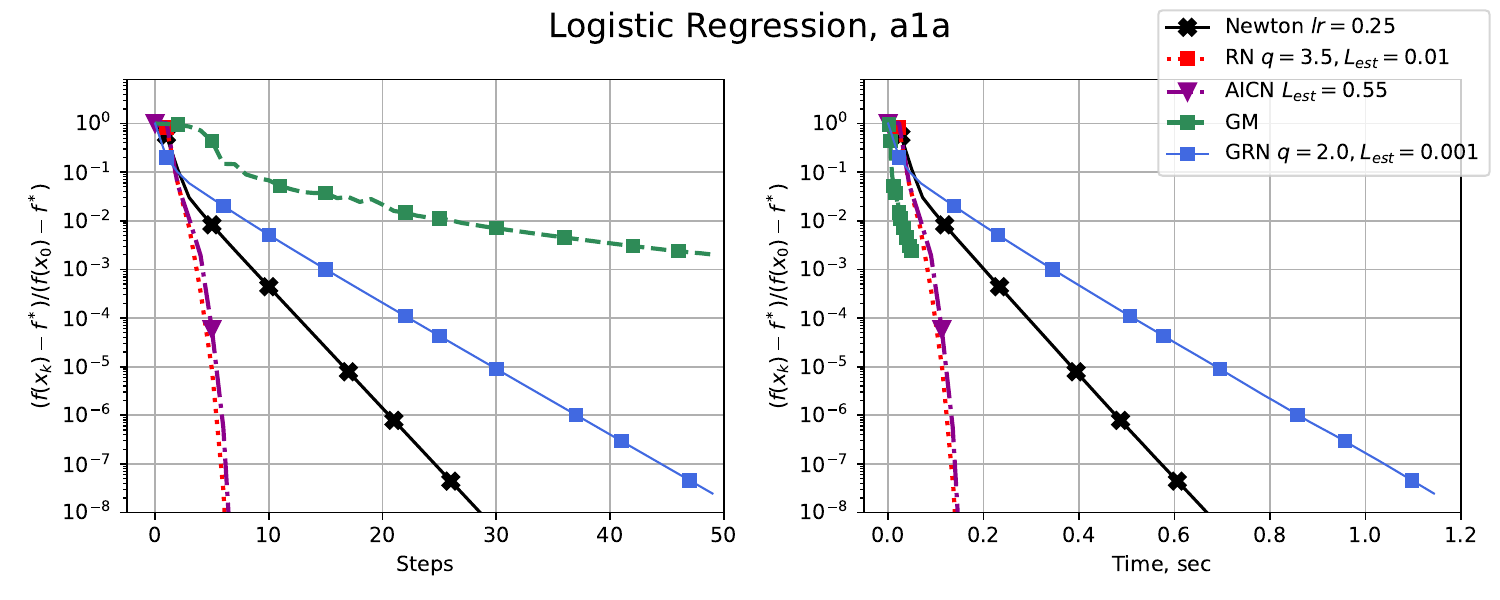}
        \hfill
        \includegraphics[width=\sfigsize]{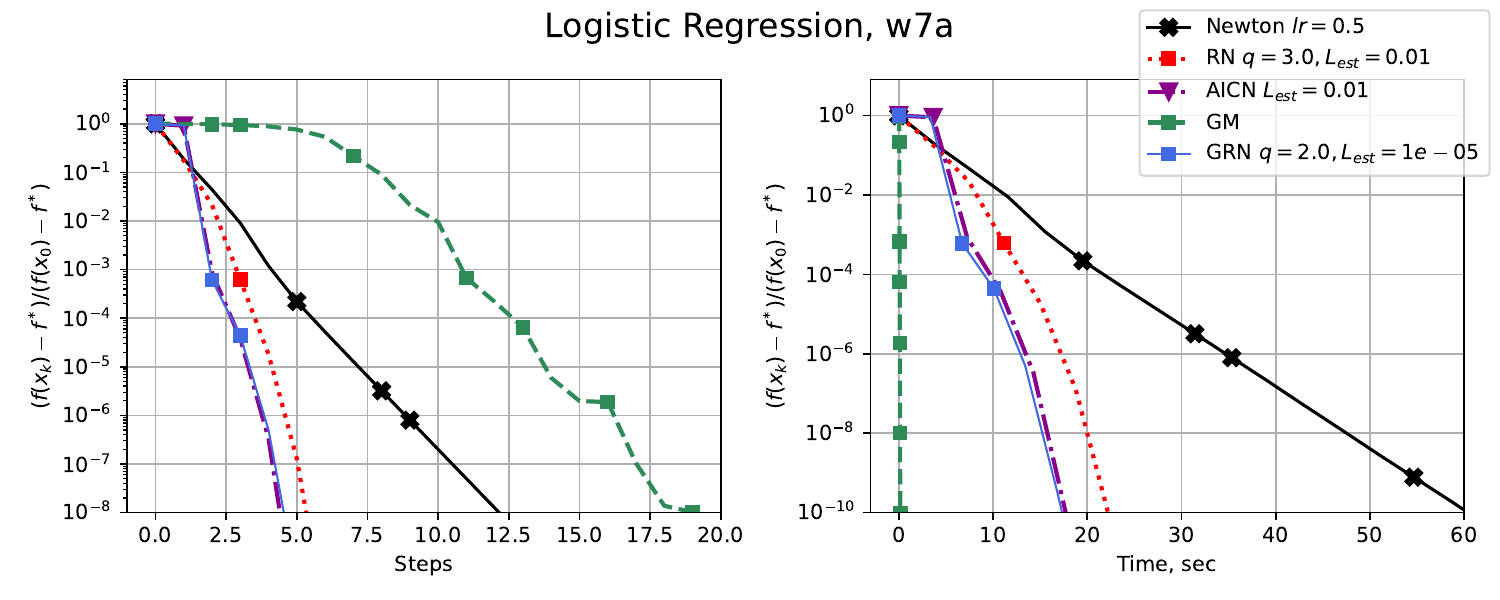}
        
        \includegraphics[width=\sfigsize]{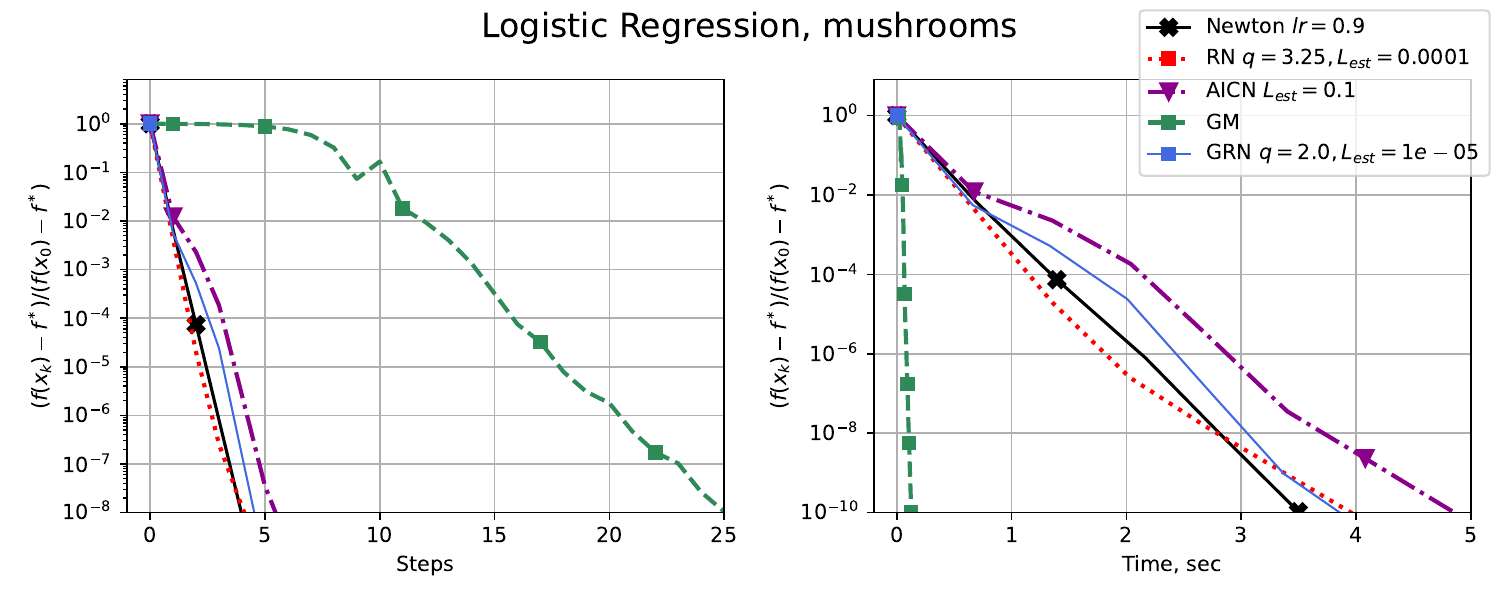}
        \hfill
        \includegraphics[width=\sfigsize]{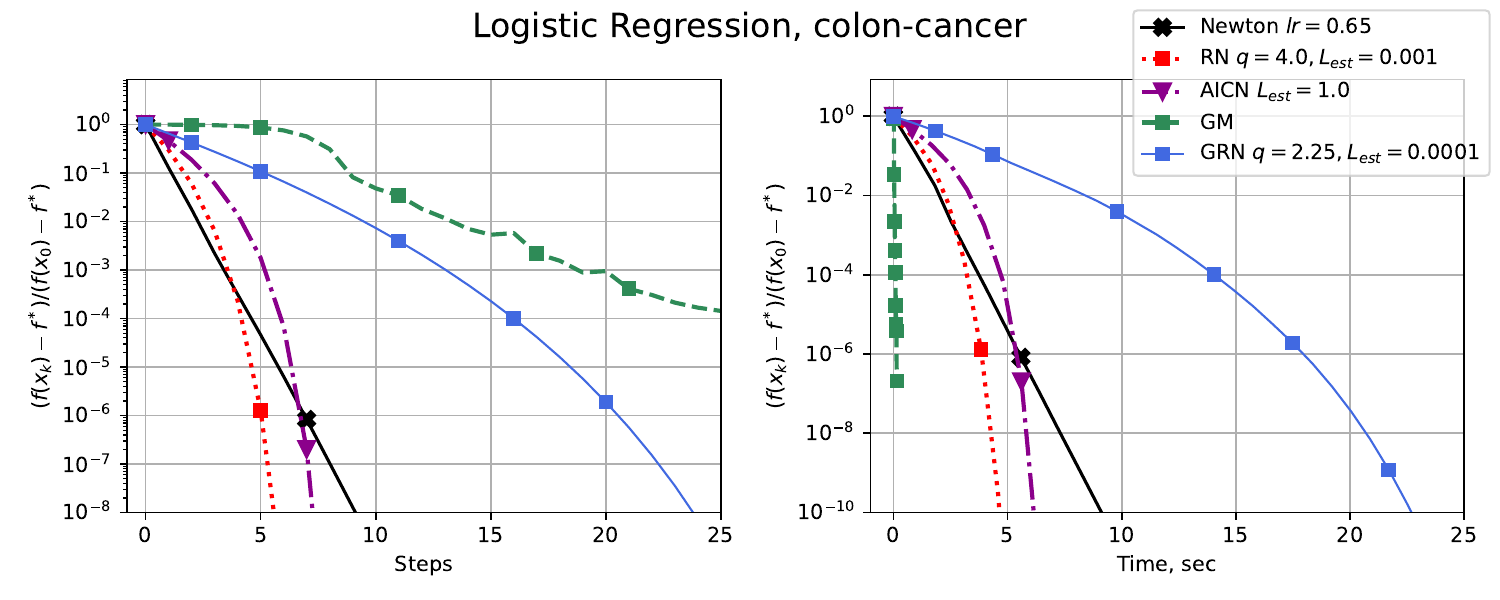}
    \caption{Performance of \rn{} compared to other higher-order methods \textit{without} any linesearch procedure.}
    \label{fig:logreg-rn}
    \end{subfigure}
    
    \begin{subfigure}{\figsize}
        \includegraphics[width=\sfigsize]{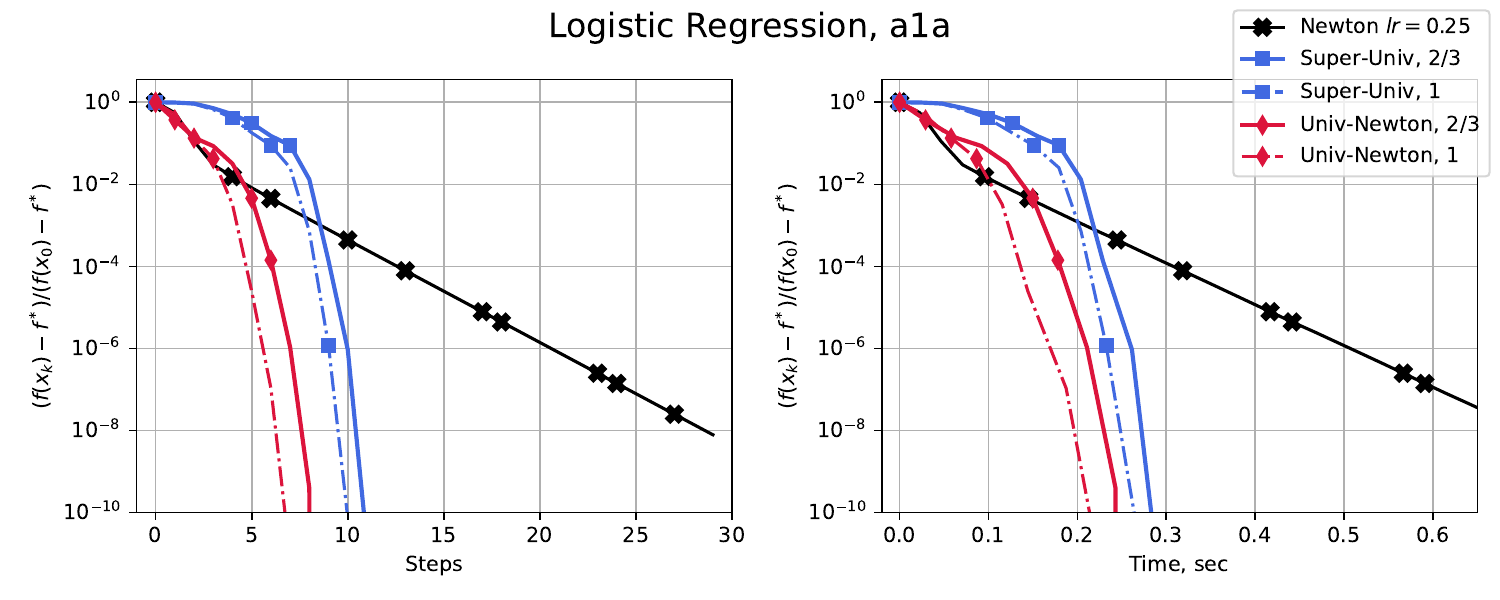}
        \hfill
        \includegraphics[width=\sfigsize]{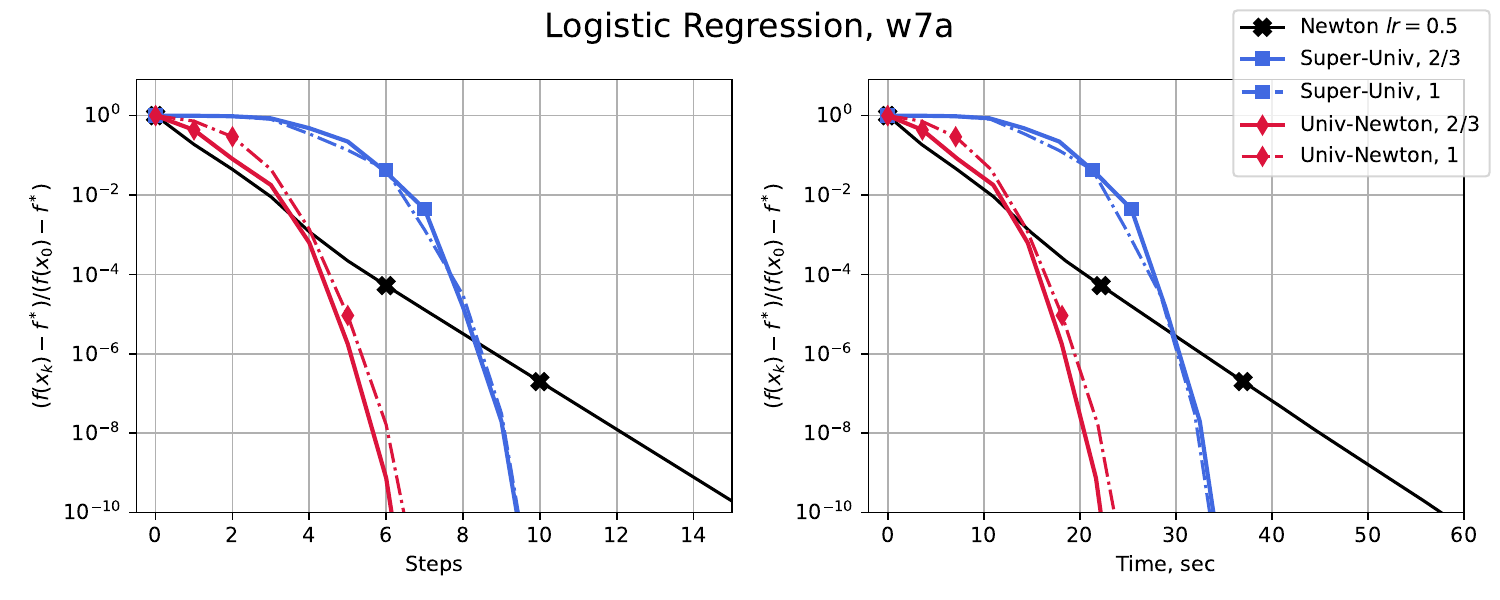}
        
        \includegraphics[width=\sfigsize]{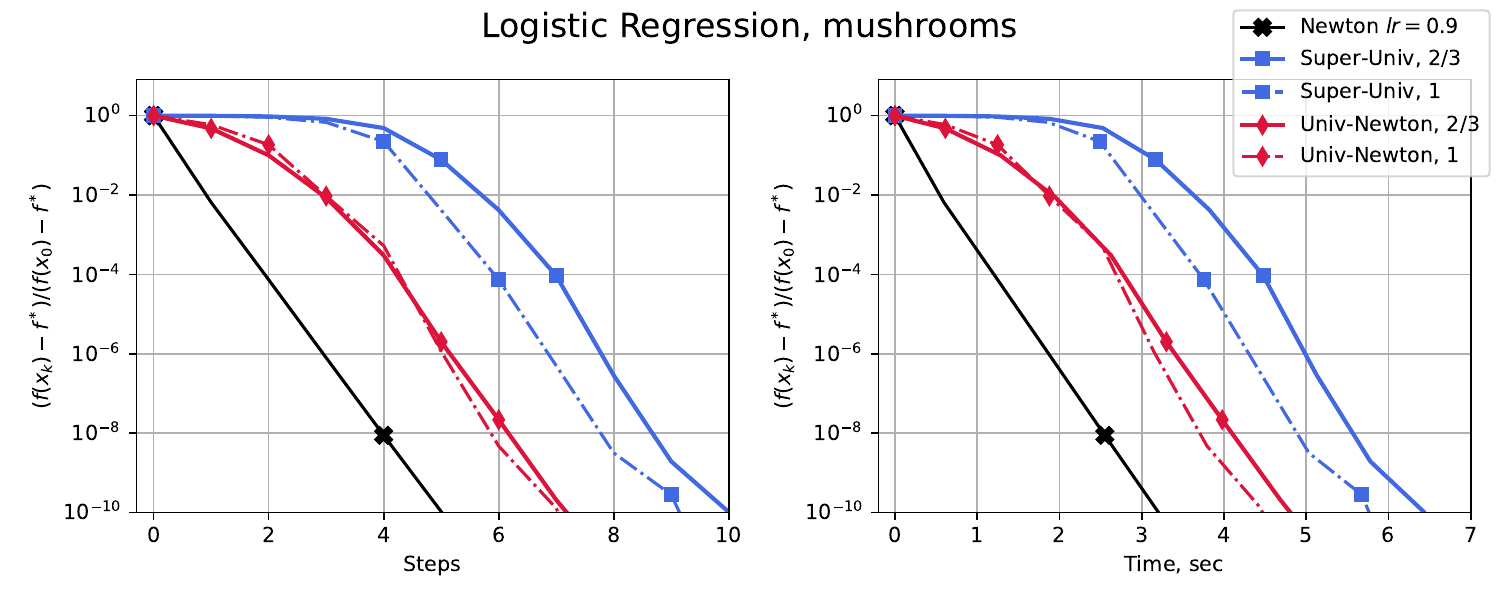}
        \hfill
        \includegraphics[width=\sfigsize]{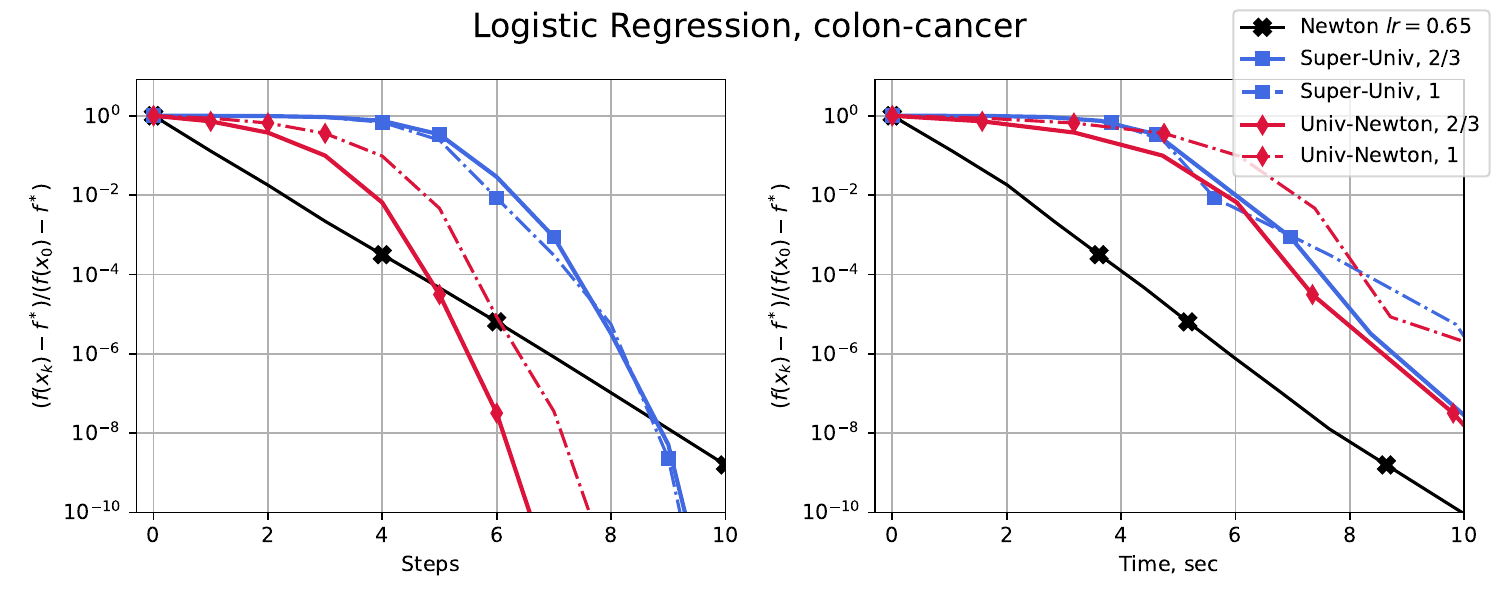}
    \caption{Performance of \un{}  compared to other higher-order regularization methods \textit{with} smoothness estimation procedures.}
    \label{fig:logreg-univ}
    \end{subfigure}

    \begin{subfigure}{\figsize}
    \centering
        \includegraphics[width=\sfigsize]{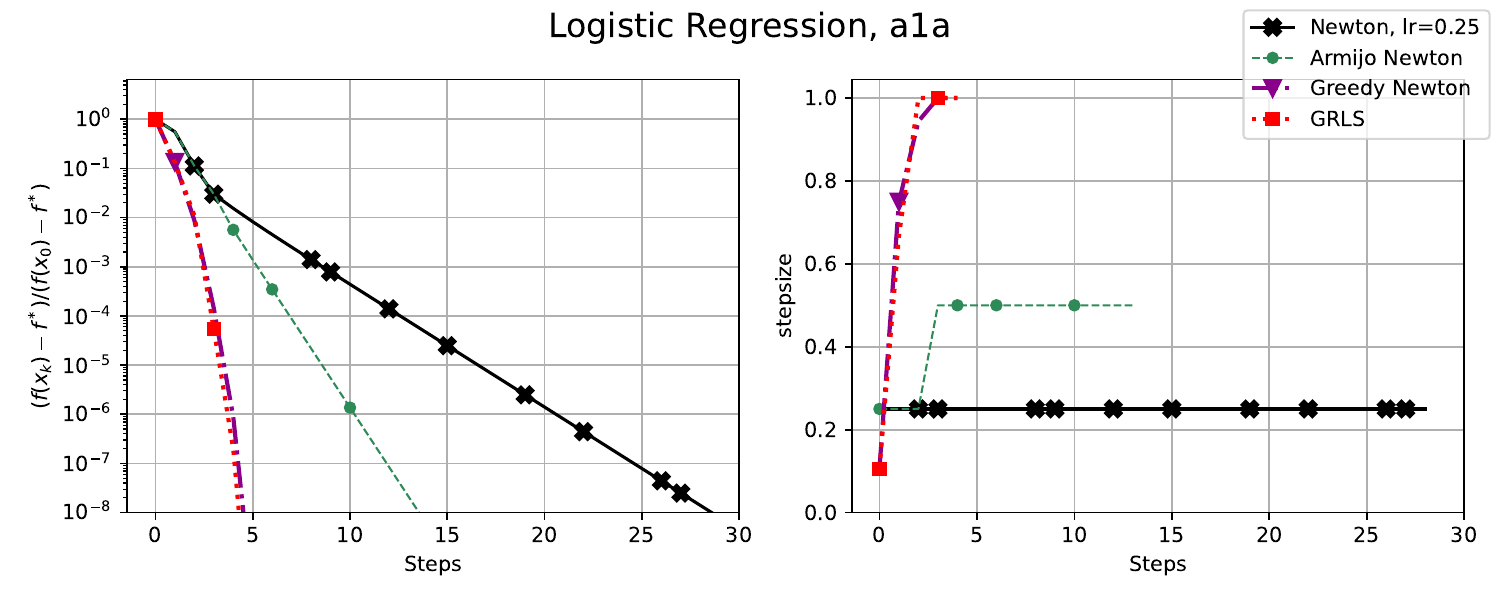}
        \hfill
        \includegraphics[width=\sfigsize]{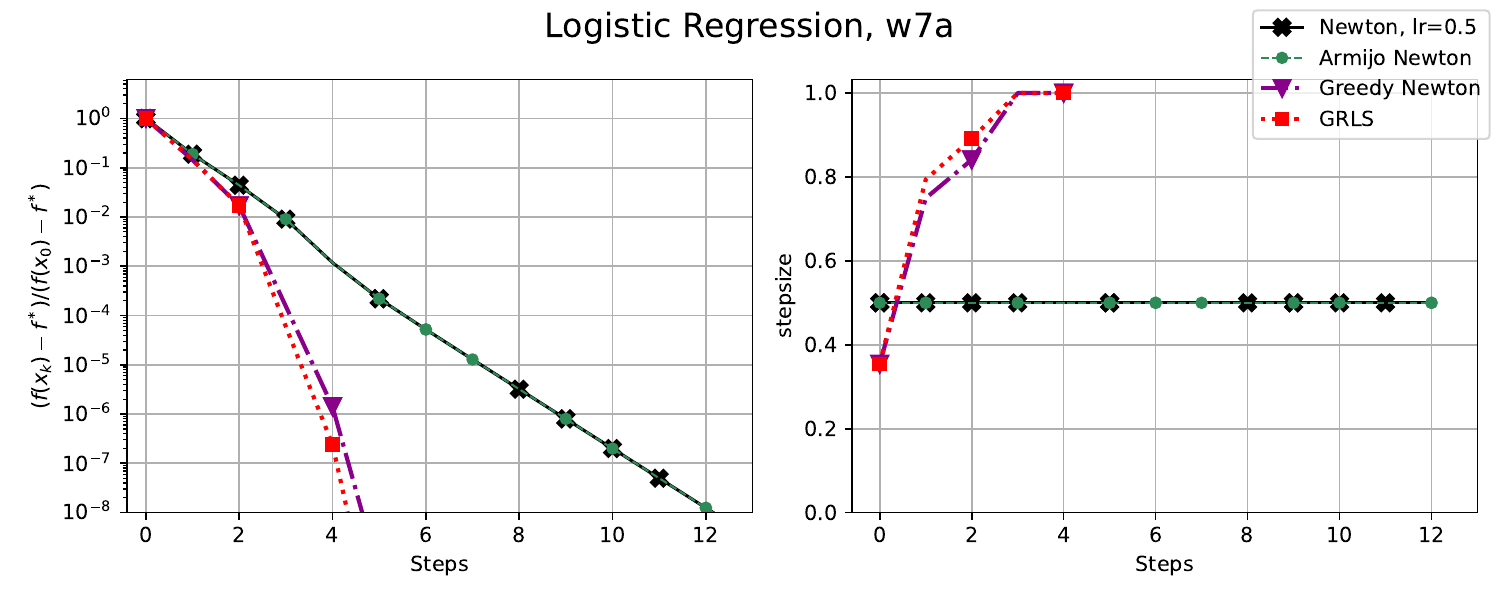}
        
        \includegraphics[width=\sfigsize]{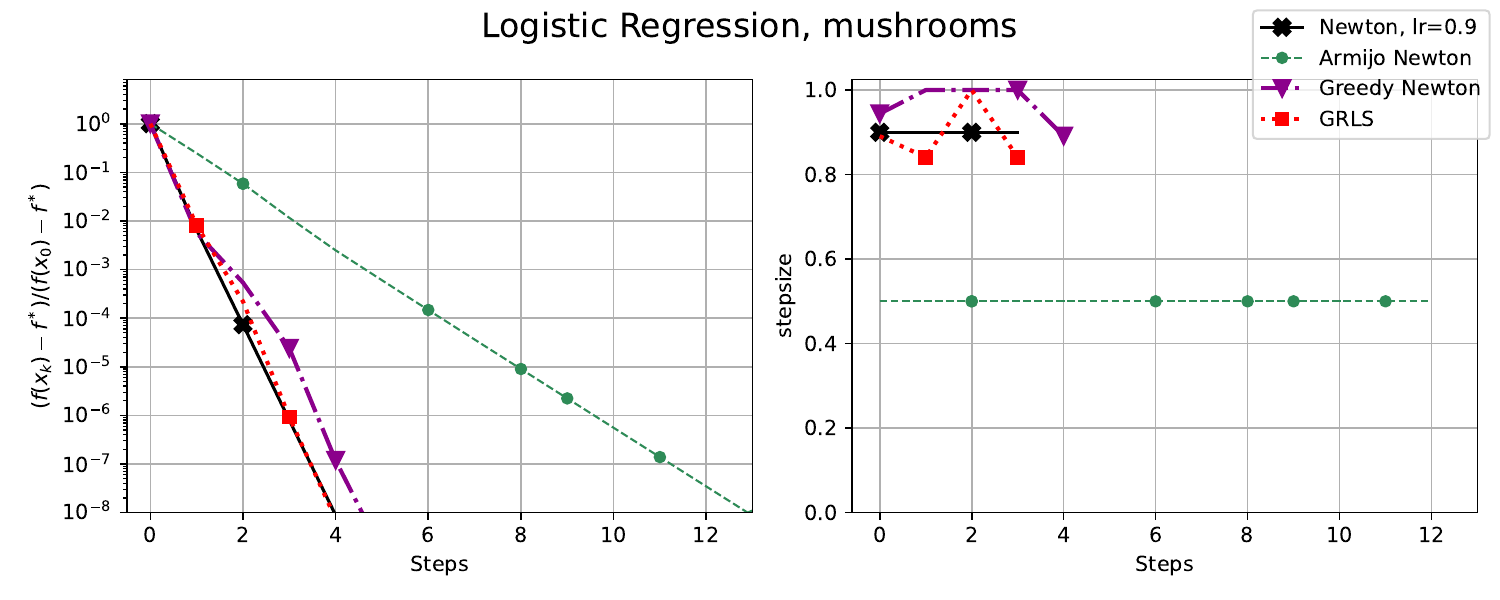}
        \hfill
        \includegraphics[width=\sfigsize]{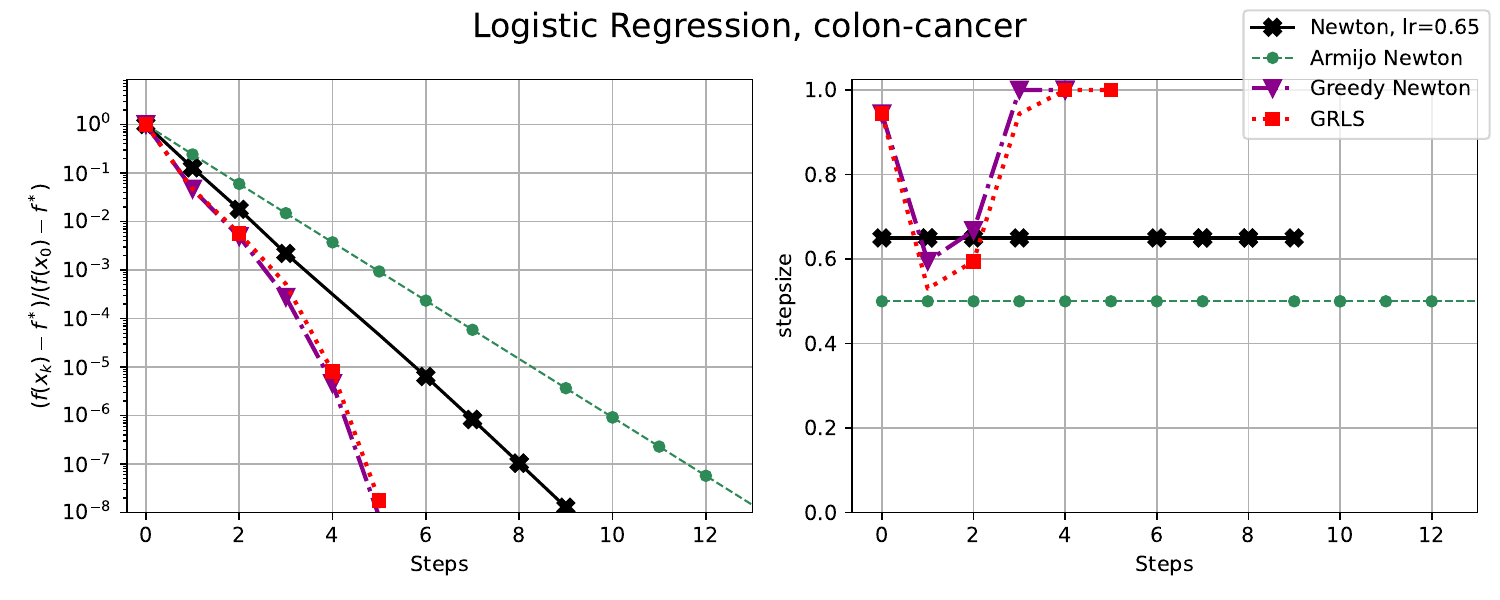}
    \caption{Performance of Linesearch \grls{} \eqref{eq:grls}  compared to other linesearch procedures.}
    \label{fig:logreg-linesearch}
    \end{subfigure}
    \caption{Binary classification \textbf{logistic regression} problem on \libsvm{} datasets.}
    \label{fig:logistic}
\end{figure*}

\begin{figure*}
\centering
    \begin{subfigure}{\figsize}
        \centering
        \includegraphics[width=\sfigsize]{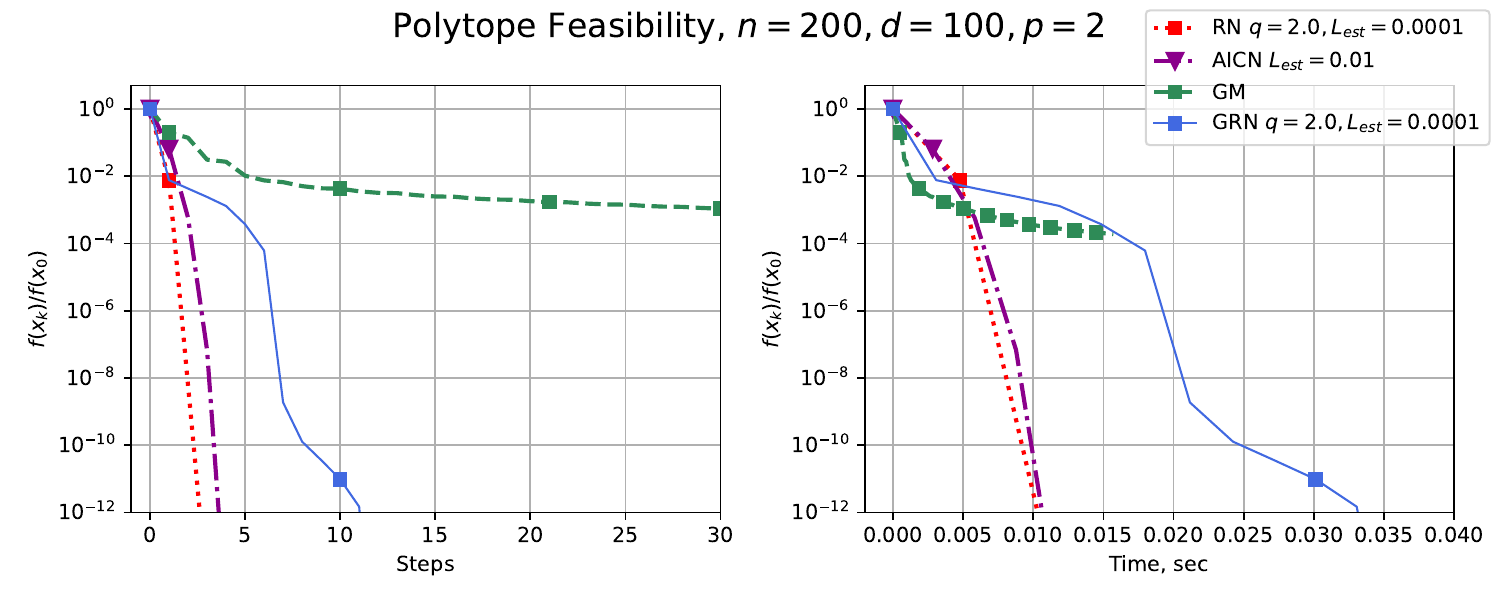}
        \hfill
        \includegraphics[width=\sfigsize]{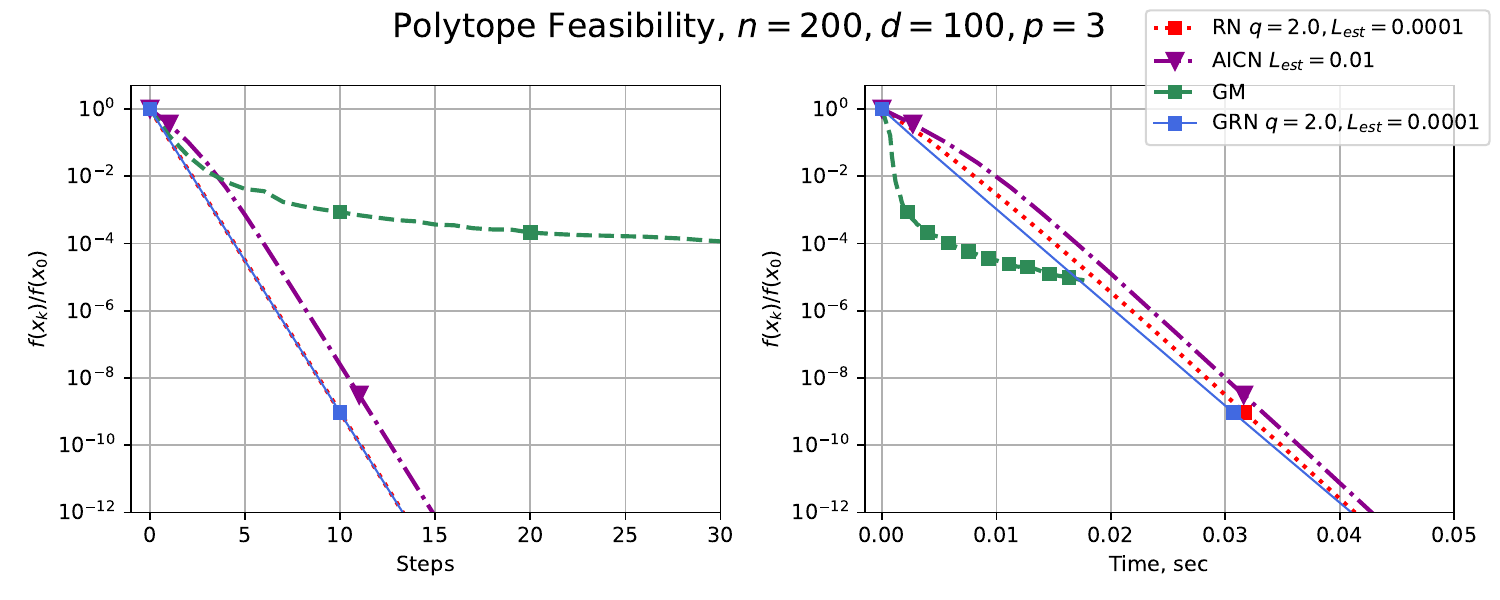}
        
        \includegraphics[width=\sfigsize]{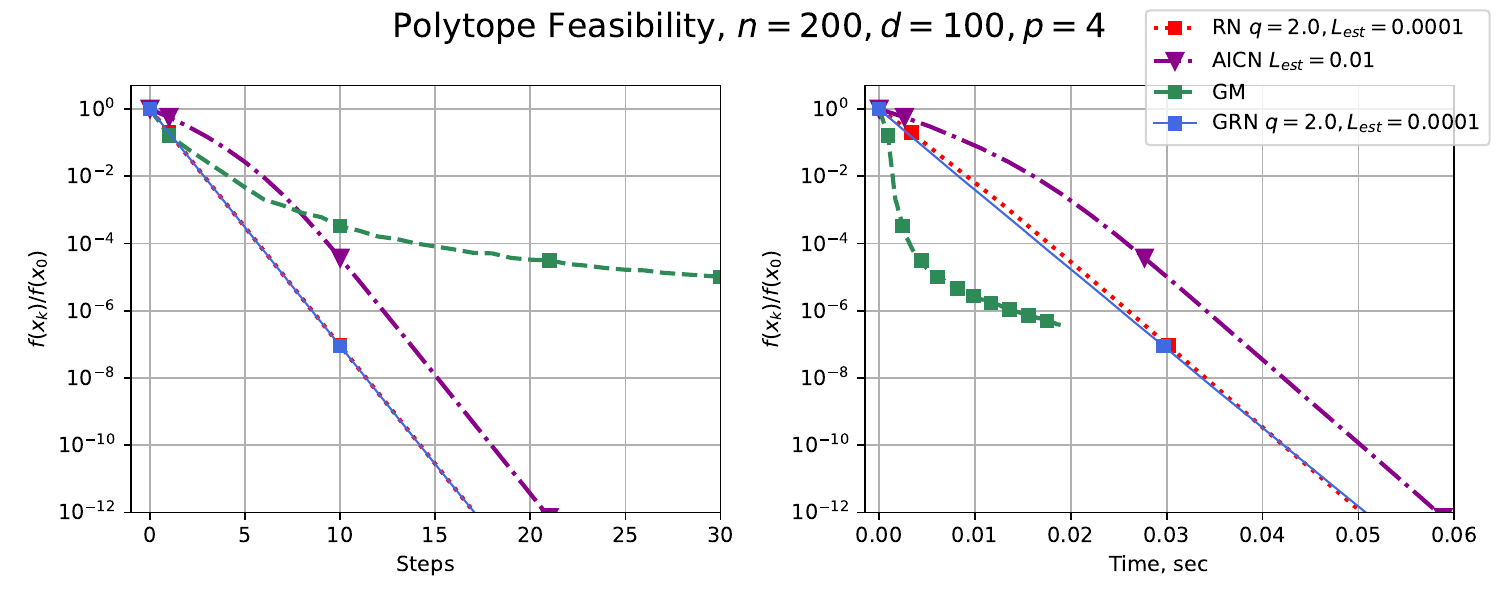}
        \hfill
        \includegraphics[width=\sfigsize]{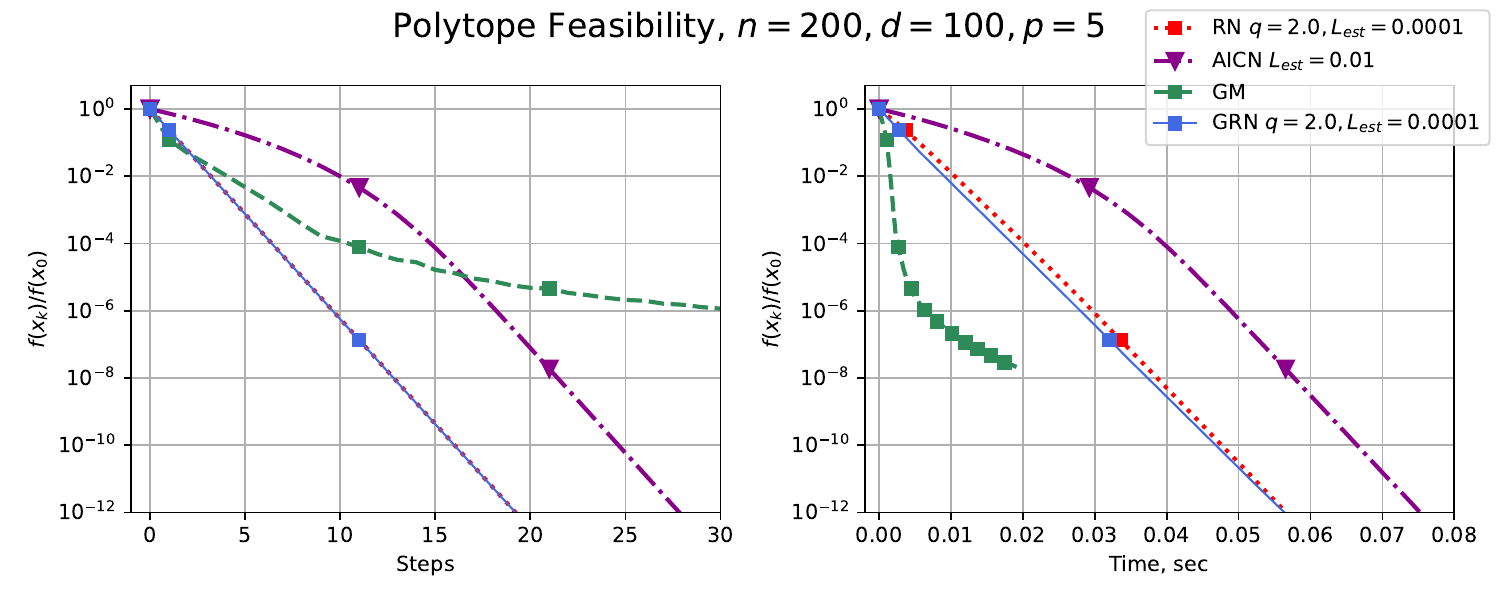}
        \caption{Performance of \rn{}  compared to other higher-order methods \textit{without} any linesearch procedure.}
        \label{fig:poly-rn}
    \end{subfigure}
    
    \begin{subfigure}{\figsize}
        \centering
        \includegraphics[width=\sfigsize]{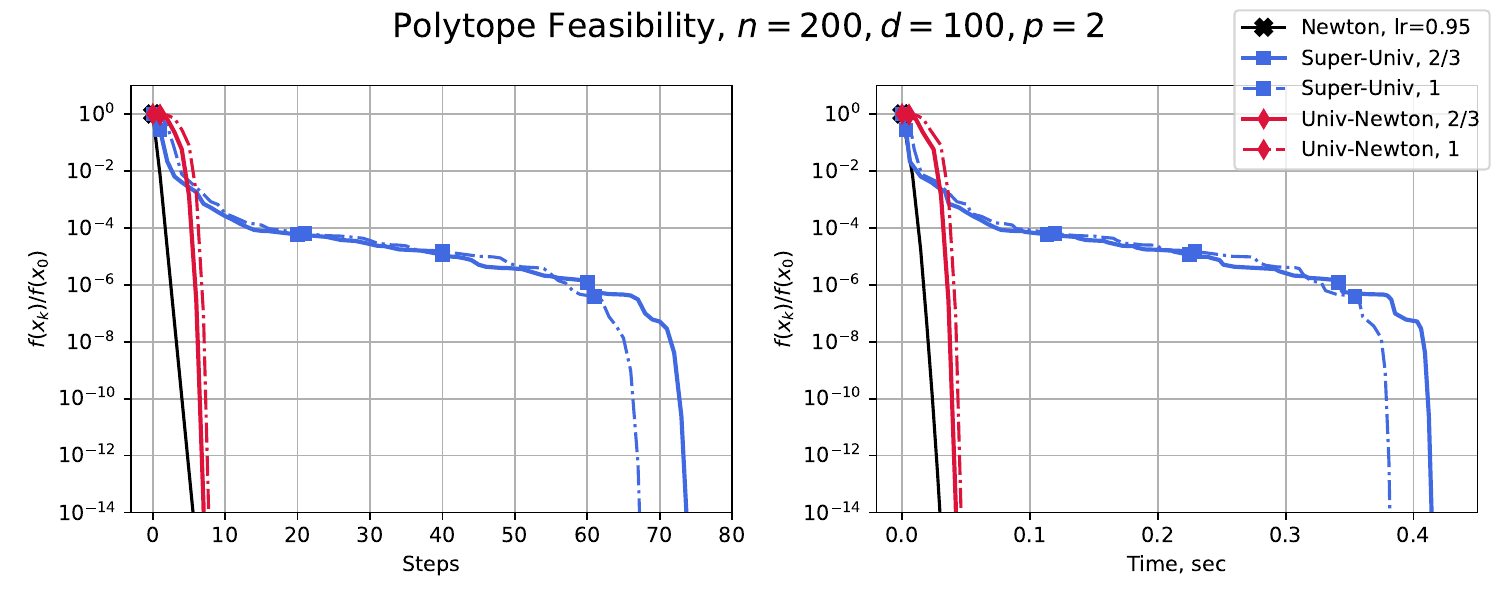}
        \hfill
        \includegraphics[width=\sfigsize]{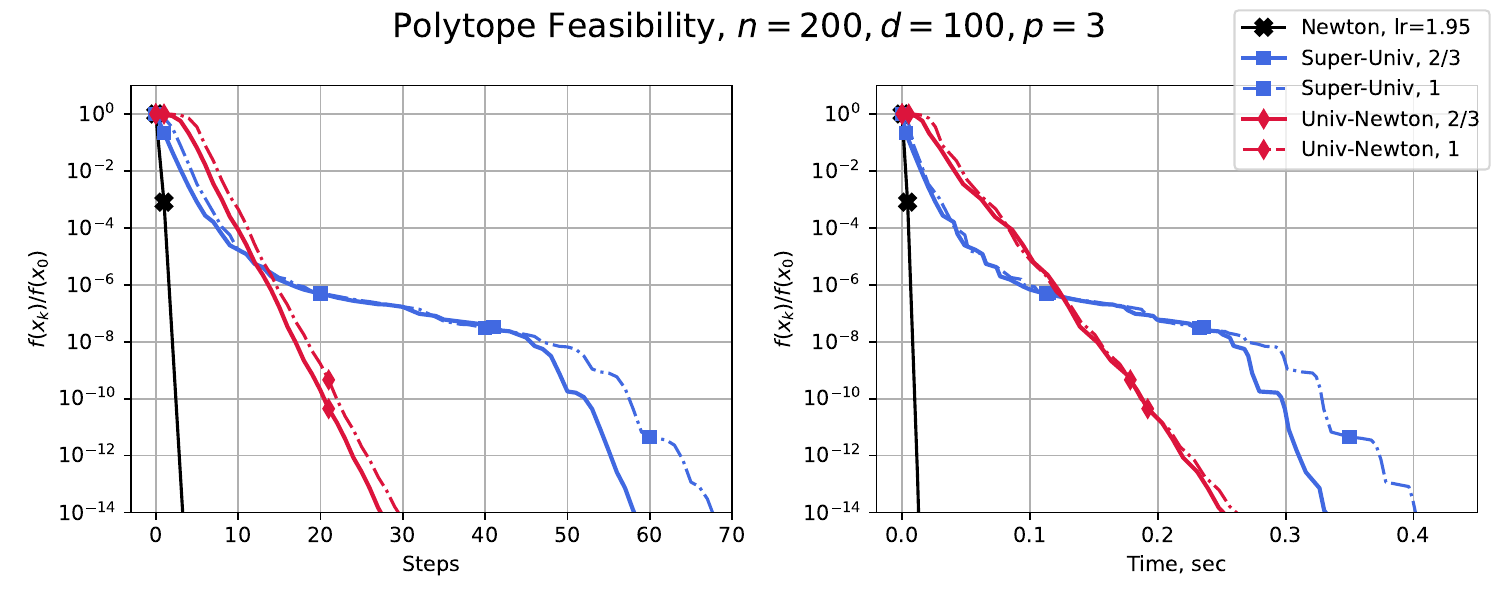}
        
        \includegraphics[width=\sfigsize]{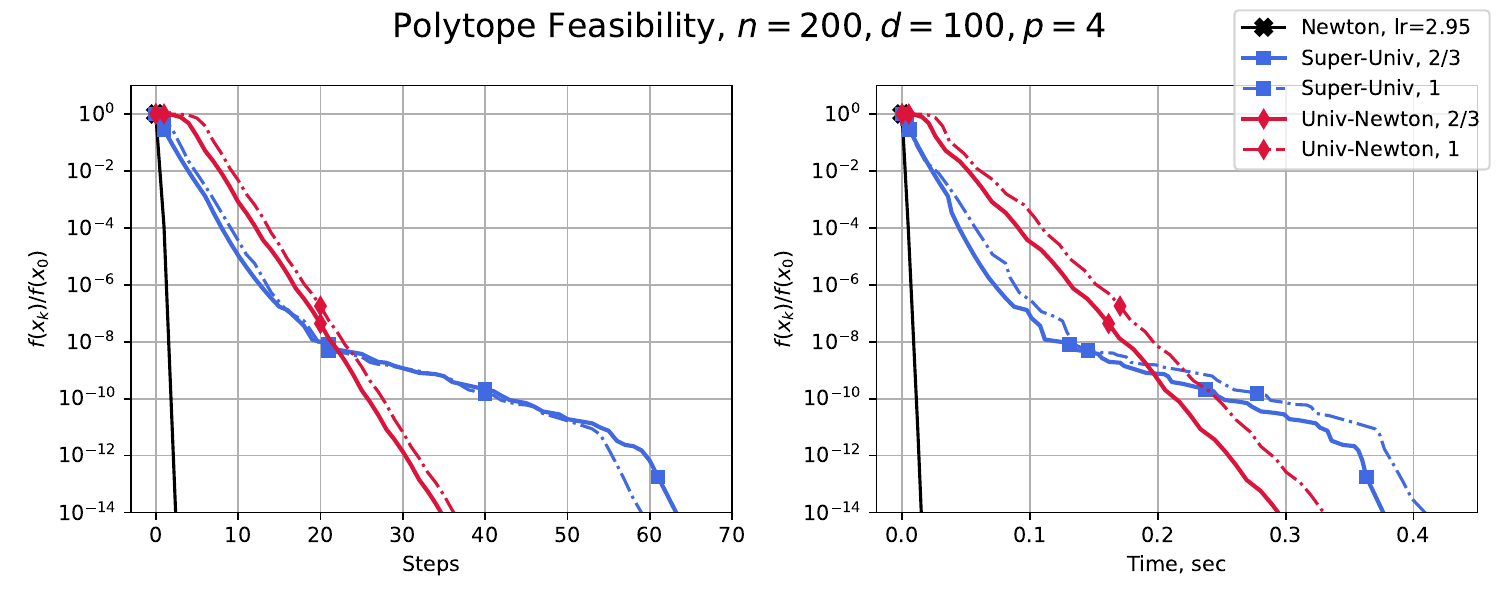}
        \hfill
        \includegraphics[width=\sfigsize]{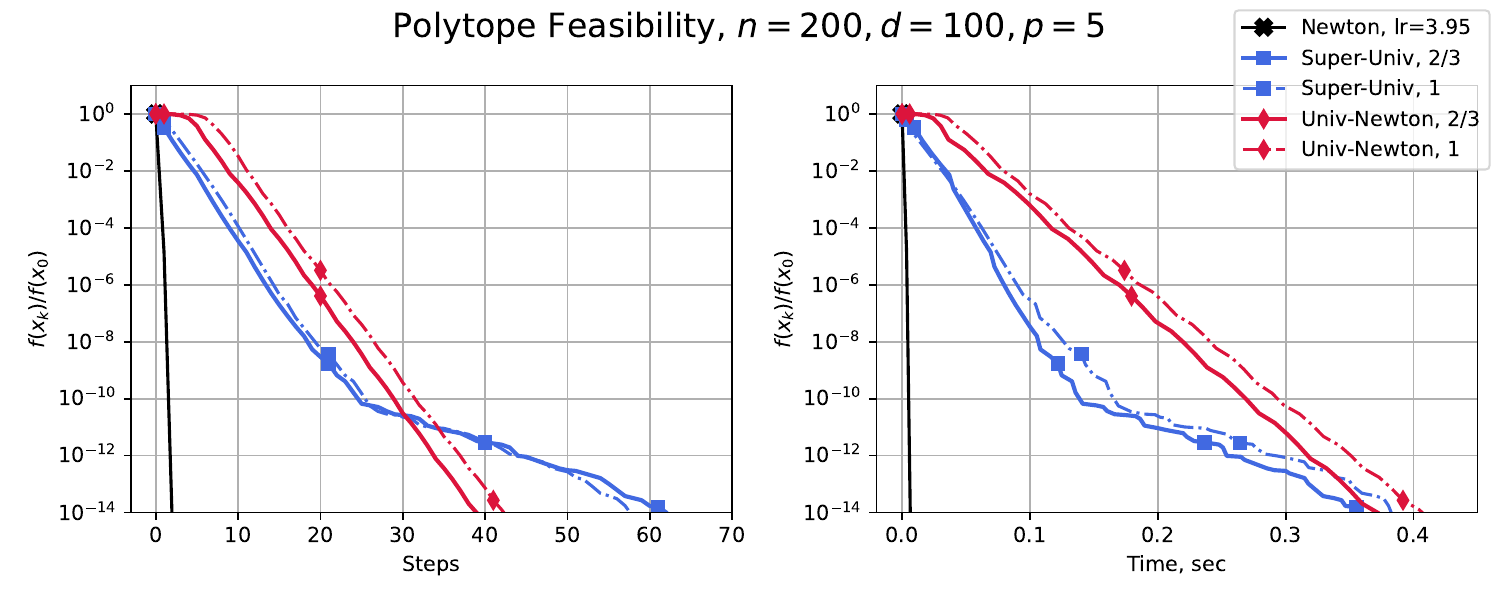}
        \caption{Performance of \un{}  compared to other higher-order regularization methods \textit{with} smoothness estimation procedures.}
        \label{fig:poly-univ}
    \end{subfigure}

    \begin{subfigure}{\figsize}
        \includegraphics[width=\sfigsize]{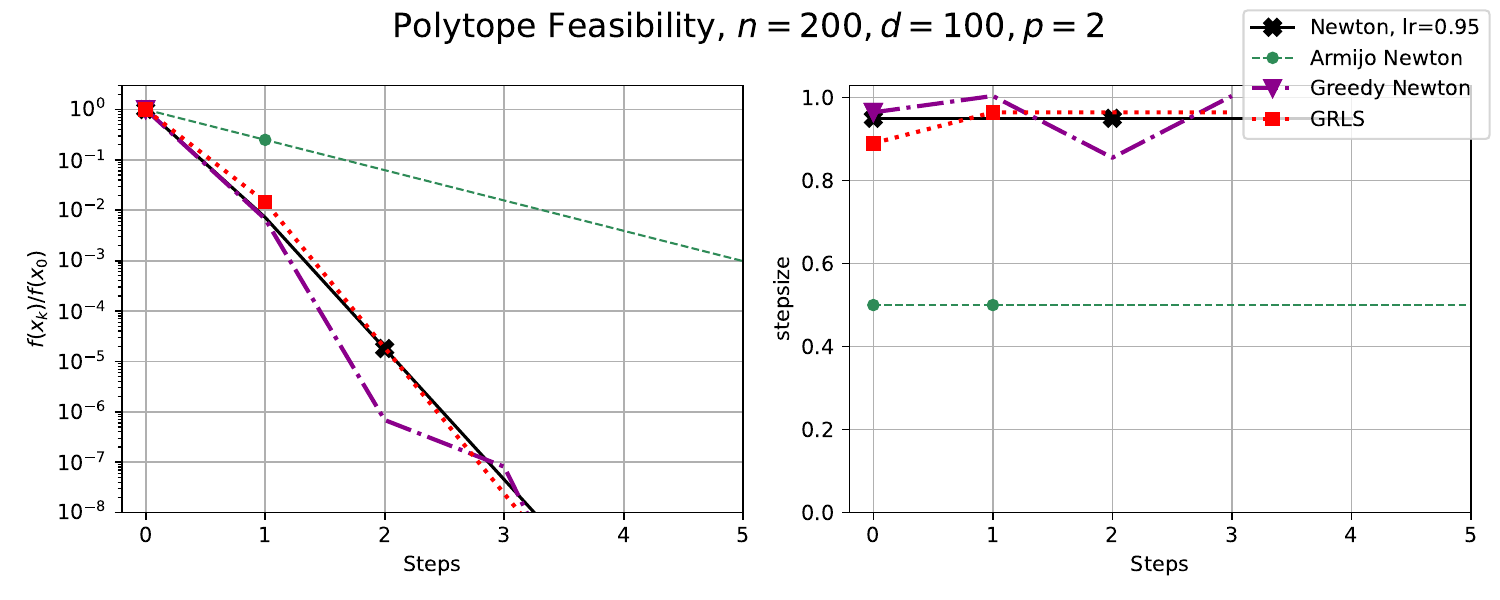}
        \hfill
        \includegraphics[width=\sfigsize]{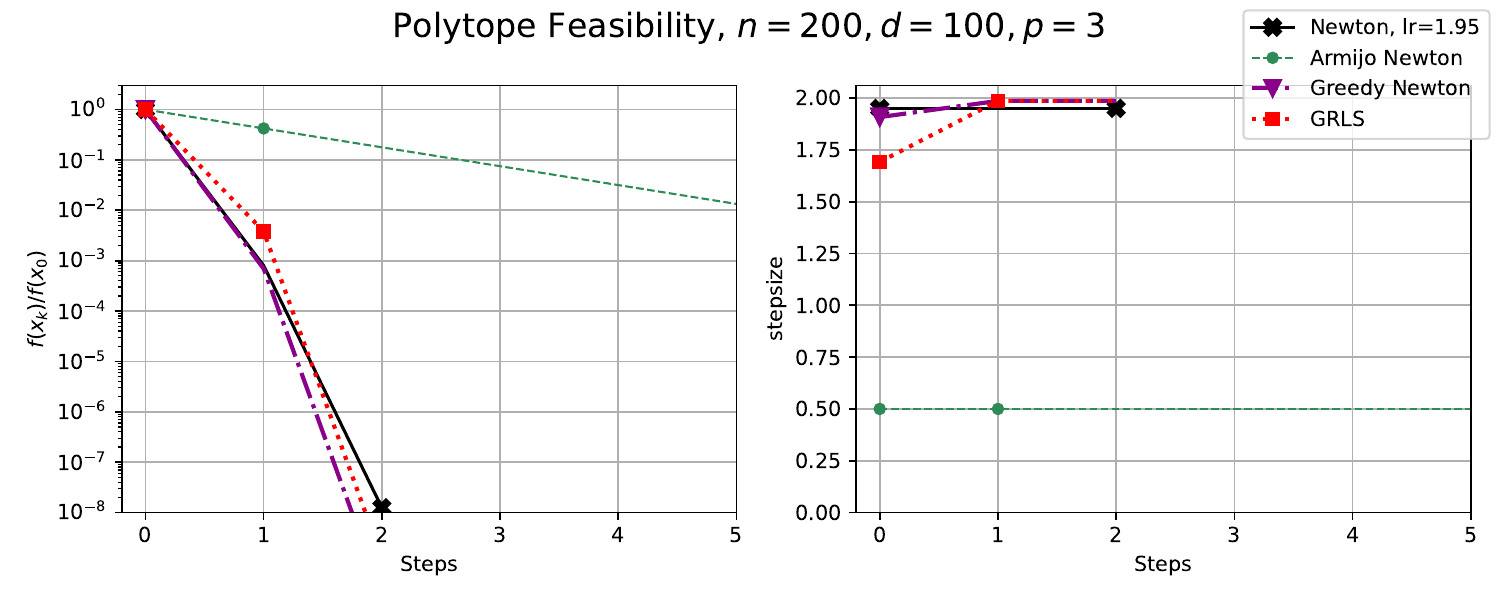}
        
        \includegraphics[width=\sfigsize]{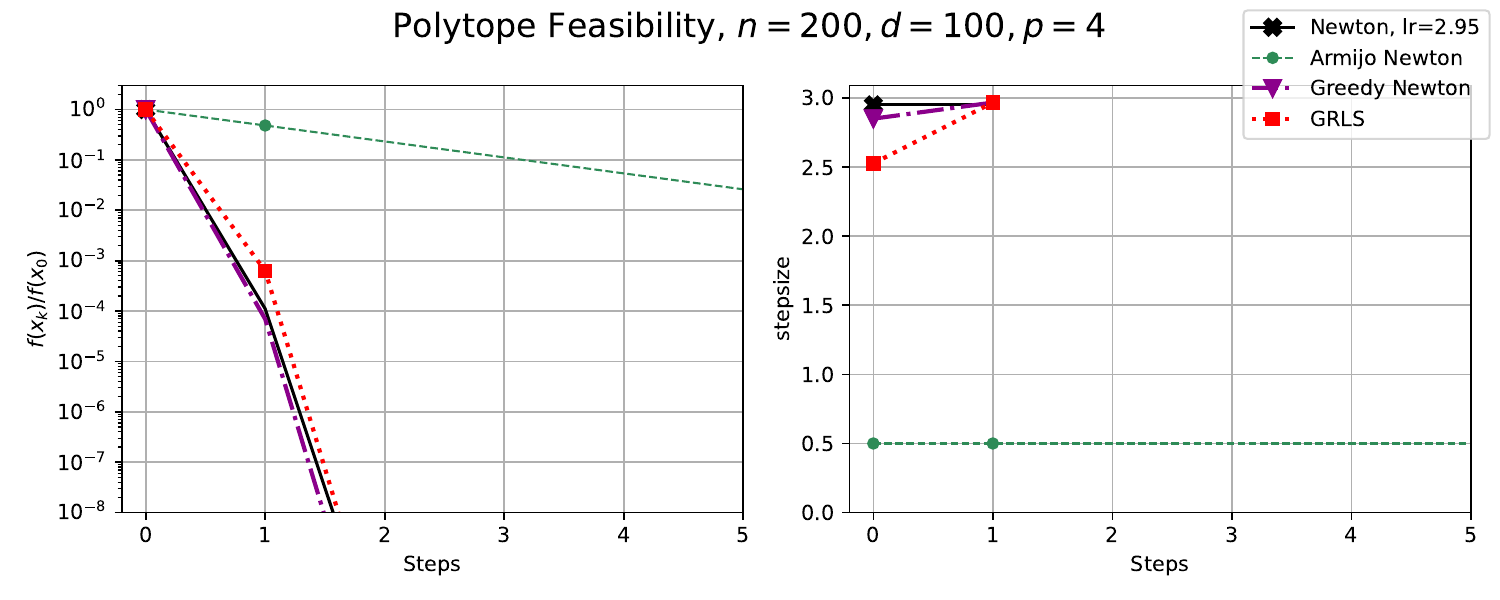}
        \hfill
        \includegraphics[width=\sfigsize]{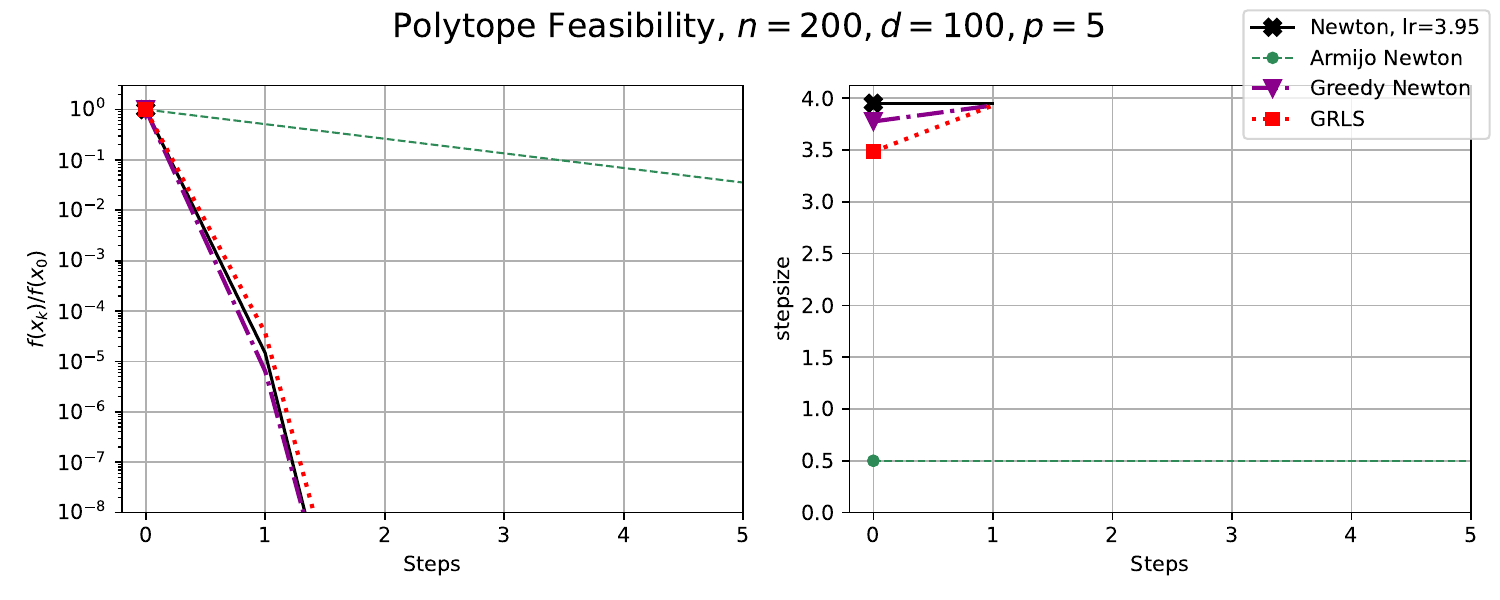}
    \caption{Performance of Linesearch \grls{} \eqref{eq:grls}  compared to other linesearch procedures.}
    \label{fig:poly-linesearch}
    \end{subfigure}
    \caption{\textbf{Polytope feasibility} problem \eqref{eq:polytope} on a synthetic datasets.}
    \label{fig:polytope}
\end{figure*}

    

\onecolumn
\bibliographystyle{plainnat}
\bibliography{UN}

\part*{Appendix} \label{p:appendix}
\appendix

\section{Technical details of experiments}

All hyperparameters were fine-tuned to achieve the best possible performance for both objectives and every dataset. All experiments were conducted on a workstation with specifications: AMD EPYC 7742 64-Core Processor with 32Gb of RAM. Source code is available at \url{https://anonymous.4open.science/r/root-newton-8D65}.

\subsection*{Extended comparison on Rosenbrock function}
Here we present an extended comparison of linesearch procedures on Rosenbrock function \eqref{eq:rosenbrock} (similar to \Cref{fig:rosenbrock}), with $10$ random initializations and the limit of $1000$ steps. We observe that none of the considered algorithms consistently converge to the exact solution for all of the random seeds, and that \grls{} performs better than the other linesearch methods.
\begin{figure}[h]
    \centering
\includegraphics[width=\linewidth]{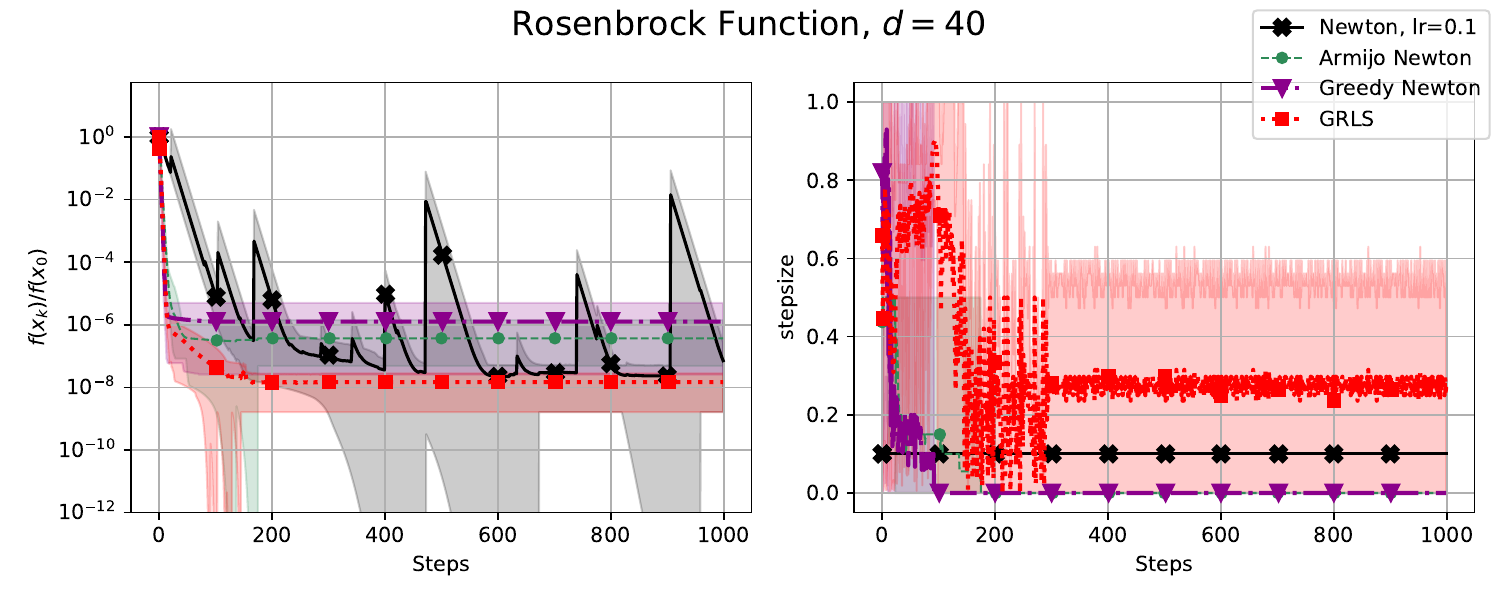}
    \caption{Performance of Newton method stepsize lineserch procedures on nonconvex \textbf{Rosenbrock function} \eqref{eq:rosenbrock}. 
    We plot mean $\pm$ standard deviation of $10$ random initializations. We crop stepsize standard deviation at $0$.}
    \label{fig:rosenbrock_long}
\end{figure}

\section{Fast convergence guarantees for Greedy Newton linesearch} \label{sec:greedy_newton}

If the inequality $\normMd {\g(y)} {x^k} \leq \csup \normMd{ \g(x^k) }{x^k}$ holds for constant $\csup \geq 0$, we have that for stepsizes in a range $[\alow, \aup]$ holds
\begin{align}
     \min_{\makecell{\alpha \in [\alow, \aup]\\ y=x-\alpha \ndirk}} \frac {f(y)-f(x^k)}{\normsMd{\g(x^k)}{x^k}}
    \leq & \csup^2 \cdot \min_{\makecell{\alpha \in [\alow, \aup]\\ y=x-\alpha \ndirk}} \frac {f(y)-f(x^k)}{\normsMd{\g(y)}{x^k}},
\end{align}
proving that Greedy Newton minimizes the target metric of \grls{} up to a constant $\times \csup^2$. If we denote $\hat \cc$ constant with which \grls{} satisfies \Cref{le:one_step_to_rate}, then Greedy Newton satisfies \Cref{le:one_step_to_rate} with constant $\hat \cc \csup^2$ and guarantee convergence similar to \Cref{col:linesearch}.


Now we are going to discuss how constant $\csup$ can be found in different scenarios.

\begin{remark}[General $\Lq$-Hölder continuous functions]
To find $\csup$ we note that \Cref{th:holder_un} shows that stepsize $\alunk \eqdef \frac {1-\prk} \prk \geq \lr 9\Lq \rr^{\frac 1 {q-1}}  \normM {\g(x^k)}{x^k} ^{*\frac {q-2}{q-1}}$ for $\Lq$-Hölder continuous function implies
\begin{align*}
    \frac 1{2(1-\prk)} \normsMd { \g(y)}{x^k}, 
    \leq \la \g(y), \ls \h(x^k) \rs ^{-1}\g(x^{k}) \ra 
    \leq \normMd {\g(y)} {x^k} \normMd{\g(x^{k})}{x^k},
\end{align*}
which after rearranging yields $\normMd { \g(y)}{x^k} \leq 2(1-\prk)\normMd{\g(x^{k})}{x^k}$. 
Therefore if 
\begin{align}
\alpha 
\leq \frac 1 {1+ \lr 9\Lq \rr^{\frac 1 {q-1}}  \normM {\g(x^k)}{x^k} ^{*\frac {q-2}{q-1}}}
\end{align}
or equivalently 
\begin{equation} \label{eq:aup_limit}
\aup 
\leq \lr 1+ \lr 9\Lq \rr^{\frac 1 {q-1}}  \normM {\g(x^k)}{x^k} ^{*\frac {q-2}{q-1}} \rr^{-1}
\leq \lr 1+ \sup_{q\in [2,4]}\lr 9\Lq \rr^{\frac 1 {q-1}}  \normM {\g(x^0)}{x^0} ^{*\frac {q-2}{q-1}} \rr^{-1}.
\end{equation}
In such case, $\csup$ can be set as $\csup=2(1-\alow)$.\\

Note that \eqref{eq:aup_limit} is satisfied by smaller stepsizes, which damped Newton methods use globally until they converge to the neighborhood of the solution.
\end{remark}

\begin{remark}[Hölder continuity of Hessians]
For $\Lsupp 2$-Holder, Lemma 8 yields
\begin{align}
    \normMd {\g(y)} {x^k}
        &\leq \lr |1-\alpha| + \frac {\Lsupp 2}{1+ \nu} \alpha^{1+\nu} \normM {\g(x^k)}{x^k} ^{*\nu} \rr \normMd{ \g(x^k) }{x^k},
\end{align}
ensuring that without any limitation on $\aup$
\begin{align}
\cupper 
&\eqdef \sup_{\alpha \in [\alow, \aup]} |1-\alpha|+\frac {\Lsupp 2}{1+ \nu} \alpha^{1+\nu} \normM {\g(x^k)}{x^k} ^{*\nu}\\
&=\max_{\alpha \in \lc \alow, \aup, 1\rc } |1-\alpha|+\frac {\Lsupp 2}{1+ \nu} \alpha^{1+\nu} \normM {\g(x^k)}{x^k} ^{*\nu}.
\end{align}
For $\alow \leftarrow 0, \aup \leftarrow 1$, we can set 
\begin{align}
    \csup
    &= \max \lc 1, \frac {\Lsupp 2}{1+ \nu} \normM {\g(x^k)}{x^k} ^{*\nu}\rc
    \leq \max \lc 1, \frac {\Lsupp 2}{1+ \nu} \normM {\g(x^0)}{x^0} ^{*\nu} \rc.
\end{align}
\end{remark}

\begin{remark}[$\Lsuppnu 20$-Hölder continuity]
For $\Lsuppnu 20$-Hölder functions with $\Lsuppnu 20 \geq 1$, constant $\csup$ simplifies to $\csup \eqdef \aup \frac { \Lsuppnu 20}{2} + |1-\aup|$, because
\begin{equation}
\begin{cases}
\aup \lr\frac { \Lsuppnu 20}{2} - 1\rr + 1 \geq \alpha \lr\frac { \Lsuppnu 20}{2} - 1\rr + 1 \geq \frac 12, & \text{if } \alpha \leq 1,\\
\aup \lr\frac { \Lsuppnu 20}{2} + 1\rr - 1 \geq \alpha \lr\frac { \Lsuppnu 20}{2} + 1\rr - 1 \geq \frac { \Lsuppnu 20}{2}, & \text{if } \alpha \geq 1.
\end{cases}
\end{equation}
\end{remark}

\section{Connection between stepsizes and regularization} \label{sec:stepsize_is_regularization}
We show connections of particular stepsizes to regularized Newton methods. For fixed $\regc>0, \expo \geq 0$ define regularized model as
\begin{align}
    \modelun{x} \eqdef \argmin_{y \in \R^d} \left \{ f(x)+ \la \g(x), y-x \ra + \frac 12 \normsM {y-x} x + \frac \regc {2+\expo} \normM {y-x} x ^{2+\expo} \right\}.
\end{align}

We can define optimization algorithm \rn{} as
\begin{equation}
x^{k+1} \eqdef \modelun {x^k}
\end{equation}

By first-order optimality condition, solution of model $h^*\eqdef \modelun {x} - x$ satisfy
\begin{gather}
    \left( 1+ \regc \normM {h^*} x ^{\expo} \right) \left[ \h(x) \right] h^* = -\g(x),\\
    h^* = -\underbrace{\left( 1+ \regc \normM {h^*} x ^{\expo} \right)^{-1}}_{\eqdef \pr>0} \left[ \h(x) \right]^{-1} \g(x). 
\end{gather}
Now iterates of \rn{} are in the direction of Newton method (for any $\regc$ and $\expo$) and we can write
\begin{gather}
    h^*= -\pr \left[ \h(x) \right]^{-1} \g(x),\\
    \left[ \h(x) \right] h^*= -\pr \g(x),\\
    \normM {h^*} x = \pr \normMd {\g(x)} x.
\end{gather}
Substituting $\ls\h(x) \rs h^*$ back to the first-order optimality conditions we get \begin{equation}
    0 = \g(x) \left( 1- \pr - \pr^{1+\expo} \regc \normM {\g(x)} x ^{* \expo} \right).
\end{equation}
Thus, $\pr$ defined as a root of the polynomial
\begin{equation}
    P[\alpha] \eqdef 1- \alpha - \alpha^{1+\expo} \regc \normM {\g(x)} x ^{* \expo}
\end{equation}
satisfies first-order optimality condition. Note that $P[0] >0$ and $P[1]\leq 0$, hence $P$ has root on interval $(0,1]$. This will be the stepsize of our algorithm. Also note that P is monotone on $\R_+$,
\begin{equation}
    P'[\alpha] = - 1 - (1+\expo) \alpha^{\expo} \regc \normM {\g(x)} x ^{* \expo} <0,
\end{equation}
and consequently, the positive root of $P$ is unique.

\section{Extra smoothness relations} \label{sec:smoothness_relations}
Let $\gamma \in [0,1]$.
From Hölders continuity, triangle inequality and definition of $\Lsup$,
\begin{align}
    \normM{\nabla^3 f(x)[y-x]} {op}
    & \leq \normM {\h(x) - \h(y)} {op} + \frac {\Lsupp 3 }{1+\nu} \normM{y-x} x ^{1+ \nu}\\
    & \leq \Lsuppnu 2 \gamma \normM {x-y} x ^ \gamma + \frac {\Lsupp 3 }{1+\nu} \normM{y-x} x ^{1+ \nu}
\end{align}
For $y \leftarrow x + \tau h $, where $\normM h x=1, \tau >0$, we can continue
\begin{align}
    \normM{\nabla^3 f(x)} {op}
    & \leq \frac {\Lsuppnu 2 \gamma} {\tau^{1-\gamma}} + \frac {\Lsupp 3 }{1+\nu} \tau^\nu,\\
    & \leq \frac {2+ \nu}{1+\nu} [\Lsuppnu 2 \gamma] ^{\frac {\nu}{1+\nu-\gamma}} { \tau^{1-\gamma}} [\Lsupp 3 ]^{\frac 1{1+\nu-\gamma}}, \qquad // \text{ by } \tau \leftarrow \left[ \frac {\Lsuppnu 2 \gamma}{\Lsupp 3} \right]^{\frac 1{1+\nu-\gamma}}\\
    & \leq \frac 32 \sqrt{\Lsuppnu 2 0 \Lsuppnu 3 1}, \hspace{4cm}  // \text{ by } \gamma \leftarrow 0, \nu \leftarrow 1
\end{align}
and we can summarize
\begin{align}
    \Lsuppnu 3 0 
    &= \sup_{x \neq y} \normM{\nabla^3 f(x)-\nabla ^3 f(y) }{op} 
    \leq \sup_{x \neq y} \left( \normM{\nabla^3 f(x) }{op} + \normM{\nabla ^3 f(y) }{op} \right)
    = 2\sup_{x} \normM{\nabla^3 f(x) }{op}\\
    &\leq \begin{cases}
        2 \Lsuppnu 2 1\\
        3 \sqrt{\Lsuppnu 2 0 \Lsuppnu 3 1}
    \end{cases}.    
\end{align}

\begin{lemma}\label{le:local_decrease}
    If $\Lsupp 2$ exists, for points $x^k, x^{k+1} = x^k-\prk \ls \h(x^k) \rs^{-1}\g(x^k)$ holds decrease
    \begin{align*}
        \normMd {\g(x^{k+1})} {x^k}
        &\leq \lr \alunk + \frac {\Lsupp 2}{1+ \nu} \prk^{\nu} \normM {\g(x^k)}{x^k} ^{*\nu} \rr \prk \normMd{ \g(x^k) }{x^k},
     \end{align*}
    and hence,
    if $\nu >0$ and $\alunk\geq \normM{ \g(x^k) }{x^k}^{* \varepsilon}$ for $\varepsilon>0$, and if the bound \eqref{eq:consbound} exists (meaning that the Hessian does not change much), we have guaranteed superlinear local rate.

    \begin{remark}
    \citet{hanzely2022damped} shows that $\Lsuppnu 21$-Hölder continuity implies self-concordance, and \citep[Theorem 4.1.3]{10.5555/3317111} proves that self-concordance implies positive definiteness of Hessian $\h$ the domain of function $f$ contains no straight line.
    \end{remark}
\end{lemma}

\section{Simplified regularization} \label{sec:simplified_reg}
In the view of \Cref{ssec:regularization_generalized} and \Cref{le:polynomial_ag}, we can bound the majorization as
\begin{align}
    \modelun{x} 
    &= \argmin_{y \in \R^d} \left \{ f(x)+ \la \g(x), y-x \ra + \frac 12 \normsM {y-x} x + \frac \regc {2+\expo} \normM {y-x} x ^{2+\expo} \right\}\\
    &\leq \argmin_{y \in \R^d} \left \{ f(x)+ \la \g(x), y-x \ra + \frac \expo {2(\expo +2)} +\frac {\regc +1} {2+\expo} \normM {y-x} x ^{2+\expo} \right\}\\
    &= x- \lr\frac {1}{(\regc+1) \normM {\g(x^k)} {x^k} ^{* \expo}} \rr^{\frac 1{1+\expo}} \ls \h(x) \rs^{-1} \g(x), \label{eq:unbounded}
\end{align}
where stepsize was obtained as the positive root of polynomial $P[\alpha] \eqdef 1- \alpha^{1+\expo} (\regc+1) \normM {\g(x^k)} {x^k} ^{* \expo}$. 

Surprisingly, stepsize is unbounded, and when $\normMd {\g(x)}x \to 0$, then $\alpha \to \infty$. 
This puzzling result has a simple explanation -- such stepsize converges only to a neighborhood of the solution. 

In practice, we could not observe stepsize larger than $5$ on any considered dataset. When close to the solution and the stepsize becomes larger than one, algorithm \eqref{eq:unbounded} stops converging closer to the solution, and functional values oscillate. 

\section{Analysis under uniform star-convexity assumption in local norms}\label{sec:glob_superlinear}

\begin{proof}[\pof{\Cref{th:optami_superlinear}}]
We have that updates of \rn{} with $q=p+\nu = 2+\expo$ and any $\regc \geq \Lq$ can be written as
\begin{align}
    f(x^{k+1}) 
    & \leq \taylor {x^k} (x^{k+1}) + \frac {\regc} {q} \normM {x^{k+1} - x^k}{x^k} ^{q}\\
    & = \min_{y\in \R^d} \lc \taylor {x^k} (y) + \frac {\regc} {q} \normM {y - x}{x^k} ^{q} \rc,\\
    \intertext{using standard integration arguments from $\Lq$-Hölder continuity}
    & \leq \min_{y\in \R^d} \lc f(y) + \frac {\Lq} {(p+1)!} \normM {y - x^k}{x^k} ^{q} + \frac {\regc} {q} \normM {y - x^k}{x^k} ^{q}\rc\\
    &= \min_{y\in \R^d} \lc f(y) + \lr \frac {\Lq} {(p+1)!} + \frac {\regc} {q} \rr \normM {y - x^k}{x^k} ^{q}\rc,\\
    \intertext{setting $y\leftarrow x + \eta_k (\xopt -x^k)$ for arbitrary $\eta_k \in [0,1],$}
    &\leq f\lr x^k + \eta_k (\xopt -x^k) \rr + \eta_k^q\lr \frac {\Lq} {(p+1)!} + \frac {\regc} {q} \rr \normM {x^k - \xopt}{x^k} ^{q},\\
    \intertext{assuming $\mu_s$-strong star-convexity for $q\geq s \geq 2,$}
    &\leq (1-\eta_k)f(x^k) + \eta_k \fopt - \frac {\eta_k (1-\eta_k) \mu_s}s \normM {x^k-\xopt} {x^k} ^s + \eta_k^q \lr \frac {\Lq} {(p+1)!} + \frac {\regc} {q} \rr \normM {x^k - \xopt}{x^k} ^{q},\\
    \intertext{denoting functional suboptimality \gbox{$\delta_k \eqdef f(x^k)-\fopt$}, }
    \delta_{k+1} &
    \leq (1-\eta_k)\delta_k - \eta_k \normM {x^k-\xopt} {x^k} ^s \lr  (1-\eta_k) \frac {\mu_s}s - \eta_k^{q-1} \lr \frac {\Lq} {(p+1)!} + \frac {\regc} {q} \rr \normM {x^k - \xopt}{x^k} ^{q-s} \rr.
\end{align}
Denote expression \gbox{$E(\eta) \eqdef (1-\eta) \frac {\mu_s}s - \eta^{q-1} \lr \frac {\Lq} {(p+1)!} + \frac {\regc} {q} \rr \normM {x - \xopt}x ^{q-s}$} for $\eta \in [0,1]$. Observe that $E'(\eta)<0$ and therefore $E$ is monotonically decreasing on $\R^+$; with $E(0)\geq 0 \leq E(1)$ we can conclude that it has a unique root $\hat \eta$ on $[0,1]$.
With choice $\eta \leftarrow \hat \eta$ in the last inequality we can conclude global convergence rate
\begin{align}
    \delta_{k+1} &
    \leq (1-\hat \eta_k)\delta_k. \label{eq:optami_superlinear_delta}
\end{align}
Note that the root of the expression $E$ is inversely proportional to the distance from the solution $\normM {x - \xopt}x$, and therefore as the method converges, $x^k\to \xopt$, then the size of its root increases $\hat \eta_k \to 1$. Therefore, the global convergence rate \eqref{eq:optami_superlinear_delta} is superlinear.

Unrolling the recurrence \eqref{eq:optami_superlinear_delta} yields the inequality from the \Cref{th:optami_superlinear}.\\

Note that the decrease is based solely on the decrease in functional values, which allows us to prove the identical guarantee for Greedy Newton linesearch \greedy{}. In particular, \greedy{} implies $f(x^+_{\greedy{}}) \leq f(x^+_{\rn{}})$, and we can analogically conclude
\begin{align}
    f(x^{k+1}_{\greedy{}}) - \fopt &
    \leq \lr f(x^{k}_{\greedy{}}) - \fopt\rr (1-\hat \eta_k). \label{eq:optami_superlinear_greedy}
\end{align}
\end{proof}

\section{Analysis under $s$-relative size assumption} \label{sec:doikov_linear}
In this section, we present global convergence guarantees under a novel characteristic called $s$-relative size recently proposed by \citet{doikov2024super}.
\begin{definition}[\citep{doikov2024super}]\label{def:relsize}
For strictly convex function $f: \R^d \to \R$ we call \emph{$s$-relative size} characteristic 
\begin{align*}
    \ds &\eqdef \sup_{x,y \in \level} \lc \normM {x-y}x \lr \frac {\brmax}{\br xy} \rr^{\frac 1s} \rc,
\end{align*}
where $\br xy \eqdef \la \nabla f(x)-\nabla f(y) , x-y \ra>0$ and $\brmax \eqdef \sup_{x,y \in \level} \br xy$.
    \end{definition}
\begin{theorem}\label{th:doikov_linear_from_main}
    Let function $f$ be $\Lsup$-Hölder continuous, with finite relative size $D_q<\infty$ for $q=p+\nu$ (\Cref{def:relsize}) and $\cbound$-bounded Hessian change (\Cref{def:cbound}).
    Algorithms \rn{}, \un{} and \grls{} find points in the $\varepsilon$-neighborhood, $f(x^k)-f(x^*) \leq \varepsilon$, in
    \begin{align*}
        k
        &\leq  \cO \lr \cbound \lr \frac {\Lq\dss q ^q}{\brmax} \rr^{\frac 1{q-1}} \ln \frac {f_0} \varepsilon 
        + \ln \frac {\normMd {\g(x^0)} {x^{0}} D} {\varepsilon} \rr
        \hspace{-4.5mm}
    \end{align*}
    iterations, enjoying a global linear convergence rate.
\end{theorem}


Strict convexity implies $\br xy>0,$ we also have $\lim_{s \to \infty} D_{s}=D,$ also $\frac {\br xy}{\brmax} \leq 1,$ and
\begin{equation}
    \la \nabla f(x)-\nabla f(y) , x-y \ra 
    \geq \brmax \lr \frac {\normM {x-y} x }{\ds} \rr^s
\end{equation}
Characteristic $\ds$ is log-convex function in $s,$ and if $\dss {s_1}, \dss {s_2} <\infty,$ then for  $2 \leq s_1 \leq s \leq s_2$ holds
\begin{equation}
    \ds \leq \ls \dss {s_1} \rs^{\frac {s_2-s} {s_2-s_1}} \ls \dss {s_2} \rs^{\frac {s-s_1} {s_2-s_1}},
\end{equation}
and $\ds$ is continuous on this segment.

\begin{remark}
    For self-concordant functions, it holds $\br xy \geq \normsM {y-x} x$, and $\ds \leq D^{1-\frac 2s} \brmax^{\frac 1s}$.
\end{remark}
\begin{remark}
    For functions such that $\br xy \geq \mu_s \normM{x-y} x^s$ it holds $\ds \leq \lr \frac {\brmax} {\mu_s} \rr^{\frac 1s}$.
    In particular, for self-concordant functions holds $\br xy \geq \normsM {y-x} x$, and therefore $\dss 2 \leq \sqrt {\brmax}$.
\end{remark}

\begin{assumption}
    For some $s\geq 2,$ value of $\ds$ is finite, $\ds < \infty$.
\end{assumption}

\begin{lemma}\label{le:dsd_decrease}
    For any $2\leq s \leq q$, we have 
    \begin{equation}
        \lr \frac {\dss q} D \rr^q
        \leq \lr \frac {\dss s} D \rr^s
    \end{equation}
\end{lemma}
\begin{proof}[\pof{\Cref{le:dsd_decrease}}]
    Analogical to \citet{doikov2024super}.
\end{proof}

Now for any $x, y \in \level,$
\begin{align}
    f(y)
    &=f(x) + \la \g(x), y-x \ra + \int_0^1 \frac 1 \tau \la \g(x+\tau(y-x))-\g(x), \tau (y-x) \ra d\tau\\
    &\geq f(x) + \la \g(x), y-x \ra + \frac 1 s \brmax \lr \frac {\normM {x-y} x }{\ds} \rr^s,
\end{align}
and minimizing both sides w.r.t. $y$ independently, we get
\begin{align} \label{eq:strictly_gnorm_bound}
     \frac {s-1}s \lr \frac {\ds \normMd {\g(x)}x} {\brmax} \rr ^{\frac s {s-1}} 
    \geq \frac {f(x)-\fopt}\brmax
\end{align}

\newcommand{\omqs}{\omega_{q,s}}
\newcommand{\gqs}{{\hat \gamma}}
Let us denote some constants that will appear in proofs.
\begin{mdframed}[backgroundcolor=lightgray!20,hidealllines=true]
\begin{align}
    \gqs &\eqdef \frac {q(s-1)}{(q-1)s} \in \ls \frac 23,2 \rs, \qquad \text{and} \qquad 1-\gqs = \frac {q-s}{(q-1)s}\\
    \omqs &\eqdef \frac 12 \lr \frac s {s-1} \rr^{\gqs} \lr \frac {\brmax^{\frac qs}} {9\Lq\ds^q} \rr^{\frac 1{q-1}} =  
    \frac 12 \lr \frac s {s-1} \rr^{\frac {q(s-1)}{(q-1)s}} \lr \frac {\brmax^{\frac qs}} {9\Lq\ds^q} \rr^{\frac 1{q-1}}\\
    C_q &\eqdef 2\cbound(q-1)(9\Lq)^{\frac 1 {q-1}} D^{\frac q{q-1}}
\end{align}
\end{mdframed}

Note that $\frac {\omqs C_q} {\cbound (q-1)} = \lr \lr \frac s {s-1} \rr^{\frac {s-1}{s}} \frac {\brmax^{\frac 1s} D} \ds \rr^{\frac q{q-1}}.$

\begin{lemma} \label{le:strictly_one_step}
    For $q\in [2,4]$ and $s \in [2,\infty), $ we have
    \begin{equation}
        \frac 1{(\gqs-1) f_{k+1}^{\gqs-1}} - \frac 1{(\gqs-1) f_{k}^{\gqs-1}}
        \geq \omqs \frac {\normsMd{\g(x_{k+1})}{x_{k+1}}}{\normsMd{\g(x_{k})}{x_k}}.
    \end{equation}
\end{lemma}
\begin{proof}
    Analogically to \citet{doikov2024super}.
    \begin{align}
        f_k-f_{k+1}
        &\stackrel{\eqref{eq:one_step_gen}}\geq \frac 12 \lr \frac {1} {9\Lq} \rr^{\frac 1 {q-1}} \frac{\normsMd {\g(x^{k})}{x^k}}{ \normM {\g(x^k)}{x^k}^{*2} } \normM {\g(x^k)}{x^k}^{*\frac{q}{q-1}}\\
        &\stackrel{\eqref{eq:strictly_gnorm_bound}}\geq \frac 12 \lr \frac {1} {9\Lq} \rr^{\frac 1 {q-1}} \frac{\normsMd {\g(x^{k+1})}{x^k}}{ \normM {\g(x^k)}{x^k}^{*2} } \lr \frac {\brmax^{\frac 1s}} \ds \rr^{\frac q{q-1}} \lr \frac s {s-1} \rr^{\gqs} f_k^\gqs\\
        &= \frac 12 \lr \frac s {s-1} \rr^{\gqs} \lr \frac {\brmax^{\frac qs}} {9\Lq\ds^q} \rr^{\frac 1{q-1}} \frac{\normsMd {\g(x^{k+1})}{x^k}}{ \normM {\g(x^k)}{x^k}^{*2} } f_k^\gqs\\
        &= \omqs \frac{\normsMd {\g(x^{k+1})}{x^k}}{ \normM {\g(x^k)}{x^k}^{*2} } f_k^\gqs.
    \end{align}
    If $s\geq q,$ then $\gqs \in [1,2]$ and the function $y(x) \eqdef x^{\gqs-1}$ is concave. With monotonicity of $\lc f_k \rc_{k\geq 0}$, we have
    \begin{align}
        \frac 1{(\gqs-1) f_{k+1}^{\gqs-1}} - \frac 1{(\gqs-1) f_{k}^{\gqs-1}}
        = \frac {f_{k}^{\gqs-1} - f_{k+1}^{\gqs-1}}{(\gqs-1) f_{k+1}^{\gqs-1} f_{k}^{\gqs-1}}
        \geq \frac {f_{k} - f_{k+1}}{f_{k+1}^{\gqs-1} f_{k}}
        \geq \omqs \frac {\normsMd{\g(x_{k+1})}{x_{k}}}{\normsMd{\g(x_{k})}{x_k}}.
    \end{align}
    If $2\leq s<q,$ then $\gqs<1$ and the function $y(x) \eqdef x^{\gqs-1}$ is concave. We have 
    \begin{align}
        \frac 1{(\gqs-1) f_{k+1}^{\gqs-1}} - \frac 1{(\gqs-1) f_{k}^{\gqs-1}}
        = \frac {f_{k}^{1-\gqs} - f_{k+1}^{1-\gqs}}{1-\gqs}
        \geq \frac {f_{k} - f_{k+1}}{f_k^\gqs}
        \geq \omqs \frac {\normsMd{\g(x_{k+1})}{x_{k}}}{\normsMd{\g(x_{k})}{x_k}}.
    \end{align}
\end{proof}

\begin{theorem}\label{th:linrate_1}
    Let function $f$ be $\Lsup$-Hölder continuous with finite $s$-relative size and $\cbound$-bounded Hessian change, $\Lq, \ds< \infty$ for some $q\in [2,4]$ and $s\geq q$ and sequence of iterates $x^0, \dots, x^k$ by generated by one of the algorithms \rn{}, \un{}, \grls{}. If all iterates had function suboptimality worse than $\varepsilon>0$, $f_t\geq \varepsilon$ for $t \in \lc 0, \dots k\rc$, then the algorithm did at most
    \begin{align}
        k
        &\leq \frac \cbound {\omqs (\gqs-1)} \ls \frac 1{f_{k}^{\gqs-1}} - \frac 1{f_{0}^{\gqs-1}} \rs + 2\ln \frac {\normMd {\g(x^0)} {x^{0}} D} {f_k}\\
        &\leq  2\cbound \frac {s(q-1)} {s-q} \lr \frac {s-1} {s} \rr^{\frac {q(s-1)}{(q-1)s}} \lr \frac {9\Lq\ds^q}{\brmax^{\frac qs}} \rr^{\frac 1{q-1}} \ls \varepsilon^{-\frac{s-q}{s(q-1)}} - f_{0}^{-\frac{s-q}{s(q-1)}} \rs + 2\ln \frac {\normMd {\g(x^0)} {x^{0}} D} {\varepsilon}
    \end{align}
    steps. If $s=q,$ treating RHS as limit together with $\lim_{a \to 0} \frac{b^{-a} - c^{-a}} a = \ln \lr \frac c b \rr$ guarantees the linear convergence rate
    \begin{align}
        k
        &\leq  2\cbound \frac {q-1} {q}  
        \lr \frac {9\Lq\dss q ^q}{\brmax} \rr^{\frac 1{q-1}} \ln \frac {f_0} \varepsilon 
        + 2\ln \frac {\normMd {\g(x^0)} {x^{0}} D} {\varepsilon}.
    \end{align}
\end{theorem}
\begin{remark}
We can analogically guarantee the global linear convergence of Greedy Newton linesearch \greedy{} \eqref{eq:greedy_linesearch}, but with a slightly different constant.    
\end{remark}
\begin{proof} Telescoping \Cref{le:strictly_one_step},
    \begin{align}
        \frac 1{(\gqs-1) f_{k}^{\gqs-1}} - \frac 1{(\gqs-1) f_{0}^{\gqs-1}}
        &\geq \omqs \sum _{t=0} ^{k-1} \frac{\normsMd {\g(x^{t+1})}{x^t}}{\normsMd {\g(x^t)}{x^t}}\\
        &\geq k \omqs \lr \prod _{t=0} ^{k-1} \frac{\normsMd {\g(x^{t+1})}{x^t}}{\normsMd {\g(x^t)}{x^t}} \rr^{\frac 1k}\\
        &\geq   \frac{k \omqs} \cbound  \lr \frac {f_k} {\normMd {\g(x^0)} {x^{0}} D} \rr^{\frac k2}\\
        &\geq   \frac{k \omqs} \cbound  \exp \lr - \frac 2k \ln \frac {\normMd {\g(x^0)} {x^{0}} D} {f_k} \rr\\
        &\geq   \frac{k \omqs} \cbound  \lr 1 - \frac 2k \ln \frac {\normMd {\g(x^0)} {x^{0}} D} {f_k} \rr\\
        &=\frac{k \omqs} \cbound - \frac{2 \omqs} \cbound \ln \frac {\normMd {\g(x^0)} {x^{0}} D} {f_k},
    \end{align}
    hence 
    \begin{align}
        k
        &\leq \frac \cbound {\omqs (\gqs-1)} \ls \frac 1{f_{k}^{\gqs-1}} - \frac 1{f_{0}^{\gqs-1}} \rs + 2\ln \frac {\normMd {\g(x^0)} {x^{0}} D} {f_k}\\
        &\leq \frac \cbound {\omqs (\gqs-1)} \ls \frac 1{f_{k}^{\gqs-1}} - \frac 1{f_{0}^{\gqs-1}} \rs + 2\ln \frac {\normMd {\g(x^0)} {x^{0}} D} {\varepsilon}.
    \end{align}
\end{proof}

\begin{theorem}\label{th:linrate2}
    Let funciton $f$ be $\Lsup$-Hölder continuous with finite $s$-relative size and $\cbound$-bounded Hessian change, $\Lq, \ds< \infty$ for some $q\in [2,4]$ and $2\leq s \leq q$ and sequence of iterates $x^0, \dots, x^k$ by generated by one of the algorithms \rn{}, \un{}, \grls{}. If all iterates were far from solution, $f_t\geq \varepsilon>0$ and $g_t\eqdef\normMd {\g(x^t)}{x^t} \geq \delta>0$ for $t \in \lc 0, \dots k\rc$, then the algorithm did at most
    \begin{align}
        k
        & \leq 2 \cbound \frac qs  \lr \frac {s-1}s \rr^{\frac {s-1}{q-1}} 
        \lr \frac {9\Lq\ds^s D^{q-s}} {\brmax} \rr^{\frac 1{q-1}}  \frac  {s(q-1)} {q-s} \ls 1 - \frac sq \lr \lr \frac s {s-1} \rr^{s-1} \frac {\ds^s} {\brmax D^s} \varepsilon \rr^{\frac {q-s}{s(q-1)}} \rs \nonumber \\
        & \qquad + 2 \ln \frac {g_0} \delta
    \end{align}
    steps. If $s=q,$ treating RHS as a limit guarantees linear convergence rate
    \begin{align}
        k
        & \leq 2 \cbound \frac {q-1}q
        \lr \frac {9\Lq D_q^q} {\brmax} \rr^{\frac 1{q-1}} \ln \lr \lr \frac q {q-1} \rr^{ q-1} \frac {\brmax D^q} {D_q^q \varepsilon} \rr + 2 \ln \frac {g_0} \delta.
    \end{align}
\end{theorem}
\begin{proof}
    Note $1-\gqs=\frac {q-s}{s(q-1)}>0$. Let's split the analysis of the method into two stages, $k=m+n$. With $C_q = 2\cbound(q-1)(9\Lq)^{\frac 1 {q-1}} D^{\frac q{q-1}}$, we bound the first stage,
    \begin{align}
        C_q \frac 1 {f_{m}^{\frac 1{q-1}}}
        &\geq C_q \ls \frac 1 {f_{m}^{\frac 1{q-1}}} - \frac 1 {f_{0}^{\frac 1{q-1}}} \rs
        \stackrel{\eqref{eq:one_step_universal}}\geq m \lr \frac {g_m}{g_0} \rr^{\frac 2m}
        = m \exp \lr \frac 2m \ln \frac {g_m}{g_0} \rr\\
        &\geq m+2\ln \frac {g_m}{g_0}
        = m+2\ln \frac {g_m}\delta -2\ln \frac {g_0}\delta.
    \end{align}

    For the second stage, telescoping inequalities for $t=m,\dots, k-1$
    \begin{align}
        \frac 1{\omqs(1-\gqs)} \ls f_{t+1}^{1-\gqs} - f_{t}^{1-\gqs} \rs
        \geq \frac {\normsMd{\g(x_{t+1})}{x_{t+1}}}{\normsMd{\g(x_{t})}{x_t}},
    \end{align}
    we get
    \begin{align}
        \frac \cbound {\omqs(1-\gqs)} \ls f_{m}^{1-\gqs} - \varepsilon^{1-\gqs} \rs
        &\geq \cbound \sum_{t=m}^{k-1} \frac {\normsMd{\g(x_{t+1})}{x_{t+1}}}{\normsMd{\g(x_{t})}{x_t}}
        \geq n \lr \frac {g_k}{g_m}\rr ^{\frac 2n}
        \geq n \lr \frac {\delta}{g_m}\rr ^{\frac 2n}\\
        &\geq n - 2 \ln \frac {g_m} {\delta}.
    \end{align}
    Expressing $n,m$ from the inequalities above and adding them together yields
    \begin{align}
        k\leq C_q \frac 1 {f_{m}^{\frac 1{q-1}}} + \frac \cbound {\omqs(1-\gqs)} \ls f_{m}^{1-\gqs} - \varepsilon^{1-\gqs} \rs + 2 \ln \frac {g_0} \delta.
    \end{align}
    Note that $1-\gqs=\frac {q-s}{s(q-1)}$. Minimizer of RHS in $f_m$ is achieved at
    \begin{equation}
    f_m^*
    \eqdef \lr \frac {C_q \omqs} {\cbound(q-1)} \rr ^{\frac {s(q-1)}q} 
    = \lr \frac s {s-1} \rr^{\frac s-1} \frac {\brmax D^s} {\ds^s}.
    \end{equation}
    Substituting definitions of $f_m^*, \omqs, C_q, \gqs$ into the terms we get
    \begin{align*}
        C_q \frac 1 {f_{m}^{* \, \frac 1{q-1}}}
        &= 2\cbound(q-1)
        \lr \frac {s-1}s \rr^{\frac {s-1}{q-1}}  \lr \frac { 9\Lq \ds^s D^{q-s}} { \brmax} \rr^{\frac 1 {q-1}},\\
        \frac \cbound {\omqs(1-\gqs)} f_{m}^{*\, (1-\gqs)}
        &= \cbound \frac  {s(q-1)} {q-s} \frac 1 \omqs f_{m}^{* \, \frac {q-s}{s(q-1)}}\\
        &= 2\cbound \frac  {s(q-1)} {q-s} \lr \frac {s-1}s \rr^{\frac {s-1}{q-1}} 
        \lr \frac {9\Lq\ds^s D^{q-s}} {\brmax} \rr^{\frac 1{q-1}},\\
        \frac \cbound {\omqs(1-\gqs)} \varepsilon^{1-\gqs} 
        &= 2 \cbound \frac  {s(q-1)} {q-s}  \lr \frac {s-1}s \rr^{\frac {q(s-1)}{(q-1)s}} \lr \frac {9\Lq\ds^q} {\brmax^{\frac qs}} \rr^{\frac 1{q-1}} \varepsilon^{\frac {q-s}{s(q-1)}},
    \end{align*}
    and plugging them back in, we conclude
    \begin{align*}
        k
        &\leq C_q \frac 1 {f_{m}^{* \, \frac 1{q-1}}} + \frac \cbound {\omqs(1-\gqs)} \ls f_{m}^{*\, (1-\gqs)} - \varepsilon^{1-\gqs} \rs + 2 \ln \frac {g_0} \delta\\
        &= 2\cbound (q-1) \frac  {q} {q-s}  \lr \frac {s-1}s \rr^{\frac {s-1}{q-1}} 
        \lr \frac {9\Lq\ds^s D^{q-s}} {\brmax} \rr^{\frac 1{q-1}}
        - \frac \cbound {\omqs(1-\gqs)} \varepsilon^{1-\gqs}  + 2 \ln \frac {g_0} \delta\\
        &= 2 \cbound \frac qs  \lr \frac {s-1}s \rr^{\frac {s-1}{q-1}} 
        \lr \frac {9\Lq\ds^s D^{q-s}} {\brmax} \rr^{\frac 1{q-1}}
         \frac  {s(q-1)} {q-s} \times\\
         & \hspace {4cm} \times \ls 1 - \frac sq \lr \lr \frac s {s-1} \rr^{s-1} \frac {\brmax D^s} {\ds^s} \rr^{\frac {q-s}{s(q-1)}} \varepsilon^{\frac {q-s}{s(q-1)}} \rs + 2 \ln \frac {g_0} \delta.
    \end{align*}
\end{proof}

\section{Proofs} \label{sec:proofs}


\subsection{Proof of Lemma \ref{le:polynomial_ag}}
\begin{proof}[\pof{\Cref{le:polynomial_ag}}]
    Using weighed AG inequality, for $0\leq b \leq p$, we have
    \begin{equation}
        x^b \leq \frac {(p-b) + bx^p}p .
    \end{equation}
    We use this inequality for each term of the polynomial.
\end{proof}

\subsection{Proof of Proposition \ref{prop:holder_eq}}

\begin{proof}[\pof{\Cref{prop:holder_eq}}]
    We can derive all of the inequalities straightforwardly
    \begin{align*}
        \g(y)-\g(x)-\h(x) \ls y-x \rs
        & = \int_0^1 \lr \h(x-\tau(y-x))-\h(x)\rr [y-x] d\tau\\
        \normMd{\g(y)-\g(x)-\h(x) \ls y-x \rs} x
        & \leq \int_0^1 \normM {\h(x-\tau(y-x))-\h(x)}{op} \normM{y-x}x d\tau\\
        & \leq \Lsupp 2 \normM{y-x}x^{1+\nu} \int_0^1 \tau^\nu d\tau\\
        & = \frac{\Lsupp 2}{1+\nu} \normM{y-x}x^{1+\nu},
    \end{align*}
    \begin{align*}
        \h(y)-\h(x)-\td(x) \ls y-x \rs 
        & = \int_0^1 \lr \td(x-\tau(y-x))-\td(x)\rr [y-x] d\tau\\
        \normM{\h(y)-\h(x)-\td(x) \ls y-x \rs} {op}
        & \leq \int_0^1 \normM {\td(x-\tau(y-x))-\td(x)}{op} \normM{y-x}x d\tau\\
        & \leq \Lsupp 3 \normM{y-x}x^{1+\nu} \int_0^1 \tau^\nu d\tau\\
        & = \frac{\Lsupp 3}{1+\nu} \normM{y-x}x^{1+\nu},
    \end{align*}
    \begin{align*}
        \g(y)-\g(x)-\h(x) \ls y-x \rs - \frac 12 \td(x)[y-x]^2
        & = \int_0^1 \int_0 ^\tau \lr \td(x+\regc (y-x))-\td(x)\rr [y-x]^2 d\regc d\tau\\
        \normMd{\g(y)-\g(x)-\h(x) \ls y-x \rs - \frac 12 \td(x)[y-x]^2}{x}
        & \leq \int_0^1 \int_0 ^\tau \normMd{\td(x+\regc (y-x))-\td(x)}x  \normsM{y-x}x d\regc d\tau\\
        & \leq \Lsupp 3 \normM{y-x}x^{2+\nu} \int_0^1 \int_0^\tau \regc^\nu d\regc d\tau\\
        & = \frac{\Lsupp 3}{(1+\nu)(2+\nu)} \normM{y-x}x^{2+\nu}.
    \end{align*}
\end{proof}

\subsection{Proof of Lemma \ref{le:td_bound}}
\begin{proof}[\pof{\Cref{le:td_bound}}]
For any $x,h, y \in \EE$ and taking $y=x+\tau u $ for $\tau >0, \normM u x =1$
\begin{align*}
    0 
    &\leq \normsM h y \leq \normsM h x + \la \nabla ^3 f(x) [h]^2, y-x \ra + \frac {\Lsupp 3 } {1+\nu} \normM{y-x}{x}^{1+\nu} \normsM {h} {x}\\
    0 &\leq \frac 1 \tau \normsM h x + \la \nabla ^3 f(x) [h]^2, u \ra + \frac {\Lsupp 3 \tau^\nu} {1+\nu} \normsM {h} {x}\\
    \normMd {\nabla ^3 f(x) [h]^2} x  &\leq \lr \frac 1 \tau + \frac {\Lsupp 3 \tau^\nu} {1+\nu} \rr \normsM {h} {x}\\
\end{align*}
Setting
\begin{equation*}
    \tau = \lr \frac {1 + \nu}{\Lsupp 3} \rr^{\frac 1 {1+\nu}},
\end{equation*}
we get
\begin{align*}
    \normMd {\nabla ^3 f(x) [h]^2} x  
    &\leq 2 \lr \frac {\Lsupp 3 } {1+\nu} \rr ^{\frac 1 {1+\nu}} \normsM {h} {x}.
\end{align*}
Setting $x^k = x, h= \xdiff$ we get 
\begin{align*}
    \normMd {\nabla ^3 f(x^k) [\xdiff]^2} {x^k}  
    &\leq 2 \lr \frac {\Lsupp 3 } {1+\nu} \rr ^{\frac 1 {1+\nu}} \normsM {\xdiff} {x^k}
    = 2 \lr \frac {\Lsupp 3 } {1+\nu} \rr ^{\frac 1 {1+\nu}} \prk^2 \normsMd {\g(x^k)} {x^k}    
\end{align*}
\end{proof}

\subsection{Proof of Lemma \ref{le:local_decrease}}
\begin{proof}{\pof{\Cref{le:local_decrease}}}
    \begin{align*}
        \normMd {\g(x^{k+1})} {x^k}
        &= \normMd {\g(x^{k+1}) \blue{-\h(x^k) \ls \xdiff \rs - \prk \g(x^k) }}{x^k}\\
        &= \normMd {\g(x^{k+1}) -\g(x^k) - \h(x^k) \ls \xdiff \rs + \lr 1- \prk \rr \g(x^k) }{x^k}\\
        &\leq \normMd {\g(x^{k+1}) -\g(x^k) - \h(x^k) \ls \xdiff \rs}{x^k} + \lr 1- \prk \rr \normMd{ \g(x^k) }{x^k}\\
        &\leq \frac {\Lsupp 2}{1+ \nu} \normM {\xdiff}{x^k} ^{1+\nu} + \lr 1- \prk \rr \normMd{ \g(x^k) }{x^k} \tag{if $\Lsupp 2$ exists}\\
        &= \frac {\Lsupp 2}{1+ \nu} \prk^{1+\nu} \normM {\g(x^k)}{x^k} ^{*(1+\nu)} + \lr 1- \prk \rr \normMd{ \g(x^k) }{x^k}\\
        &= \lr 1- \prk + \frac {\Lsupp 2}{1+ \nu} \prk^{1+\nu} \normM {\g(x^k)}{x^k} ^{*\nu} \rr \normMd{ \g(x^k) }{x^k}\\
        &= \lr \alunk + \frac {\Lsupp 2}{1+ \nu} \prk^{\nu} \normM {\g(x^k)}{x^k} ^{*\nu} \rr \prk \normMd{ \g(x^k) }{x^k}.
    \end{align*}
    Hence
    \begin{align*}
        \normMd {\g(x^{k+1})} {x^k} \leq 
        \begin{cases}
        2 \frac {\Lsupp 2}{1+ \nu} \prk^{1+\nu} \normM{ \g(x^k) }{x^k}^{*(1+\nu)} & \text{ if $\alunk \leq \frac {\Lsupp 2}{1+\nu}\prk^{\nu} \normM{ \g(x^k) }{x^k}^{*\nu}$}\\
        2\alunk \prk \normMd{ \g(x^k) }{x^k} & \text{ if $\alunk \geq \frac {\Lsupp 2} {1+\nu} \prk^{\nu} \normM{ \g(x^k) }{x^k}^{*\nu}$}
        \end{cases}
    \end{align*}
\end{proof}

\subsection{Proof of Lemma \ref{le:holder_bound_two}}
\begin{proof} [\pof{\Cref{le:holder_bound_two}}] 
    We can rewrite the Hölder continuity for points $x^k, x^{k+1}$ s.t. $x^{k+1}=x^k-\prk\lr\h(x^k)\rr^{-1}\g(x^k)$ 
    \begin{align*}
        &\lr\frac {\Lsupp 2} {1+\nu} \lr\prk \normMd {\g(x^k)}{x^k} \rr ^{1+\nu}\rr^2\\
        &=\lr\frac {\Lsupp 2} {1+\nu} \normM {\xdiff}{x^k} ^{1+\nu}\rr^2\\
        &\geq \normsMd{\g(x^{k+1})-\g(x^k)-\h(x^k)\ls\xdiff\rs}{x^k}\\
        &= \normsMd{\g(x^{k+1})-\g(x^k) + \prk\g(x^k)} {x^k}\\
        &= \normsMd{\g(x^{k+1})- \lr 1 -\prk \rr \g(x^k)} {x^k}\\
        &= \normsMd{\g(x^{k+1})}{x^k} + \lr 1 -\prk \rr^2\normsMd{\g(x^k)} {x^k} - 2\lr 1-\prk \rr \la \g(x^{k+1}), \ls \h(x^k)\rs^{-1} \g(x^k) \ra.
    \end{align*}
    We are going to set $\regc$ so that
    \begin{align}
        \frac {1 -\prk } 2 \normsMd{\g(x^k)} {x^k} \geq \frac 1 {2(1-\prk)} \lr\frac {\Lsupp 2} {1+\nu}  \lr\prk \normMd {\g(x^k)}{x^k} \rr ^{1+\nu}\rr^2, \label{eq:regc_two}
    \end{align}
    and hence, we can conclude the proof by rearranging,
    \begin{align*}
        &\la \g(x^{k+1}), \ls \h(x^k)\rs^{-1} \g(x^k) \ra \\
        &\geq \frac 1{2(1-\prk)}  \normsMd{\g(x^{k+1})}{x^k} + \frac {1 -\prk } 2 \normsMd{\g(x^k)} {x^k} - \frac 1 {2(1-\prk)} \lr\frac {\Lsupp 2} {1+\nu}  \lr\prk \normMd {\g(x^k)}{x^k} \rr ^{1+\nu} \rr^2\\
        &\geq \frac 1{2(1-\prk)}  \normsMd{\g(x^{k+1})}{x^k}.
    \end{align*}    
    Now we are going to choose $\regc$ to satisfy \eqref{eq:regc_two}.
    Because $\prk$ is a root of a polynomial $P$, we have
    \begin{equation*}
        1- \prk - \prk^{1+\expo} \lunk =0,
    \end{equation*}
    so the equation \eqref{eq:regc_two} is equivalent to
    \begin{align*}
        1-\prk
        = \prk^{1+\expo} \lunk
        &\geq \frac {\Lsupp 2} {1+\nu}  \prk^{1+\nu} \normM {\g(x^k)}{x^k} ^{*\nu},\\
        \alunk 
        &\geq \frac {\Lsupp 2} {1+\nu} \prk^{\nu} \normM {\g(x^k)}{x^k} ^{*\nu}.
    \end{align*}
\end{proof}

\subsection{Proof of Lemma \ref{le:holder_bound_three}}
\begin{proof}[\pof{\Cref{le:holder_bound_three}}]
    We can rewrite the Hölder continuity for points $x^k, x^{k+1}$ s.t. $x^{k+1}=x^k-\prk\lr\h(x^k)\rr^{-1}\g(x^k)$ 
    \begin{align}
        \hspace{-1cm}&\frac {\Lsupp 3 }{(1+\nu)(2+\nu)} \lr\prk \normMd {\g(x^k)}{x^k} \rr ^{2+\nu}\\
        & =\frac {\Lsupp 3 }{(1+\nu)(2+\nu)} \normM{\xdiff} {x^k} ^{2+ \nu} \\
        & \geq \normMd { \g(x^{k+1}) - \g(x^k) - \h (x^k)[\xdiff] - \frac 12 \nabla^3 f(x^k)[\xdiff]^2 } {x^k}\\
        & = \normMd { \g(x^{k+1}) - (1- \prk) \g(x^k) - \frac 12\nabla^3 f(x^k)[\xdiff]^2 } {x^k}.
    \end{align}
    Squaring
    \begin{align*}
        \hspace{-1cm} & \purple{\lr \frac {\Lsupp 3 }{(1+\nu)(1+ \nu)} \lr\prk \normMd {\g(x^k)}{x^k} \rr ^{2+\nu} \rr^2}\\
        & \geq \normsMd { \g(x^{k+1}) - (1- \prk) \g(x^k) - \frac 12\nabla^3 f(x^k)[\xdiff]^2 } {x^k}\\
        & = \normsMd { \g(x^{k+1})}{x^k}  + (1- \prk)^2 \normsMd {\g(x^k)}{x^k} + \frac 14 \normsMd {\nabla^3 f(x^k)[\xdiff]^2 } {x^k}\\
        & \qquad \magenta{- 2(1-\prk) \la \g(x^{k+1}), \ls \h(x^k) \rs ^{-1}\g(x^{k}) \ra}\\
        & \qquad  {+ (1-\prk) \la \ls \h(x^k) \rs ^{-\frac 12} \g(x^{k}), \ls \h(x^k) \rs ^{-\frac 12} \nabla^3 f(x^k)[\xdiff]^2 \ra}\\
        & \qquad - \la \ls \h(x^k)\rs^{-\frac 12} \g(x^{k+1}), \ls \h(x^k) \rs ^{-\frac 12} \nabla^3 f(x^k)[\xdiff]^2 \ra\\
        & \geq \frac 12 \normsMd { \g(x^{k+1})}{x^k}  + (1- \prk)^2 \normsMd {\g(x^k)}{x^k} \orange{- \frac 14 \normsMd {\nabla^3 f(x^k)[\xdiff]^2 } {x^k}}\\
        & \magenta{\qquad - 2(1-\prk) \la \g(x^{k+1}), \ls \h(x^k) \rs ^{-1}\g(x^{k}) \ra}\\
        & \qquad {- (1-\prk) \normMd {\g(x^{k})}{x^k} \normM {\nabla^3 f(x^k)[\xdiff]^2 } {x^k}}\\
        & \geq \frac 12 \normsMd { \g(x^{k+1})}{x^k}  + (1- \prk)^2 \normsMd {\g(x^k)}{x^k} \orange{- \lr \frac {\Lsupp 3 } {1+\nu} \rr ^{\frac 2 {1+\nu}} \prk^4 \normM {\g(x^k)} {x^k}^4}\\
        & \qquad \magenta{- 2(1-\prk) \la \g(x^{k+1}), \ls \h(x^k) \rs ^{-1}\g(x^{k}) \ra}\\
        & \qquad {- 2 \lr \frac {\Lsupp 3 } {1+\nu} \rr ^{\frac 1 {1+\nu}}  \prk^2(1-\prk) \normM {\g(x^k)} {x^k}^{*3}}.        
    \end{align*}
    Rearranging yields
    \begin{align*}
        \hspace {-1cm}& \magenta{\la \g(x^{k+1}), \ls \h(x^k) \rs ^{-1}\g(x^{k}) \ra} \\
        &\geq \frac 1{4(1-\prk)} \normsMd { \g(x^{k+1})}{x^k}  + \frac  {1- \prk}2 \normsMd {\g(x^k)}{x^k} \orange{- \frac 12 \lr \frac {\Lsupp 3 } {1+\nu} \rr ^{\frac 2 {1+\nu}} \frac{ \prk^4}{1-\prk} \normM {\g(x^k)} {x^k}^{*4}}\\
        & \qquad {- \lr \frac {\Lsupp 3 } {1+\nu} \rr ^{\frac 1 {1+\nu}}  \prk^2 \normM {\g(x^k)} {x^k}^{*3}} 
        \purple{- \frac 1{2(1-\prk)} \lr \frac {\Lsupp 3 }{(1+\nu) (2+ \nu)}  \rr^2 \lr\prk \normMd {\g(x^k)}{x^k} \rr ^{2(2+\nu)}}.
    \end{align*}
    Finally, we are going to set $\alunk$ so that 
    \begin{align}
        \frac  {1- \prk}6 \normsMd {\g(x^k)}{x^k} &\geq \orange{\frac 12 \lr \frac {\Lsupp 3 } {1+\nu} \rr ^{\frac 2 {1+\nu}} \frac{ \prk^4}{1-\prk} \normM {\g(x^k)} {x^k}^{*4} }\label{eq:regc_threea}\\
        \frac  {1- \prk}6 \normsMd {\g(x^k)}{x^k} & \geq { \lr \frac {\Lsupp 3 } {1+\nu} \rr ^{\frac 1 {1+\nu}}  \prk^2 \normM {\g(x^k)} {x^k}^{*3}} \label{eq:regc_threeb}\\
        \frac  {1- \prk}6 \normsMd {\g(x^k)}{x^k} & \geq \purple{ \frac 1{2(1-\prk)} \lr \frac {\Lsupp 3 }{(1+\nu) (2+ \nu)}  \rr^2 \lr\prk \normMd {\g(x^k)}{x^k} \rr ^{2(2+\nu)} }\label{eq:regc_threec}
    \end{align}
    and then we can conclude
    \begin{align*}
        \la \g(x^{k+1}), \ls \h(x^k) \rs ^{-1}\g(x^{k}) \ra 
        &\geq \frac 1{4(1-\prk)} \normsMd { \g(x^{k+1})}{x^k} .
    \end{align*}
    Note that the choice of stepsize implies 
    \begin{equation*}
    1-\prk = \prk^{1+\expo} \lunk
    \end{equation*}        
    and \eqref{eq:regc_threea}, \eqref{eq:regc_threeb}, \eqref{eq:regc_threec} are satisfied as
    \begin{align*}
        &1-\prk = \prk^{1+\expo} \lunk  \geq \\
        & \quad \begin{cases}
        \sqrt 3 \lr \frac {\Lsupp 3 } {1+\nu} \rr ^{\frac 1 {1+\nu}} \prk^2 \normMd {\g(x^k)} {x^k} & \text { if } \alunk \geq \sqrt {3} \lr \frac {\Lsupp 3 } {1+\nu} \rr ^{\frac 1 {1+\nu}} \prk \normMd {\g(x^k)} {x^k}\\
        6 \lr \frac  {\Lsupp 3 } {1+\nu} \rr ^{\frac 1 {1+\nu}}  \prk^2 \normMd {\g(x^k)} {x^k} & \text{ if } \alunk \geq 6 \lr \frac  {\Lsupp 3 } {1+\nu} \rr ^{\frac 1 {1+\nu}}  \prk\normMd {\g(x^k)} {x^k}\\
        \frac {\sqrt 3\Lsupp 3 }{(1+\nu) (1+ \nu)}  \prk^{2+\nu} \normM {\g(x^k)}{x^k} ^{*(1+\nu)} & \text{ if } \alunk \geq \frac {\sqrt 3\Lsupp 3 }{(1+\nu) (2+ \nu)}  \prk^{1+\nu} \normM {\g(x^k)}{x^k} ^{*(1+\nu)}.
        \end{cases} 
    \end{align*}
    We can ensure \eqref{eq:regc_threea}, \eqref{eq:regc_threeb}, \eqref{eq:regc_threec} by 
    \begin{align*}
        \alunk 
        &\geq \prk\normMd {\g(x^k)} {x^k} \max \lc 6 \lr \frac  {\Lsupp 3 } {1+\nu} \rr ^{\frac 1 {1+\nu}}, \frac {\sqrt {3}\Lsupp 3 }{(1+\nu) (2+ \nu)}  \prk^{\nu} \normM {\g(x^k)}{x^k} ^{*\nu} \rc.
    \end{align*}
\end{proof}

\subsection{Towards the proof of Theorem \ref{th:holder_un}}
We unify cases $p=2,3$ with the \Cref{le:alunk_alpha}. 
\begin{corollary} \label{col:holder_three}
    \Cref{le:alunk_alpha} with $\gamma=\nu$ implies that choice $\alunk = \lr \frac {\Lsupp 2} {1+\nu} \rr^{\frac 1 {1+\nu}}  \normM {\g(x^k)}{x^k} ^{*\frac \nu{1+\nu}}$ satisfies \eqref{eq:alunkmintwo}, and therefore
    \Cref{le:holder_bound_two} implies decrease as \citet{doikov2024super},
    \begin{align}
        f(x^k)-f(x^{k+1})
        &\geq \frac {1} {\alunk}\normsMd { \g(x^{k+1})} {x^k}
        \geq \lr \frac {1+\nu} {\Lsupp 2} \rr^{\frac 1 {1+\nu}} \frac{\normsMd {\g(x^{k+1})}{x^k}}{ \normM {\g(x^k)}{x^k} ^{*\frac \nu{1+\nu}} } .
    \end{align}
    
    \Cref{le:alunk_alpha} with $\gamma \in \{1, 1+\nu \}$ implies that the choice 
    \begin{align}
        \alunk
        &\geq \normM {\g(x^k)}{x^k} ^{* \frac 12} \max \lc \lr \frac  {6^{1+\nu}\Lsupp 3 } {1+\nu} \rr ^{\frac 1 {2(1+\nu)}}, \lr \frac {\sqrt {3} \Lsupp 3 }{(1+\nu) (2+ \nu)} \rr^{\frac 1{2+\nu}} \normM {\g(x^k)}{x^k} ^{*\frac{\nu}{2(2+\nu)}} \rc, \label{eq:alunkthreetight}
    \end{align}
    satisfies \eqref{eq:alunkminthree}, and therefore \Cref{le:holder_bound_three} implies decrease
    \begin{align}
        &f(x^k)-f(x^{k+1})
        \geq \frac {1} {2\alunk}\normsMd { \g(x^{k+1})} {x^k}\\
        &\qquad \geq \frac 1 { \max \lc \lr \frac  {6^{1+\nu}\Lsupp 3 } {1+\nu} \rr ^{\frac 1 {2(1+\nu)}}, \lr \frac {\sqrt {3} \Lsupp 3 }{(1+\nu) (2+ \nu)} \rr^{\frac 1{2+\nu}} \normM {\g(x^k)}{x^k} ^{*\frac{\nu}{2(2+\nu)}} \rc} \frac{\normsMd {\g(x^{k+1})}{x^k}}{ \normM {\g(x^k)}{x^k}^{*\frac12} }.
    \end{align}

    On the other hand, choice of $\alunk = \lr \frac  {6^{1+\nu}\Lsupp 3 } {1+\nu} \rr ^{\frac 1 {2+\nu}} \normM {\g(x^k)}{x^k}^{*\frac {1+\nu}{2+\nu}}$ in \Cref{le:holder_bound_three} implies decrease as \citet{doikov2024super},
    \begin{align}
        f(x^k)-f(x^{k+1})
        &\geq \frac {1} {2\alunk}\normsMd { \g(x^{k+1})} {x^k}
        \geq \frac 1 2\lr \frac {1+\nu} {6^{1+\nu}\Lsupp 3} \rr^{\frac 1 {2+\nu}} \frac{\normsMd {\g(x^{k+1})}{x^k}}{ \normM {\g(x^k)}{x^k}^{*\frac{1+\nu}{2+\nu}} }. \label{eq:pthreedoikov}
    \end{align}
\end{corollary}

\subsubsection{Proof of Theorem \ref{th:holder_un}}
We can combine previous corollaries.
\begin{proof} [\pof{\Cref{th:holder_un}}]    
    For $p=2$, choice $\alunk = \lr \frac {\Lsup } {p-1+\nu} \rr^{\frac 1 {p-1+\nu}}  \normM {\g(x^k)}{x^k} ^{*\frac {p-2+\nu}{p-1+\nu}}$ implies
    \begin{align}
        f(x^k)-f(x^{k+1})
        \geq \lr \frac {p-1+\nu} {\Lsup} \rr^{\frac 1 {p-1+\nu}} \frac{\normsMd {\g(x^{k+1})}{x^k}}{ \normM {\g(x^k)}{x^k} ^{*\frac {p-2+\nu}{p-1+\nu}} } .
    \end{align}

   For $p=3$, choice $\alunk = 6\lr \frac  {\Lsup } {3(p-1+\nu)} \rr ^{\frac 1 {p-1+\nu}} \normM {\g(x^k)}{x^k}^{*\frac {p-2+\nu}{p-1+\nu}}$ implies
    \begin{align}
        f(x^k)-f(x^{k+1})
        \geq \frac 1 {12} \lr \frac {3(p-1+\nu)} {\Lsup} \rr^{\frac 1 {p-1+\nu}} \frac{\normsMd {\g(x^{k+1})}{x^k}}{ \normM {\g(x^k)}{x^k}^{*\frac{p-2+\nu}{p-1+\nu}} }. 
    \end{align}
    And for any $p \in \lc 2,3 \rc$ we have that $\alunk = 6\lr \frac  {\Lsup } {3(p-1+\nu)} \rr ^{\frac 1 {p-1+\nu}} \normM {\g(x^k)}{x^k}^{*\frac {p-2+\nu}{p-1+\nu}}$ implies
    \begin{align}
        f(x^k)-f(x^{k+1})
        \geq \frac 1 {12} \lr \frac {3(p-1+\nu)} {\Lsup} \rr^{\frac 1 {p-1+\nu}} \frac{\normsMd {\g(x^{k+1})}{x^k}}{ \normM {\g(x^k)}{x^k}^{*\frac{p-2+\nu}{p-1+\nu}} }.
    \end{align}
\end{proof}

\subsection{Proof of Lemma \ref{le:alunk_alpha}}
\begin{proof}[\pof{\Cref{le:alunk_alpha}}]
    Consider any $c_2, \delta>0$. Inequality $\alunk \geq c_2^{\frac 1{1+\delta}} $ implies
    \begin{align*}
        \frac 1 {\alunk^\delta} c_2
        \geq c_2 \prk ^{\delta},
    \end{align*}
    which is ensured by 
    \begin{align*}
        \alunk \geq \frac 1 {\alunk^\delta} c_2,
    \end{align*}
    or equivalently 
    \begin{align*}
        \alunk \geq c_2^{\frac 1{1+\delta}}.
    \end{align*}
    Now, choice
    $c_2=c_3\normM{\g(x^k)}{x^k}^{*\delta}$ guarantees that 
        $\alunk \geq c_3^{\frac 1{1+\delta}} \normM{\g(x^k)}{x^k} ^{*\frac {\delta}{1+\delta}}$
    ensures
        $\alunk \geq c_3 \lr \prk \normMd{\g(x^k)}{x^k} \rr ^{\delta}.$
\end{proof}

\subsection{Proof of Corollary \ref{col:holder_three}}
\begin{proof}[\pof{\Cref{col:holder_three}}]

    For the first part of \eqref{eq:alunkminthree}, we use $\prk, \nu \in [0,1]$ to bound $\frac 1 {\alunk^{\frac 1 {1+\nu}}}\geq \prk^{\frac 1{1+\nu}} \geq \prk$ and
    \begin{align*}
        \frac 1 {\alunk^{\frac 1 {1+\nu}}} 6 \lr \frac  {\Lsupp 3 } {1+\nu} \rr ^{\frac 1 {1+\nu}} \normMd {\g(x^k)} {x^k}
        \geq 6 \lr \frac  {\Lsupp 3 } {1+\nu} \rr ^{\frac 1 {1+\nu}} \prk\normMd {\g(x^k)} {x^k}.
    \end{align*}
    Now, the first part of \eqref{eq:alunkminthree} is ensured by $\alunk$ so that 
    \begin{align*}
        \alunk 
        \geq \frac 1 {\alunk^{\frac 1 {1+\nu}}} 6 \lr \frac  {\Lsupp 3 } {1+\nu} \rr ^{\frac 1 {1+\nu}} \normMd {\g(x^k)} {x^k},
    \end{align*}
    or equivalently
    \begin{align*}
        \alunk 
        \geq \lr \frac  {6^{1+\nu} \Lsupp 3 } {1+\nu} \rr ^{\frac 1 {2+\nu}} \normM {\g(x^k)} {x^k}^{*\frac{1+\nu}{2+\nu}}.
    \end{align*}
    
    We ensure the second part of \eqref{eq:alunkminthree} directly using \Cref{le:alunk_alpha} and together with first part we have
    \begin{align*}
        \alunk 
        & \geq \max \lc\lr \frac  {6^{1+\nu} \Lsupp 3 } {1+\nu} \rr ^{\frac 1 {2+\nu}} \normM {\g(x^k)} {x^k}^{*\frac{1+\nu}{2+\nu}},
        \lr \frac {\sqrt {3}\Lsupp 3 }{(1+\nu) (2+ \nu)} \rr^{\frac 1 {2+\nu}} \normM {\g(x^k)}{x^k}^{*\frac {1+\nu}{2+\nu}}
        \rc\\
        & = \lr \frac  {\Lsupp 3 } {1+\nu} \rr ^{\frac 1 {2+\nu}} \normM {\g(x^k)}{x^k}^{*\frac {1+\nu}{2+\nu}} \max \lc 6^{\frac{1+\nu}{2+\nu}},
        \lr \frac {\sqrt 3}{2+ \nu} \rr^{\frac 1 {2+\nu}}
        \rc\\
        &= \lr \frac  {6^{1+\nu} \Lsupp 3 } {1+\nu} \rr ^{\frac 1 {2+\nu}} \normM {\g(x^k)}{x^k}^{*\frac {1+\nu}{2+\nu}}.
    \end{align*}
\end{proof}

\subsection{Proof of Lemma \ref{le:one_step_to_rate}}
\begin{proof} [\pof{\Cref{le:one_step_to_rate}}]
For $0 \leq \beta \leq 1$, function $y(x)=x^\beta, x \geq 0$ is concave, which implies
\begin{equation}\label{eq:concave}
a^\beta-b^\beta \geq \frac{\beta}{a^{1-\beta}}(a-b), \quad \forall a>b \geq 0,
\end{equation}
which we will be using for $\gb \eqdef \frac 1 {q-1} =\lr 0, 1\rs.$ We rewrite functional value decrease as
\begin{align}
    \frac 1 {f_{k+1}^\gb} - \frac 1 {f_{k}^{\gb}}
    & = \frac {f_k^{\gb}-f_{k+1}^{\gb}}{f_{k}^{\gb} f_{k+1}^{\gb}}
    \stackrel{\eqref{eq:concave}}\geq \frac{ \gb (f_{k} - f_{k+1})}{f_{k} f_{k+1}^{\gb}}        
    \stackrel{\eqref{eq:one_step_gen}}{\geq} \gb \cc \frac{\normsMd {\g(x^{k+1})}{x^k}}{ \normM {\g(x^k)}{x^k}^{*\frac{q-2}{q-1}} } \frac 1 { f_k f_{k+1}^{\frac 1{q-1}}}\\
    &\geq \gb \cc \frac{\normsMd {\g(x^{k+1})}{x^k}}{ \normM {\g(x^k)}{x^k}^{*\lr2- \frac{q}{q-1} \rr} } \frac 1 { f_k^{\frac q{q-1}}}
    \geq \frac {\gb \cc} {D^{1+\gb}} \frac{\normsMd {\g(x^{k+1})}{x^k}}{ \normsMd {\g(x^k)}{x^k}},
\end{align}
where in the last step we used the convexity of $f$ in the form $f_k \leq D \normMd {\g(x^k)} {x^k}$. We can continue by summing it for $k=0, \dots, n-1$, 
    
\begin{align}
    \frac 1 {f_{n}^\gb} - \frac 1 {f_{0}^{\gb}}
    & \geq\frac {\gb \cc} {D^{1+\gb}} \sum _{k=0} ^{n-1} \frac{\normsMd {\g(x^{k+1})}{x^k}}{\normsMd {\g(x^k)}{x^k}} \\
    & \stackrel {AG} \geq  \frac {\gb \cc n} {D^{1+\gb}} \lr \prod _{k=0} ^{n-1} \frac{\normsMd {\g(x^{k+1})}{x^k}}{ \normsMd {\g(x^k)}{x^k}} \rr^{\frac 1 n} \\
    & = \frac {\gb \cc n} {D^{1+\gb}} \lr \prodin k{n-1} \frac{\normsMd {\g(x^{k})}{x^{k-1}}}{ \normsMd {\g(x^k)}{x^{k}}} \rr^{\frac 1 n} \lr \frac{\normMd {\g(x^{n})}{x^{n-1}}}{ \normMd {\g(x^0)}{x^0}} \rr^{\frac 2 n} \\
    & \geq \frac {\gb \cc n} {\cbound D^{1+\gb}} \lr \frac {f_n}{\normMd {\g(x^0)} {x^{0}} D} \rr^{\frac 2 n} \label{eq:one_step_universal}\\
    & = \frac {\gb \cc n} {\cbound D^{1+\gb}} \exp \lr - \frac 2 n \ln \lr \frac {\normMd {\g(x^0)} {x^{0}} D} {f_n} \rr \rr\\
    & \geq \frac {\gb \cc n} {\cbound D^{1+\gb}} \lr 1 - \frac 2 n \ln \lr \frac {\normMd {\g(x^0)} {x^{0}} D} {f_n} \rr \rr
\end{align}
    We can bound $f_n$ based on the size of $\frac 2 n \frac {\normMd {\g(x^0)} {x^{0}} D} {f_n}$.
    \begin{enumerate}
        \item If $\frac 2 n \ln \lr \frac {\normMd {\g(x^0)} {x^{0}} D} {f_n} \rr \geq \frac 12$, then $f_n \leq \normMd {\g(x^0)} {x^{0}} D \exp\lr -\frac k4\rr$.
        \item If $\frac 2 n \ln \lr \frac {\normMd {\g(x^0)} {x^{0}} D} {f_n} \rr < \frac 12$, then      
                \begin{gather}
                    \frac 1 {f_{n}^{\gb}}
                    >  \frac 1 {f_{n}^{\gb}} - \frac 1 {f_0^{\gb}}
                    \geq \frac {\gb \cc n} {2\cbound D^{1+\gb}}
                    \Leftrightarrow f_n < \lr \frac {2\cbound D^{1+\gb}} {\gb \cc n} \rr ^{\frac 1 \gb}
                    = \frac {D^q \lr 2\cbound (q-1) \rr ^{q-1}}{\cc^{q-1} n^{q-1}} 
                \end{gather}                
    \end{enumerate}
    Hence 
    \begin{equation}
        f_n \leq\frac {D^q \lr 2\cbound (q-1) \rr ^{q-1}}{\cc^{q-1} n^{q-1}} + \normMd {\g(x^0)} {x^{0}} D \exp\lr -\frac k4\rr.
    \end{equation}
\end{proof}

\subsection{Proof of Theorem \ref{th:superlinear}}
\begin{proof}[\pof{\Cref{th:superlinear}}]
Bounded Hessian change together with condition \eqref{eq:ip_un} in \Cref{th:holder_un} imply inequalities
\begin{align}
    &\normMd {\g(x^{k+1})} {x^k} \normMd{\g(x^{k})}{x^k}
    \geq \la \g(x^{k+1}), \ls \h(x^k) \rs ^{-1}\g(x^{k}) \ra 
    \geq \frac 1{2\prk\alunk} \normsMd { \g(x^{k+1})}{x^k}, \nonumber\\
    &\normMd{\g(x^{k})}{x^k}
    \geq \frac 1{2\prk\alunk} \normMd { \g(x^{k+1})}{x^k}
    \geq \frac \cbound{2\prk\alunk} \normMd { \g(x^{k+1})}{x^{k+1}}
    \qquad \lr \geq \frac \cbound 2 \normMd{\g(x^{k+1})}{x^{k+1}} \rr, \label{eq:consbound}
\end{align}
which for $\alunk$ from \eqref{eq:alunkminun} guarantees local superlinear rate for $q>2$.
\end{proof}

\subsection{Proof of Theorem \ref{th:global}}
\begin{proof} [\pof{\Cref{th:global}}]
\Cref{th:holder_un} implies that \Cref{alg:rn} satisfies requirements of \Cref{le:one_step_to_rate} with correspondent $q$ and $\cc=\frac 12 \lr \frac {1} {9\Lq} \rr^{\frac 1 {q-1}}.$ The convergence rate follows.
\end{proof}

\subsection{Proof of Lemma \ref{le:universal}}
\begin{proof}[\pof{\Cref{le:universal}}]
    We will prove the statement by induction. The base for $\regc_0$ holds. For $k$-th iteration, consider $2$ cases based on the number of iterations of the inner loop.
    \begin{enumerate}
        \item Algorithm continues after $j_k>0$ inner iterations. Note that if $\alunkest {j_k-1}$ satisfied \eqref{eq:alunkminun}, \Cref{th:holder_un} guarantees the continuation condition to be satisfied for $j_k-1$. Consequently, $\alunkest {j_k-1}$ does not satisfy \eqref{eq:alunkminun} for any $q\in[2,4]$, and hence
        \begin{align}
        \regc_{k+1}
        = \frac {\alunkest {j_k-1}}{ \normM {\g(x^{k})} {x^{k}}^{*\expo}}
        < \inf_{q\in[2,4]} \lr 9\Lq \rr^{\frac 1 {q-1}}  \normM {\g(x^k)}{x^{k}} ^{*\frac {q-2}{q-1}-\expo}
        = \univ{x^k}.
        \end{align}
        \item Algorithm continues after $j=0$ iterates, then from \eqref{eq:consbound} we have
        \begin{align}
            \regc_{k+1} 
            = \frac {\regck} \cbound
            \leq \frac 1 \cbound \univ{x^{k-1}}
            \leq \cbound^{\frac {q-2}{q-1}-1} \univ{x^k}
            \leq \univ{x^k}.
        \end{align}
    \end{enumerate}
    For the total number of oracle calls $N_K,$
    \begin{align}
        N_K 
        &= \sum _{k=0} ^{K-1} (1+j_k)
        = K + \sum _{k=0} ^{K-1} \log_c \frac{c\regc_{k+1}}{\regck}
        = 2K+ \log_c \frac{\regc_{K}}{\regc_0}\\
        &\leq 2K+ \log_c \frac{\univ{\normMd{x^{k-1}}{x^{k-1}}}}{\regc_0}.
    \end{align}
\end{proof}

\subsection{Proof of Theorem \ref{th:global_universal}}
\begin{proof}[\pof{\Cref{th:global_universal}}]
    \Cref{alg:un} sets $x^{k+1}=x^k_{j_k}$ so that
    \begin{align}
        \la \g(x^k_{j_{k-1}}), n^k \ra &< \frac 1 {2\pr_{k,j_{k-1}} \alunkest {j_{k-1}}} \normsMd{\g(x^k_{j_{k-1}})} {x^k}, \\
        \la \g(x^k_{j_k}), n^k \ra &\geq \frac 1 {2\pr_{k,j_k} \alunkest {j_k}} \normsMd{\g(x^k_{j_k})} {x^k}.
    \end{align}
    From \Cref{th:holder_un} we can see that while $\alunkest {j_{k-1}}=\alunkest {j_k}/\cbound$ does not satisfy \eqref{eq:ip_un} for any $q\in[2,4]$ and $\alunkest {j_k}$ satisfies \eqref{eq:alunkminun} for some $q$, therefore
    \begin{align}
        \alunkest {j_k} &\geq \lr 9\Lq \rr^{\frac 1 {q-1}}  \normM {\g(x^k)}{x^k} ^{*\frac {q-2}{q-1}} \qquad \exists q\in[2,4]\\
        \alunkest {j_k} &< \cbound \lr 9\Lq \rr^{\frac 1 {q-1}}  \normM {\g(x^k)}{x^k} ^{*\frac {q-2}{q-1}} \qquad \forall q \in[2,4]\\ 
        \alunkest {j_k} &< \cbound \inf_{q \in[2,4]} \lr 9\Lq \rr^{\frac 1 {q-1}}  \normM {\g(x^k)}{x^k} ^{*\frac {q-2}{q-1}},
    \end{align}
    hence estimate $\alunkest {j_k}$ is at most constant $\cbound $ times worse than any plausible parametrization of $(q,\Lq)$, and therefore, even the best plausible parametrization. In particular, for 
    \begin{align}
        q^*\eqdef \argmin_{q\in [2,4]} \frac {9\Lq D^q \lr 4 \cbound^2 (q-1) \rr ^{q-1}}{k^{q-1}} + \normMd {\g(x^0)} {x^{0}} D \exp\lr -\frac k4\rr,
    \end{align}
    we have that from \Cref{th:holder_un}
    \begin{align}
        f(x^k)-f(x^{k+1})
        \geq \frac 1{2\cbound} \lr \frac {1} {9M_{q^*}} \rr^{\frac 1 {q^*-1}} \frac{\normsMd {\g(x^{k+1})}{x^k}}{ \normM {\g(x^k)}{x^k}^{*\frac{q^*-2}{q^*-1}} }.
    \end{align}
    The rest of the proof is analogous to the proof of \Cref{th:global}.
\end{proof}

\ifneurips
    \include{styles/neurips_2024_checklist}
\fi

\end{document}